\documentclass[12pt,letter]{amsart}
\usepackage[
    top=2.9cm, bottom=2.9cm, inner=3.1cm, outer=3.1cm,
    marginparwidth=2cm 
    ]{geometry}

\usepackage{amssymb,latexsym, amsmath, amsxtra, mathrsfs, bm}

\usepackage[dvips]{graphics}
\usepackage{xypic,xcolor}
\usepackage{subcaption}

\usepackage{verbatim}
\usepackage[abs]{overpic}
\usepackage{hyperref}
\usepackage{mathtools}

\DeclareMathAlphabet{\mathbbold}{U}{bbold}{m}{n}

\allowdisplaybreaks[3]

\theoremstyle{plain}
        \newtheorem{theorem}{Theorem}[section]
        \newtheorem*{theorem*}{Theorem}
        \newtheorem*{conj*}{Conjecture}
        \newtheorem{lemma}[theorem]{Lemma}
        \newtheorem{prop}[theorem]{Proposition}

\theoremstyle{definition}
        \newtheorem{definition}[theorem]{Definition}
        \newtheorem{rem}[theorem]{Remark}
         \newtheorem{rems}[theorem]{Remarks}
         
         \newtheorem*{assumptions}{The Assumptions}

\theoremstyle{remark}
        \newtheorem*{remark}{Remark}

\numberwithin{equation}{section}
\numberwithin{theorem}{section}
\numberwithin{table}{section}
\numberwithin{figure}{section}

\providecommand{\defn}[1]{\emph{#1}}

\renewcommand{\leq}{\leqslant}

\renewcommand{\geq}{\geqslant}

\newcommand{\diam}  {\operatorname{diam}}

\newcommand{\inte}  {\operatorname{inte}}
\newcommand{\inter}  {\operatorname{int}}

\newcommand{\card} {\operatorname{card}}
\newcommand{\supp}{\operatorname{supp}}
\newcommand{\R}{\mathbb{R}}
\newcommand{\B}{\mathbb{B}}    
\newcommand{\C}{\mathbb{C}}      
\newcommand{\Pillow}{\mathbb{P}}

\newcommand{\N}{\mathbb{N}}      
\newcommand{\Z}{\mathbb{Z}}      

\newcommand{\D}{\mathbb{D}}

\providecommand{\abs}[1]{\lvert#1\rvert}
\newcommand{\Abs}[1]{\left\lvert#1\right\rvert}
\providecommand{\Absbig}[1]{\bigl\lvert#1\bigr\rvert}
\providecommand{\Absbigg}[1]{\biggl\lvert#1\biggr\rvert}

\providecommand{\norm}[1]{\|#1\|}
\providecommand{\Norm}[1]{\left\|#1\right\|}

\renewcommand{\:}{\colon}

\renewcommand{\b}{\mathfrak{b}}
\newcommand{\w}{\mathfrak{w}}

\newcommand{\I}{\mathbf{i}}

\newcommand{\crit}{\operatorname{crit}}
\newcommand{\post}{\operatorname{post}}

\newcommand{\LIP}{\operatorname{LIP}}

\newcommand{\lip}{\operatorname{lip}}

\newcommand{\Lip}{\operatorname{Lip}}

\newcommand{\CC}{\mathcal{C}}

 \newcommand{\DD}{\mathbf{D}}

\newcommand{\X} {\mathbf{X}}

\newcommand{\E} {\mathbf{E}}

\newcommand{\V} {\mathbf{V}}
 
\newcommand{\CCC}{C}

\newcommand{\PPP}{\mathcal{P}}

\newcommand{\MMM}{\mathcal{M}}

\newcommand{\Holder}[1] {\CCC^{0,#1}}

\newcommand{\holder}[1] {c^{0,#1}}

\newcommand{\Hseminorm}[2] {\Abs{#2}_{#1}}

\newcommand{\Hseminormbig}[2] {\Absbig{#2}_{#1}}

\newcommand{\Hnorm}[3] {\Norm{#2}_{{#1},\,{#3}}}

\newcommand{\RR}{\mathcal{L}}

\newcommand{\lcm}{\operatorname{lcm}}

\newcommand{\XX}{\mathbb{X}}

\newcommand{\wt}[1]{\widetilde{#1}}

\newcommand{\circsmall}{\scalebox{0.6}[0.6]{$\circ$}}

\newcommand{\ti}{\vartriangle}

\renewcommand{\=}{\coloneqq}

\renewcommand{\O}{\mathcal{O}}

\newcommand{\Mmax}{\MMM_{\operatorname{max}}}

\newcommand{\wh}{\widehat}

\newcommand{\cl}{\operatorname{cl}}

\newcommand{\Lock}{\operatorname{Lock}}

\newcommand{\lock}{\operatorname{lock}}

\newcommand{\cA}{\mathcal{A}}

\newcommand{\cF}{\mathcal{F}}
\newcommand{\cG}{\mathcal{G}}

\newcommand{\cK}{\mathcal{K}}

\newcommand{\sF}{\mathscr{F}}

\newcommand{\sP}{\mathscr{P}}

\newcommand{\oW}{\overline{W}}

\begin{document}
\title[Ground states and periodic orbits]{Ground states and periodic orbits for\\ expanding Thurston maps}
\author{Zhiqiang~Li \and Yiwei~Zhang}
\address{Zhiqiang~Li, School of Mathematical Sciences \& Beijing International Center for Mathematical Research, Peking University, Beijing 100871, China.}
\email{zli@math.pku.edu.cn}
\address{Yiwei~Zhang, Department of Mathematics, Southern University of Science and Technology, Shenzhen, Guangdong 518055, China.}
\email{zhangyw@sustech.edu.cn}
\dedicatory{Dedicated to Professor Mario Bonk on the occasion of his sixtieth birthday}




\begin{abstract}
Expanding Thurston maps form a class of branched covering maps on the topological $2$-sphere $S^{2}$, which are topological models of some non-uniformly expanding rational maps without any smoothness or holomorphicity assumption initially investigated by W.~P.~Thurston, M.~Bonk, D.~Meyer, P.~Ha\"issinsky, and K.~M.~Pilgrim. The measures of maximal entropy and the absolutely continuous invariant measures for these maps have been studied by these authors, and equilibrium states by the first-named author. In this paper, we initiate the investigation on two new classes of invariant measures, namely, the maximizing measures and ground states, and establish the Liv\v{s}ic theorem, a local Anosov closing lemma, and give a positive answer to the Typically Periodic Optimization Conjecture from ergodic optimization for these maps. As an application, we establish these results for Misiurewicz--Thurston rational maps (i.e., postcritically-finite rational maps without periodic critical points) on the Riemann sphere including the Latt\`es maps with respect to the spherical metric. Our strategy relies on the visual metrics developed by the above authors.

In particular, we verify, in a first non-uniformly expanding setting, the Typically Periodic Optimization Conjecture, establishing that for a generic H\"{o}lder continuous potential, there exists a unique maximizing measure, moreover, this measure is supported on a periodic orbit,  it satisfies the locking property, and it is the unique ground state. The expanding Thurston maps we consider include those that are not topologically conjugate to rational maps; in particular, they can have periodic critical points.
\end{abstract}

\maketitle

\tableofcontents

\section{Introduction}     \label{sct_Introduction}
Uniformization and rigidity problems play important roles in the studies in geometry, group theory, dynamics, analysis, and the intersections of these fields, especially in relation to classical complex analysis (see for example, \cite{Kl06, Bon06, Sp04}). Inspired by the quest for a proof of Cannon's Conjecture \cite{Ca94} in geometric group theory, which can be seen as a quasisymmetric uniformization problem on the topological $2$-sphere $S^2$ closely related to Thurston's Hyperbolization Conjecture (\cite{Kl06, Bon06}), M.~Bonk and D.~Meyer \cite{BM10, BM17} initiated investigations on a class of branched covering maps on $S^2$ on the opposite side of Cannon's Conjecture in Sullivan's dictionary. See also related works of P.~Ha\"issinsky and K.~M.~Pilgrim \cite{HP09}.

In the early 1980s, D.~P.~Sullivan \cite{Su85, Su83} introduced a ``dictionary,'' known as \emph{Sullivan's dictionary} nowadays, linking two aspects of conformal dynamics, namely, geometric group theory and complex dynamics. The former mainly concerns the study of Kleinian groups acting on the Riemann sphere, and the latter focuses mainly on rational maps. Many dynamical objects in both areas can be similarly defined, and results can be similarly proven, yet essential and important differences remain.

In Sullivan's dictionary, Kleinian groups, i.e., discrete subgroups of M\"obius transformations on the Riemann sphere, correspond to rational maps, and convex-cocompact Kleinian groups correspond to rational maps that exhibit certain expansion properties such as hyperbolic rational maps, semi-hyperbolic rational maps, and postcritically-finite sub-hyperbolic rational maps, depending on the context of investigations. See insightful discussions on this part of the dictionary in \cite[Chapter~1]{BM17}, \cite[Chapter~1]{HP09}, and \cite[Section~1]{LM97}.

The class of branched covering maps M.~Bonk and D.~Meyer proposed to study in \cite{BM10, BM17}, called \emph{expanding Thurston maps}, are those whose finitely-many critical points are all preperiodic and who expand in a subtle way despite the presence of critical points. These maps are topological models of postcritically-finite rational maps on the Riemann sphere with empty Fatou sets. The expanding Thurston maps belong to the bigger class of branched covering maps on $S^2$ investigated by W.~P.~Thurston in his celebrated combinatorial characterization theorem, sometimes known as the fundamental theorem of complex dynamics, in which he characterized postcritically-finite rational maps among a class of more general topological maps, known as Thurston maps nowadays \cite{DH93}. A \emph{Thurston map} is a (non-homeomorphic) branched covering map on the topological $2$-sphere $S^2$ whose finitely-many critical points are all preperiodic. Thurston's theorem asserts that a Thurston map is essentially a postcritically-finite rational map if and only if there exists no so-called \emph{Thurston obstruction}, i.e., a collection of simple closed curves on $S^2$ subject to certain conditions. For generalizations of Thurston's theorem, see for example, the works of Guizhen~Cui, Yunping~Jiang, D.~Sullivan, Lei~Tan, and Gaofei~Zhang \cite{CJS04, JZ09, CT11}.

Inspired by Thurston's theorem, it is desirable to investigate the most essential dynamical and geometric properties of postcritically-finite rational maps in the setting of Thurston maps instead, with the conformality or any smoothness assumptions removed.

Under Sullivan's dictionary, the counterpart to Thurston's theorem in the geometric group theory is Cannon's Conjecture \cite{Ca94}. An equivalent formulation of Cannon's Conjecture \cite[Conjecture~5.2]{Bon06} from a quasisymmetric uniformization point of view predicts that if the boundary at infinity $\partial_\infty G$ of a Gromov hyperbolic group $G$ is homeomorphic to $S^2$, then $\partial_\infty G$ equipped with a visual metric is quasisymmetrically equivalent to $S^2$ equipped with the spherical metric. The notion of quasisymmetry recalled in Definition~\ref{d_Quasi_Symmetry} is fundamental in coarse geometry. Gromov hyperbolic groups can be considered as metric-topological systems generalizing the conformal systems in the context, namely, convex-cocompact Kleinian groups. Inspired by Sullivan's dictionary and their interest in Cannon's Conjecture, M.~Bonk and D.~Meyer \cite{BM10, BM17}, P.~Ha\"issinsky and K.~M.~Pilgrim \cite{HP09}, along with others, studied a subclass of Thurston maps by imposing some additional condition of expansion. Roughly speaking, we say that a Thurston map is \emph{expanding} if for any two points $x, \, y\in S^2$, their preimages under iterations of the map get closer and closer. It is important to keep in mind that this condition is much weaker than the usual distance-expanding condition and the expansive condition due to the presence of critical points. For a related class of \emph{weakly coarse expanding} dynamical systems, see \cite{DPTUZ21}. For a closer investigation on the weak expansion properties of these maps and related dynamical systems, see \cite{Li15, LZ23}.

Expanding Thurston maps give natural topological models for investigating the dynamical and geometric properties of (non-uniformly expanding) postcritically-finite rational maps.  On the other hand, an expanding Thurston map may be \emph{obstructed}, i.e., not topologically conjugate to a rational map; in particular, it can have periodic critical points. See Appendix~\ref{apx_illustrations} for examples of expanding Thurston maps.

More examples of expanding Thurston maps can be found in \cite{BM17}, in particular, \cite[Section~1.9]{BM17}. Since a postcricitally-finite rational map (i.e., rational Thurston map) is expanding if and only if it has no periodic critical points (see \cite[Proposition~2.3]{BM17}), it is straightforward to check whether a postcricitally-finite rational map is an expanding Thurston map. For examples, $1- 2/z^2$, $1- 2/z^4$, $\frac{\I}{2} (z + 1/z)$, $\I\bigl( z^4 - \I \bigr) \big/ \bigl( z^4 + \I \bigr)$, and $ 4 z \bigl( 1-z^2 \bigr) \big/ \bigl( 1+z^2 \bigr)^2$ are all expanding Thurston maps. Examples of expanding Thurston maps not only can be given by formulas, but also can be obtained from combinatorial data. More specifically, a Thurston map can be constructed from a certain \emph{two-tile subdivision rule} on $S^2$, and moreover, such a two-tile subdivision rule can be realized by an expanding Thurston map if and only if it is \emph{combinatorially expanding}. See Appendix~\ref{apx_illustrations} for such examples, and see \cite[Chapters~12 and 14]{BM17} for more details. Note that the combinatorial construction can also produce obstructed expanding Thurston maps.

The geometric and dynamical properties of expanding Thurston maps have been explored extensively by M.~Bonk and D.~Meyer and summarized in their monograph \cite{BM10, BM17} (see also \cite{HP09}). See also the subsequent related works \cite{Yi11, HP12, Me13, HP14, Li17, BD18, HM18, Wu19, BM20, DPTUZ21, Liw22, LZ24a, LZ24b, LZ24c, LS24}, and this list is far from exhaustive.

Postcritically-finite rational maps, in particular, are at the center of various directions of active investigations extending well beyond the scope of complex dynamics. For example, for their connections to self-similar groups, see for example, \cite{Ne05, BN06, HM18}; for their connections to geometric group theory, see for example, \cite{Su85, Bon06, BM17, HP09, HP14, Me13}; for their connections to arithmetic dynamics\footnote{See \cite[Section~4.7]{HP09} for an account on some formal similarities between the $p$-adic dynamics setting and our visual metrics setting.}, see for example, \cite{BIJMST19, De18}.

In dynamics, invariant measures and periodic orbits are both indispensable tools and basic objects for investigations themselves. The locations of periodic points in terms of certain combinatorial structures were studied, and an exact formula for the number of periodic points for each period was found in \cite{Li16} for expanding Thurston maps. Equidistribution and large deviation results for periodic points were studied in \cite{Li18, Li15}. An asymptotic formula for the number of periodic orbits with a certain weight induced by a H\"older continuous potential similar to the prime number theorem in number theory was established in \cite{LZ24a, LZ24b, LZ24c} (see also \cite{LZ18}). Regarding important classes of invariant measures of these maps, the measures of maximal entropy were studied in \cite{BM10, BM17, HP09, Li16}, the absolutely continuous invariant measures in \cite{BM17}, and the equilibrium measures by the first-named author \cite{Li18} and later by Das et al.~\cite{DPTUZ21} in broader settings with a different approach building upon prior works of P.~Ha\"issinsky and K.~M.~Pilgrim \cite{HP09}.  The first-named author has been informed that similar results on equilibrium states have also been obtained independently by P.~Ha\"issinsky.

In this paper, we investigate some other basic properties of periodic points of expanding Thurston maps and initiate the studies of another two classes of invariant measures, namely, (potential-energy-)maximizing measures and ground states. More precisely, we establish in our context the Liv\v{s}ic Theorem and various closing lemmas, including a Bressaud--Quas closing lemma and a local version of Anosov closing lemma. The Liv\v{s}ic Theorem, dating back to the work of Liv\v{s}ic \cite{Liv72}, has played important roles in the study of rigidity problems in dynamics (see for example, \cite{Sp04}). It basically states that a real-valued function (called a potential) is determined uniquely (up to a coboundary of the same regularity) by the sum of its values along periodic orbits (see Theorem~\ref{t_livsic}). The usual approach to Liv\v{s}ic Theorems only yielded a partial result \cite[Proposition~8.8]{Li18} for expanding Thurston maps. Our current approach is to establish a stronger result, called the bilateral Ma\~{n}\'{e} lemma, using tools from ergodic optimization. Due to the presence of critical points, the lack of Markov partitions, and the non-uniform expansion nature in our setting, the full version of the classical Anosov closing lemma is beyond reach (c.f.\ \cite{CKY88} in a one-dimensional real dynamics setting). A version of the (global) Anosov closing lemma was established by the first-named author in \cite[Lemma~8.6]{Li18}, albeit coarse in the temporal direction, i.e., it holds only for sufficiently high iterates of the map and sufficiently long orbits. In this paper, we establish a local version (away from critical points) of the Anosov closing lemma (Lemma~\ref{l_Local_Closing_Lemma}) without the assumptions on the iterate of the map or lengths of the orbits. On the other hand, we formulate fairly general rules to deduce the Bressaud--Quas shadowing property between related systems, such as factors and iterations, before verifying this property for our maps, proving the Bressaud--Quas closing lemma. Combining these two kinds of closing lemmas, we are able to formulate and establish another local closing lemma in Lemma~\ref{l_bound_by_gap}, which is crucial in our investigations on the new classes of invariant measures, i.e., the maximizing measures and ground states. We prove the existence, uniqueness, and periodicity of the maximizing measures for generic H\"older continuous potentials for expanding Thurston maps, establishing the Typically Periodic Optimization (TPO) Conjecture (\cite[Conjecture~1.1]{YH99}) in our setting. We say that an invariant measure is \emph{periodic} if it is supported on a periodic orbit.  Outside of complex dynamics, the TPO Conjecture has previously been fully verified mostly in uniformly expanding and uniformly hyperbolic systems (see \cite{Co16} for distance expanding maps and \cite{HLMXZ19} for  Axiom A attractors and Anosov diffeomorphisms). The TPO Conjecture has connections to various fields such as the Finiteness Conjecture in control theory \cite{Boc18} and the random minimum (or maximum) mean cycle problems in probability theory \cite{BZ16, DLZ24}. Our theorem may be the first to verify the TPO Conjecture in a non-uniformly expanding setting for H\"older continuous potentials to the best of our knowledge. Finally, we prove that for an expanding Thurston map and a generic H\"older potential, there exists a unique ground state which the equilibrium state converges to at zero temperature.

To better understand our results for the invariant measures mentioned above, we quickly review some basic notions from the thermodynamic formalism in ergodic theory, dating back to the works of Ya.~G.\ Sinai, R.~Bowen, D.~Ruelle, and others around early the 1970s \cite{Do68, Sin72, Bow75, Ru78}, inspired by statistical mechanics. Recall the measure-theoretic pressure
\begin{equation}   \label{e_pressure}
P_\mu(T,\psi) = h_\mu(T) + \int\! \psi \,\mathrm{d}\mu,
\end{equation}
where $h_\mu(T)$ is the measure-theoretic entropy. In the language of statistical mechanics, the measure-theoretic entropy and the integral of the potential represent the kinetic and potential energy, respectively, while the measure-theoretic pressure represents the free energy, which is the sum of the kinetic energy and the potential energy. Under this interpretation, a measure of maximal entropy maximizes the kinetic energy (i.e., $h_\mu(T)$), and an equilibrium state maximizes the free energy (i.e., $P_\mu(T,f)$). It is thereby natural to consider measures of maximal potential energy.

We denote the maximal potential energy by 
\begin{equation}\label{e_ergodicmax}
  Q(T,\psi)\=\sup\limits_{\mu\in\MMM(X,T)}\int\! \psi \, \mathrm{d}\mu=\max\limits_{\mu\in\MMM(X,T)}\int\! \psi \, \mathrm{d}\mu.
\end{equation}
The last identity follows from the weak$^*$-compactness of the set $\MMM(X, T)$ of $T$-invariant Borel probability measures on $X$. 


We call a measure $\mu \in \MMM(X,T)$ that maximizes the potential energy $\int\!\psi \,\mathrm{d}\mu$ a \emph{measure of maximal potential energy} with respect to $\psi$ or a \emph{$\psi$-potential-energy-maximizing measure} (or a \emph{$\psi$-maximizing measure}). We denote the (nonempty) set of \emph{$\psi$-maximizing measures} by
\begin{equation}\label{e_setofmaximizngmeasures}
\Mmax(T, \, \psi) \=\biggl\{ \mu\in\MMM(X,T) :   \int \! \psi \,\mathrm{d} \mu = Q(T,\psi) \biggr\}.
\end{equation}
We say that a measure $\mu\in \MMM(X, T)$ is a \emph{ground state} for $T$ and $\psi$ if it is the limit of a sequence of equilibrium states $\mu_{t_i \psi}$ in the weak$^*$ topology for some sequence $\{t_i\}_{i\in\N}$ of real numbers that tends to infinity. Since $t=t_i$ can be considered as \emph{inverse temperature}, we say that equilibrium state $\mu_{t \psi}$ \emph{converges at zero temperature} if there exists a unique ground state. One can show that if $T$ admits a finite topological entropy, then any ground state is a $\psi$-maximizing measure (see the discussion preceding Theorem~\ref{t_gound_states}). The ground states may thereby be considered as the most physically relevant maximizing measures. 


The study of maximizing measures and ground states is the main theme for a field in ergodic theory known as \emph{ergodic optimization}. Despite its close connections to the theory of thermodynamic formalism, ergodic optimization originated in the 1990s from the works of B.~R.~Hunt and E.~Ott \cite{HO96a, HO96b}, with motivation from control theory \cite{OGY90, SGYO93}, and the Ph.D.\ thesis of O.~Jenkinson \cite{Je96}. Much of the early work in ergodic optimization focused on specific maps and potentials from finite-dimensional function spaces \cite{HO96a, HO96b, Je96, Je00, Je01, Bou00}. Guocheng~Yuan and B.~R.~Hunt conjectured \cite[Conjecture~1.1]{YH99} that for an axiom A or a uniformly expanding map, and for a (topologically) generic Lipschitz continuous or a (topologically) generic $\CCC^{1}$-smooth potential $\psi$, there is a $\psi$-maximizing measure and it is supported on a periodic orbit. It was recently settled by \cite{Co16, HLMXZ19}. In the more general setting of any suitably hyperbolic system and any space of suitably regular potentials, the Yuan--Hunt Conjecture is known as the Typically Periodic Optimization Conjecture.  In the context of Hamiltonian systems, a similar conjecture is known as Ma\~{n}\'{e}'s Conjecture \cite{Ma96}.

Complementary to the genericity in the positive statement of the TPO Conjecture, examples of non-con\-ver\-gence of equilibrium states at zero temperature have been constructed in symbolic settings  (see for example, \cite{BGT18, CRL15, CH10,vER07}). Moreover, D.~Coronel and J.~Rivera-Leterlier introduced another approach to the non-convergence of equilibrium states at zero temperature \cite{CRL15}. They constructed families of quadratic-like maps satisfying an assumption of non-uniform expansion called the Collet--Eckmann condition so that the ground states with respect to the geometrical potential are not unique in a robust sense, i.e., there is an open set of analytic $2$-parameter families of quadratic-like maps that exhibit so-called ``sensitive dependence'' of geometric Gibbs states \cite{CRL17}. More recently, D.~Coronel and J.~Rivera-Leterlier established analogous examples of the sensitivity dependence of geometric Gibbs states at positive temperature \cite{CRL19}.

Considering the recent breakthroughs on maximizing measures and ground states in uniformly expanding and uniformly hyperbolic settings mentioned above, it is natural to consider the properties of these measures for generic potentials in non-uniformly expanding settings.

One primary source of non-uniform expansion in dynamical systems arises from the existence of critical points. A basic setting of low-dimensional dynamical systems with critical points is that of rational maps on the Riemann sphere $\wh\C$. To limit the complexity, we assume that all critical points are preperiodic (i.e., have finite forward orbits) in this paper. These rational maps are called the \emph{postcritically-finite} rational maps and have been studied extensively in complex dynamics independent of the interests from ergodic theory. From the ergodic theory point of view, these maps serve as the most basic setting of non-uniformly expanding rational maps. The subclass of these maps with parabolic orbifolds, called \emph{Latt\`es maps}, are ubiquitous in the study of complex dynamics, appearing in many key theorems in the field. For recent works related to Latt\`es maps in the context of Thurston maps in particular, see \cite{LZ18, BM20, BHI21}.

Motivated by Sullivan's dictionary, for each expanding Thurston map, we can equip the topological $2$-sphere $S^2$ with a natural class of metrics, called \emph{visual metrics}, constructed in a similar fashion as the visual metrics on the boundary $\partial_\infty G$ at infinity of a Gromov hyperbolic group $G$. A characterization theorem of rational maps from a quasisymmetric uniformization point of view is established in this context by M.~Bonk and D.~Meyer \cite{BM10, BM17}, and by P.~Ha\"issinsky and K.~M.~Pilgrim \cite{HP09}:
An expanding Thurston map is conjugate to a rational map if and only if the sphere $(S^2,d)$ equipped with a visual metric $d$ is quasisymmetrically equivalent to the Riemann sphere $\widehat\C$ equipped with the spherical metric.



Our main approach to the study of the dynamics of our non-uniformly expanding systems differs from classical approaches in that it relies on the study of coarse geometric properties of the associated metric spaces. This should not be a surprise, considering the rich connections to geometric group theory and coarse geometry mentioned above. Such an approach has been pioneered by the first-named author to develop a reasonable theory of thermodynamic formalism for expanding Thurston maps in his thesis (see \cite{Li18, Li17}), and more recently, in his collaborated work with Tianyi~Zheng to establish Prime Orbit Theorems with exponential error terms \cite{LZ24a, LZ24b, LZ24c, LZ18}.

To get a glimpse of this approach, we note that in the setting of Latt\`es maps or more general postcritically-finite rational maps, with respect to the spherical metric on the Riemann sphere, the Ruelle--Perron--Frobenius (transfer) operator does not leave the space of $\alpha$-H\"older continuous functions invariant \cite[Remark~3.1]{DPU96}. Similarly, the quests for the $\alpha$-H\"older continuous fixed point of the Bousch operator associated to an $\alpha$-H\"older continuous potential, leading to the Ma\~{n}\'{e} lemma, have only produced partial results on fixed points with lower H\"older exponent in various contexts (see for example, \cite{Mo09, CLT01}). 
With the help of visual metrics, however, we get in our context well-defined Ruelle--Perron--Frobenius (transfer) operators, and similarly, Bousch operators, which admit H\"older continuous fixed points with the desired H\"older exponent. Moreover, even though the strongest form of the Anosov closing lemma is still beyond reach, we manage to establish weaker forms of closing lemmas with appropriate choices of metrics. Both of these observations turned out to be crucial in our current investigations.

\subsection{Main Results}
Let $(X,d)$ be a compact metric space with infinite cardinality. The space $\Lip(X,d)$ of real-valued Lipschitz functions consists of functions $\phi \in \CCC(X)$ with
\begin{equation*}
\sup\{ \abs{ \phi(x) - \phi(y) } : x,\,y \in X, \, d(x,y) \leq r \} = O(r) \qquad \text{ as } r \to 0.
\end{equation*}
Similarly, a function $\phi \in \CCC(X)$ is a \emph{little Lipschitz function} or \emph{locally flat Lipschitz function} if
\begin{equation*}
\sup\{ \abs{ \phi(x) - \phi(y) } : x,\,y \in X, \, d(x,y) \leq r \} = o(r) \qquad \text{ as } r \to 0.
\end{equation*}
The space of real-valued little Lipschitz functions on $(X,d)$ is denoted by $\lip(X,d)$ and is called the \emph{little Lipschitz space}. Clearly $\lip(X,d) \subseteq \Lip(X,d)$. We equip both spaces with the usual Lipschitz norm. The Lipschitz space $\Lip(X,d)$ is sometimes also called the \emph{big Lipschitz space}.

Little Lipschitz spaces, like big Lipschitz spaces, play central roles in the study of Lipschitz analysis and Lipschitz algebras in the analysis on metric spaces and functional analysis; see the beautiful monographs of N.~Weaver \cite{We18} and of \c{S}.~Cobza\c{s}, R.~Miculescu, and A.~Nicolae \cite{CMN19}, as well as references to the vast literature therein. Among other important properties, little Lipschitz spaces are double preduals of the big Lipschitz spaces under mild assumptions (see for example, \cite[Section~4.3]{We18}). To quote from \cite[Section~4.1]{We18}: ``The theory of little Lipschitz spaces is, generally speaking, parallel to but a bit more difficult than the theory of (big) Lipschitz spaces.''

Note that for each $\alpha \in (0,1]$, the snowflake $d^\alpha$ of $d$ given by $d^\alpha(x,y) \= d(x,y)^\alpha$ is also a metric. Since the space $\Holder{\alpha}(X,d)$ of real-valued H\"older continuous on $(X,d)$ with exponent $\alpha$ coincides with the space $\Lip(X,d^\alpha)$ of real-valued Lipschitz functions on $(X,d^\alpha)$ with the same norm, we sometimes say that a function $\phi \in \Lip(X,d^\alpha)$ is H\"older continuous with exponent $\alpha$ when there is little chance for confusion.  The space $\holder{\alpha}(X,d) \= \lip(X,d^\alpha)$ is called the little H\"older space with exponent $\alpha$ on $(X,d)$ by some authors.

The Liv\v{s}ic theorem plays important roles in rigidity problems in dynamics \cite{Sp04}. We state our first main theorem below.

\begin{theorem}[{\bf Liv\v{s}ic theorem for expanding Thurston maps}]   \label{t_livsic}
Let $f \: X \rightarrow X$ be either an expanding Thurston map on the topological $2$-sphere $X = S^2$ equipped with a visual metric or a postcricitally-finite rational map with no periodic critical points on the Riemann sphere $X = \wh\C$ equipped with the chordal or spherical metric. Let $\phi \: X \rightarrow \R$ be H\"older continuous. Then the following statements are equivalent:
\begin{enumerate}
\smallskip
\item[(i)] If $x\in X$ satisfies $f^n(x) = x$ for some $n\in \N$ then $\sum_{i=0}^{n-1} \phi ( f^i( x ) ) = 0$.

\smallskip
\item[(ii)] There exists a H\"older continuous function $u \: X \rightarrow \R$ such that $\phi = u -  u \circ f$.
\end{enumerate}
\end{theorem}

A special case of the Liv\v{s}ic Theorem for expanding Thurston maps was established in Proposition~8.8 in \cite{Li18} under the additional assumption of the existence of an invariant Jordan curve $\CC\subseteq S^2$ containing the postcritical points. As pointed out by the referee, Theorem~\ref{t_livsic} also follows from Theorem~1.2, the proof of Proposition~6.1, and Corollary~6.3  of \cite{DPTUZ21}.

Regarding the preservation of regularity in Theorem~\ref{t_livsic} and Theorem~\ref{t_mane} below, see Remark~\ref{r_livsic}. By working with the Bousch operators using techniques from the theory of the Ruelle operators, we establish in this paper the Ma\~{n}\'{e} lemma and the bilateral Ma\~{n}\'{e} lemma for expanding Thurston maps (without additional assumptions). The latter is a strengthening of the Liv\v{s}ic theorem for expanding Thurston maps.

\begin{theorem}[{\bf Ma\~{n}\'{e} and bilateral Ma\~{n}\'{e} lemmas for expanding Thurston maps}]   \label{t_mane}
Let $f \: X \rightarrow X$ be either an expanding Thurston map on the topological $2$-sphere $X = S^2$ equipped with a visual metric or a postcricitally-finite rational map with no periodic critical points on the Riemann sphere $X = \wh\C$ equipped with the chordal or spherical metric. Let $\phi \: X \rightarrow \R$ be H\"older continuous. Then the following statements hold:
\begin{enumerate}
\smallskip
\item[(i)] {\bf (The Ma\~{n}\'{e} lemma.)} There exists a H\"older continuous function $u \: X \rightarrow \R$ such that $\phi (x) - u(x) + ( u \circ f )(x) \leq Q (f, \phi)$ for all $x\in X$.

\smallskip
\item[(ii)] {\bf (The bilateral Ma\~{n}\'{e} lemma.)} There exists a H\"older continuous function $u \: X \rightarrow \R$ such that $- Q (f, - \phi ) \leq \phi (x) - u(x) + ( u \circ f )(x) \leq Q (f, \phi)$ for all $x\in X$.
\end{enumerate}
\end{theorem}

The bilateral Ma\~{n}\'{e} lemma was first introduced by T.~Bousch \cite{Bou02} for the doubling map on $S^1$.

For a continuous map $T\:X \rightarrow X$, we define subsets 
\begin{equation*}
\sP(X) \subseteq \CCC(X), \qquad
\Lock (X,d^\alpha) \subseteq \Lip(X,d^\alpha), \qquad  
\lock (X,d^\alpha) \subseteq \lip(X, d^\alpha)
\end{equation*}
for $\alpha \in (0,1]$, of the set $\CCC(X)$ of real-valued continuous functions as follows: $\sP(X)$ is the set of $\phi \in \CCC(X)$ with a $\phi$-maximizing measure supported on a periodic orbit of $T$. If a function $\phi \in \sP(X) \cap  \Lip(X,d^\alpha)$ (resp.\ $\phi \in \sP(X) \cap \lip(X,d^\alpha)$) satisfies $\card \Mmax(T, \, \phi) = 1$ and $\Mmax(T, \, \phi) = \Mmax(T, \, \psi)$ for all $\psi \in  \Lip(X,d^\alpha)$ (resp.\ $\psi \in \lip(X,d^\alpha)$) sufficient close to $\phi$ in $ \Lip(X,d^\alpha)$,  we say that $\phi$ has the \emph{locking property} in $ \Lip(X,d^\alpha)$ (resp.\ $\lip(X,d^\alpha)$). The set $\Lock(X,d^\alpha)$ (resp.\ $\lock(X,d^\alpha)$) consists of all $\phi \in \sP(X)$ with the locking property in $ \Lip(X,d^\alpha)$ (resp.\ $\lip(X,d^\alpha)$).

\smallskip

In this paper, by ``generic'' we mean ``open and dense''.

\begin{theorem}[{\bf Generic periodic maximization and locking for Misiurewicz--Thurston rational maps}]  \label{t_Density_Rational_little}
Let $f\: \wh\C \rightarrow \wh\C$ be a Misiurewicz--Thurston rational map (i.e., a postcritically-finite rational map without periodic critical points). Let $\sigma$ be the chordal metric on the Riemann sphere $\wh\C$. Then there exists a number $\gamma \in (0,1)$ such that for each $\beta\in (0,\gamma)$, the set $\lock  ( \wh\C, \sigma^{\beta}  )$ is an open and dense subset of $\lip ( \wh\C, \sigma^{\beta} )$, and in particular, the set $\sP (\wh\C )$ contains an open and dense subset of $\lip ( \wh\C, \sigma^{\beta} )$.
\end{theorem}

We recall that a rational Thurston map (i.e., a postcritically-finite rational map) is expanding if and only if it has no periodic critical points (\cite[Proposition~2.3]{BM17}).

For a more general expanding Thurston map $f\: S^2 \rightarrow S^2$ defined on the topological $2$-sphere $S^2$ (that may or may not be conjugate to a rational map), there is no canonical smooth structure on $S^2$, so it is natural (in view of the connection to Kleinian groups via Sullivan's dictionary) to formulate the corresponding theorem in terms of the visual metrics.

\begin{theorem}[{\bf Generic periodic maximization and locking for expanding Thurston maps}]  \label{t_Density_Thurston}
Let $f\:S^2 \rightarrow S^2$ be an expanding Thurston map (for example, a postcritically-finite rational map with no periodic points). Let $d$ be a visual metric on $S^2$ for $f$. Fix numbers $\alpha \in (0,1]$ and $\beta \in (0,1)$. Then the following statements hold:
\begin{enumerate}
\smallskip
\item[(i)] The set $\Lock(S^2, d^{\alpha})$ is an open and dense subset of  $\Lip(S^2, d^{\alpha})$. In particular, the set $\sP (S^2 )$ contains an open and dense subset of  $\Lip(S^2, d^{\alpha})$.

\smallskip
\item[(ii)] The set $\lock(S^2, d^{\beta})$ is an open and dense subset of $\lip(S^2, d^{\beta})$. In particular, the set $\sP (S^2 )$ contains an open and dense subset of $\lip(S^2, d^{\beta})$. 
\end{enumerate}
\end{theorem}

Note that an expanding Thurston map may have periodic critical points, in which case the investigations of certain dynamical properties can be more involved (compare \cite{Me13, Li15, Li18, DPTUZ21}). In our approach to establishing Theorem~\ref{t_Density_Thurston}, all expanding Thurston maps are treated simultaneously without extra care given to these special maps.

In the special case of Latt\`es maps, the following theorem holds.

\begin{theorem}[{\bf Generic periodic maximization and locking for Latt\`es maps}]  \label{t_Density_Lattes}
Let $f\: \wh\C \rightarrow \wh\C$ be a Latt\`es map. Let $\sigma$ be the chordal metric on the Riemann sphere $\wh\C$. Fix numbers $\alpha \in (0,1]$ and $\beta \in (0,1)$. Then the following statements hold:
\begin{enumerate}
\smallskip
\item[(i)] The set $\Lock(\wh\C,\sigma^{\alpha})$ is an open and dense subset of $\Lip(\wh\C, \sigma^{\alpha})$. In particular, the set $\sP (\wh\C )$ contains an open and dense subset of $\Lip(\wh\C, \sigma^{\alpha})$.

\smallskip
\item[(ii)] The set $\lock  ( \wh\C, \sigma^{\beta}  )$ is an open and dense subset of $\lip ( \wh\C, \sigma^{\beta} )$. In particular, the set $\sP (\wh\C )$ contains an open and dense subset of $\lip ( \wh\C, \sigma^{\beta} )$. 
\end{enumerate}
\end{theorem}

\begin{remark}
We can replace the chordal metric by the spherical metric in the statement and the theorem still holds due to the bi-Lipschitz equivalence between these two metrics. Note that $\beta$ cannot be $1$ since the only little Lipschitz functions in $\lip(\wh\C, \sigma)$ are the constant functions (c.f.~\cite[Example~4.9]{We18}).
\end{remark}

\begin{theorem}[{\bf Generic uniqueness of ground states}] \label{t_gound_states}
The following statements hold:

\begin{enumerate}
\smallskip
\item[(i)] Under the assumptions in Theorem~\ref{t_Density_Rational_little},~\ref{t_Density_Thurston}, or~\ref{t_Density_Lattes} with $\rho$ being $d$ or $\sigma$ accordingly, every $\phi\in\lock(X,\rho^\beta)$ has a unique ground state, and consequently, for a generic $\psi \in \lip(X,\rho^{\beta})$ there exists a unique ground state for $f$ and $\psi$, i.e., the equilibrium state $\mu_{t\psi}$ converges at zero temperature.

\smallskip
\item[(ii)] Under the assumptions in Theorem~\ref{t_Density_Thurston} or~\ref{t_Density_Lattes} with $\rho$ being $d$ or $\sigma$ accordingly, every $\phi\in\Lock(X,\rho^\alpha)$ has a unique ground state, and consequently, for a generic $\psi \in \Lip(X,\rho^{\alpha})$ there exists a unique ground state for $f$ and $\psi$, i.e., the equilibrium state $\mu_{t\psi}$ converges at zero temperature.
\end{enumerate}
\end{theorem}

The following theorem is a consequence of the results in \cite{Co16, BZ15, HLMXZ19} and our investigations on little Lipschitz functions.
\begin{theorem}[{\bf Generic periodic maximization and locking in little Lipschitz spaces}]  \label{t_Distance_expanding}
Let a map $T \: X \rightarrow X$ on a compact metric space $(X,d)$ be a distance expanding map, an Axiom A attractor, or an Anosov diffeomorphism on a compact Riemannian manifold. Let $\beta \in (0, 1)$. Then the set $\lock(X, d^{\beta})$ is an open and dense subset of $\lip(X, d^{\beta})$. In particular, the set $\sP (X )$ contains an open and dense subset of $\lip(X, d^{\beta})$. 
\end{theorem}

See \cite{Co16, HLMXZ19} for more discussions on the systems mentioned in Theorem~\ref{t_Distance_expanding}.

The result in Theorem~\ref{t_Distance_expanding} for the little Lipschitz spaces does not follow directly from the similar result for the space $\cF_{\alpha+}$ considered in \cite{CLT01}. In the case of distance expanding maps, a similar result is claimed in \cite[Section~1]{Co16} for the space $C^{!a}(X,\R)$ whose definition coincides with $\lip(X,d^\alpha)$ in our notation.

Finally, we remark that the periodic maximization phenomena in the main theorems above are only expected to hold in a generic sense. Clearly, every invariant Borel probability measure is a maximizing measure for constant potentials (and more generally, for potentials cohomologous to constants). For the doubling map on $[0,1]$, it follows from \cite[Theorem~A]{Bou00} that there exist real analytic functions not contained in $\sP ([0,1])$. Examples of such functions can be found within the one-parameter family $\phi_{\theta}(x) = \cos(2\pi(x - \theta))$. For certain values of $\theta$, the $\phi_{\theta}$-maximizing measure is a so-called Sturmian measure supported on a Cantor set. 

\subsection{Strategy and plan of the paper}
Even though results in the uniformly expanding and uniformly hyperbolic settings similar to those in this paper are either classical or have been recently established, new strategies are needed to settle them in our non-uniformly expanding context in complex dynamics.

Roughly speaking, by investigating the coarse geometric properties of various relevant metrics (spherical, chordal, visual, canonical orbifold, and singular conformal metrics) and their interplay with associated combinatorial structures (tiles, flowers, and bouquets), we convert and split the difficulties from non-uniform expansion into two categories: ones of combinatorial nature and ones related to metric geometry, and try to tackle them separately.

Due to the coarse geometric nature of our approach, various multiplicative and additive constants, both spatial and temporal, arise naturally in our analysis throughout the paper. This may make our proofs seem more technical than they really are.

We establish an appropriate form of the Ma\~{n}\'{e} lemma (to construct the so-called sub-actions with appropriate regularity) and formulate and prove fine quantitative (local) versions of closing lemmas. To investigate the maximizing measures and ground states, we adopt some of the ideas from \cite{Co16} and \cite{HLMXZ19} for distance expanding maps, Axiom A attractors, and Anosov diffeomorphisms. Since the strongest forms of the ingredients are not available in our non-uniformly expanding setting, the arguments in the proofs in Section~\ref{sct_Holder_Density} are delicate and quantitative in nature. One challenge we face is to formulate appropriate weaker versions of the classical results to fit them together in the proofs in Section~\ref{sct_Holder_Density}.

We discuss our approach in more detail below.

We need to establish a form of the closing lemma that produces, for a nonempty compact forward-invariant set $\cK$ (disjoint from critical points), a periodic orbit $\O$ close to $\cK$ in terms of its \emph{$(r,\theta)$-gap} $\Delta_{r,\,\theta} (\O)$ we introduce (see Lemma~\ref{l_bound_by_gap}). Such a closing lemma is ultimately built upon the Anosov closing lemma as well as a closing lemma due to Bressaud--Quas \cite{BQ07}. A version of the Anosov closing lemma for expanding Thurston maps was established by the first-named author in \cite[Lemma~8.6]{Li18}, albeit coarse in the temporal direction, i.e., it holds only for sufficiently high iterates of the map and sufficiently long orbits. Although it appears to be relevant, this version turns out to be insufficient for our needs in this paper. Instead, we establish a local version (away from critical points) of the Anosov closing lemma in our setting in Lemma~\ref{l_Local_Closing_Lemma}. Its proof relies on the combinatorial and metric properties of combinatorial objects like tiles, flowers, and bouquets, while successfully avoiding the more complicated combinatorics near critical points. We have to improve the qualitative \emph{uniform local injectivity property} of our maps from \cite{Li15} to a quantitative \emph{uniform local expansion property} in Section~\ref{sct_ULEAFC} for this purpose (as well as the final perturbation argument). On the other hand, if a Bressaud--Quas closing lemma holds for a dynamical system, we say such a system has the \emph{Bressaud--Quas shadowing property}. We establish fairly general rules to deduce the Bressaud--Quas shadowing property between related systems, such as factors and iterations, before establishing this property for expanding Thurston maps in Subsection~\ref{subsct_BQ}.

On the other hand, the existence and properties of sub-actions are important in our proof. A \emph{sub-action} is a function $h$ that satisfies the cohomological inequality $\phi + h \circ T - h \leq Q(T,\phi)$, where $T\: X \rightarrow X$ is a finite-to-one surjective continuous map on a compact metric space $X$ and $\phi$ is a potential. It was established by many authors in various settings that for some uniformly hyperbolic system $T$ and a sufficiently regular $\phi$, a sub-action exists \cite{CG93, Sav99, Bou00, Bou01, CLT01, Bou11, PoSh04}). Such results are called \emph{non-positive Liv\v{s}ic theorems, Ma\~{n}\'{e}--Conze--Guivarc'h lemmas} or \emph{Ma\~{n}\'{e} lemmas} for short. As a simple example, if $T$ is a subshift of finite type and $\phi$ is $\alpha$-H\"older continuous, then there exists an $\alpha$-H\"older continuous sub-action. Having such a result, including the desired regularity of $h$, is crucial in our setting. However, the regularity for sub-actions may not always be as good as that of $\phi$ even if $T$ is uniformly expanding (see \cite{BJ02}). Beyond uniformly expanding and uniformly hyperbolic systems, such an issue is even more serious. For example, for some intermittent map $T$ and some $\alpha$-H\"older continuous $\phi$, the H\"older exponent of a sub-action is strictly smaller than $\alpha$ \cite{Mo09}.

In order to find a sub-action $h$, T.~Bousch \cite{Bou00} proposed that it suffices to find a solution $h$ (also known as a \emph{calibrated sub-action} nowadays \cite{Ga17}) for the functional equation
\begin{equation}  \label{e_Bousch_function_equation}
h(x)  = - Q(T,\phi) + \max  \bigl\{ \phi (y) + h (y) : y \in T^{-1} (x)  \bigr\}.
\end{equation}
Equivalently, a function $h$ that satisfies (\ref{e_Bousch_function_equation}) is a fixed point of the nonlinear operator $\RR$ on the set of all real-valued functions on $X$ given by
\begin{equation} \label{e_Bousch_op}
\RR(u)(x) = - Q(T,\phi) + \max  \bigl\{ \phi (y) + h (y) : y \in T^{-1} (x)  \bigr\}.
\end{equation}
A related operator on a quotient function space was also studied by T.~Bousch in \cite{Bou00}.

We call the operator in (\ref{e_Bousch_op}) a \emph{Bousch operator} (or a \emph{Bousch--Lax operator}). In the context of Hamiltonian systems, an analogous construction gives the \emph{Lax--Oleinik semi-groups}\footnote{For this reason, the Bousch operator is sometimes referred to as the Lax--Oleinik operator or the Lax operators.} as studied by A.~Fathi \cite{Fa10}. The Bousch operator can be considered as a tropical version of the Ruelle--Perron--Frobenius operator in thermodynamic formalism, which was introduced by D.~Ruelle as an analog of the Ruelle--Araki transfer operator from classical statistical mechanics. In order to emphasize its connection to the Ruelle--Perron--Frobenius operator, we adopt in this paper a slightly modified version of the Bousch operator (see Section~\ref{sct_BouschOp}).

By investigating the Bousch operator in our setting with respect to visual metrics, we find a fixed point $u_\phi$ of the Bousch operator $\RR_{\overline\phi}$ for an $\alpha$-H\"older continuous $\phi$ and show that $u_\phi$ is $\alpha$-H\"older continuous (Proposition~\ref{p_calibrated_sub-action_exists}), establishing the Ma\~{n}\'{e} lemma. We view the Bousch operator as a Ruelle--Perron--Frobenius operator with respect to the max-plus algebra on $\R \cup \{-\infty\}$, and our proof of Proposition~\ref{p_calibrated_sub-action_exists} follows the proof of the corresponding result of the eigenfunctions for the Ruelle--Perron--Frobenius operators in \cite[Theorem~5.16]{Li18}.

We combine our local closing lemma and the existence of a calibrated sub-action in quantitative analysis to establish Theorem~\ref{t_Density_Lattes}~(i). More precisely, we show that for an arbitrary $\alpha$-H\"older continuous potential $\varphi \in \Lip(\wh\C, \sigma^{\alpha})$ with respect to the chordal metric $\sigma$, any perturbation of the form $\varphi' = \varphi - \epsilon \sigma (\cdot, \O)^\alpha$, with $\epsilon>0$ sufficiently small, belongs to $\sP^\alpha(\wh\C,\sigma)$, where $\O$ is some special periodic orbit produced from a calibrated sub-action and our local closing lemma away from critical points. The quantitative analysis is, however, carried out in the canonical orbifold metric $d$ on the related potentials $\wt\varphi$ and $\psi$, which are $\alpha$-H\"older continuous with respect to $d$ but not with respect to $\sigma$. In the case of Latt\`es maps, the canonical orbifold metric is also a visual metric. The technical parts are (1) to quantitatively avoid critical points $\crit f$ where the combinatorics are more involved in order to apply our local closing lemma as well as uniform local expansion property, and (2) to quantitatively avoid postcritical points $\post f$ where the conversion between $d$ and $\sigma$ is more involved. In fact, by applying various properties of the canonical orbifold metric and considering the orbifold ramification function, we get that the canonical orbifold metric and the chordal metric are ``locally comparable away from postcritical points''. We know that the identity map on $\wh\C$ between these two metrics is never bi-Lipschitz (see \cite[Appendix~A.10]{BM17}).

It is worth noting that even though sometimes certain ergodic properties of some non-uniformly expanding systems can be derived from associated symbolic models, and while an expanding Thurston map is (up to a sufficiently high iterate) semi-conjugate to a subshift of finite type via cell decompositions of $S^2$ and is semi-conjugate to a full shift via the geometric coding tree, it is not clear how to retrieve either the main theorems, or any one of the two main ingredients (i.e., the fine closing lemma in Lemma~\ref{l_bound_by_gap} and the existence of a calibrated sub-action with correct H\"older exponent in Proposition~\ref{p_calibrated_sub-action_exists}) for the proofs of the main theorems, or even the (local) Anosov closing lemma in Lemma~\ref{l_Local_Closing_Lemma} needed in the proof of Lemma~\ref{l_bound_by_gap} from the corresponding results for the symbolic systems. Our analysis is therefore focused on the phase space in order to retain as much information on the interactions between the dynamics, combinatorics, and metric geometry as possible.

For future directions, the discussions from this section suggest a ``dictionary'' for the correspondences between ergodic optimization and thermodynamic formalism. In this dictionary, the Bousch operators $\RR_\varphi$, $\RR_{\overline{\varphi}}$, and $\RR_{\wt\varphi}$ translate to the corresponding Ruelle--Perron--Frobenius operators, the maximal potential energy to the topological pressure, the maximizing measures (measures of maximal potential energy) to the equilibrium states, the existence and construction of the calibrated sub-action in Proposition~\ref{p_calibrated_sub-action_exists} to the existence and construction of the eigenfunction of the   Ruelle--Perron--Frobenius operator in \cite[Theorem~5.16]{Li18}, and finally our main theorems may be considered as the (degenerated) counterpart to the existence, uniqueness, and equidistribution of periodic points for the equilibrium state. Any new entries in this dictionary would be interesting. Moreover, our recent work \cite{DLZ24} with Jian~Ding on a probabilistic version of the TPO Conjecture, known as the Hunt--Ott conjecture \cite{HO96a, HO96b} and the prior work \cite{BZ16} suggest yet another ``language'' in the dictionary, namely, a probabilistic point of view through random maximum mean cycle problems on (directed) graphs, where periodic orbits correspond to (directed) cycles, periodic maximizing measures to cycles with maximum mean-weight, and the TPO and the Hunt--Ott Conjectures resemble the subcritical phenomena in \cite{Di13, DSW15}.

The local closing lemmas and local perturbation techniques we create in this paper should have applications in settings beyond uniformly expanding and uniformly hyperbolic ones. For example, Yinying~Huang, O.~Jenkinson, and the first-named author have applied them to a countable setting and established the Ma\~{n}\'{e} lemma and a version of the TPO Conjecture for the Gauss map \cite{HJL24}; Zelai~Hao, Yinying~Huang, O.~Jenkinson, and the first-named author have also applied them to a discontinuous setting and established the Ma\~{n}\'{e} lemma and a version of the TPO Conjecture for $\beta$-transformations \cite{HHJL24}. 

For a more general setting in low-dimensional dynamics, the topological Collet--Eckmann maps of one real or complex variable form another popular class of non-uniformly expanding systems, which has been extensively studied by S.~Smirnov, F.~Przytycki, J.~Rivera-Letelier, Weixiao~Shen, and others. It is natural to ask whether the TPO Conjecture holds for topological Collet--Eckmann maps with the geometric potential. In this setting, one may want to apply fine inducing schemes developed in the literature, but some structural stability result in the appropriate topology may also be needed, which is currently unavailable. 

Similarly, as remarked by the referee, it is plausible that a form of the TPO Conjecture should hold for weakly coarse expanding dynamical systems from \cite{DPTUZ21}, and it would be worth exploring this direction.

On the other hand, T.~Bousch established the TPO Conjecture for the space of potentials satisfying the Walters condition in \cite{Bou01} for the full shift. Similar to this space and the little Lipschitz spaces $\Lip(X,d^\alpha)$ ($\alpha \in (0,1)$),  it is interesting to consider other spaces of potentials larger than the Lipschitz space $\Lip(X,d)$ such as the Dini potentials studied by Aihua~Fan and Yunping~Jiang \cite{FJ01a, FJ01b}.

\smallskip

We now summarize the structure of this paper. In Section~\ref{sct_Notation}, some frequently used notations are recalled for the convenience of the reader. In Section~\ref{sct_Preliminaries}, we give a brief review of expanding Thurston maps, visual metrics, orbifolds, universal orbifold covering maps, Latt\`es maps, and the canonical orbifold metric. We also discuss a symbolic model for a sufficiently high iterate of an expanding Thurston map. Discussions on the little Lipschitz functions are limited to Sections~\ref{sct_lock} and~\ref{sct_lip}. In Section~\ref{sct_lock}, after recalling some facts on little Lipschitz functions, we establish a general result asserting that $\lock(X,d^\alpha)$ is equal to the dense interior of $\sP(X) \cap \lip(X,d^\alpha)$ for each $\alpha \in (0,1)$ and each continuous map on a compact metric space $(X,d)$ (Theorem~\ref{t_lock}). In Section~\ref{sct_Assumptions}, we state some assumptions on frequently-used objects in the paper for us to refer back later in order to simplify the presentation. In Section~\ref{sct_BouschOp}, we discuss the Bousch operators and some of their basic properties for general dynamical systems before proving the existence of an eigenfunction for the Bousch operator, also known as a calibrated sub-action, for an expanding Thurston map. We then establish the Liv\v{s}ic theorem, the Ma\~{n}\'{e} lemma, and the bilateral Ma\~{n}\'{e} lemma in our context. In Section~\ref{sct_ULEAFC}, we formulate and prove the uniform local expansion property of expanding Thurston maps away from critical points, which is crucial in the quantitative analysis in Sections~\ref{sct_Closing_Lemmas} and~\ref{sct_Holder_Density}. In Section~\ref{sct_Closing_Lemmas}, we establish in Lemma~\ref{l_bound_by_gap} a local closing lemma away from critical points. The proof relies on a local Anosov closing lemma and a (global) Bressaud--Quas closing lemma established in Subsection~\ref{subsct_Local_Anosov_Closing_Lemma} and Subsection~\ref{subsct_BQ}, respectively. Mechanisms to establish a Bressaud--Quas closing lemma for a general dynamical system are also discussed. In Section~\ref{sct_Holder_Density}, a proof of Theorem~\ref{t_Density_Lattes}~(i) is given, and then modifications necessary to establish Theorem~\ref{t_Density_Thurston}~(i) are described. Section~\ref{sct_lip} is devoted to the proofs of Theorems~\ref{t_Density_Rational_little},~\ref{t_Density_Thurston}~(ii),~\ref{t_Density_Lattes}~(ii),~\ref{t_gound_states}, and~\ref{t_Distance_expanding}.

An appendix is added to give illustrations of the combinatorial structures of three examples of expanding Thurston maps to help the reader gain some intuition. The rest of the paper is completely independent of the Appendix.

\section{Notation} \label{sct_Notation}
Let $\C$ be the complex plane and $\wh\C$ be the Riemann sphere. We follow the convention that $\N \coloneqq \{1, \, 2, \, 3, \, \dots\}$, $\N_0 \coloneqq \{0\} \cup \N$, and $\wh{\N} \coloneqq \N\cup \{+\infty\}$, with the order relations $<$, $\leq$, $>$, $\geq$ defined in the obvious way. For $x\in\R$, we define $\lfloor x\rfloor$ as the greatest integer $\leq x$. As usual, the symbol $\log$ denotes the logarithm to the base $e$, and $\log_c$ the logarithm to the base $c$ for $c>0$. The cardinality of a set $A$ is denoted by $\card{A}$. 

The collection of all maps from a set $X$ to a set $Y$ is denoted by $Y^X$. We denote the restriction of a map $g \: X \rightarrow Y$ to a subset $Z$ of $X$ by $g|_Z$. 

For a map $f\: X\rightarrow X$ and a real-valued function $\varphi\: X\rightarrow \R$, we write
$
S_n \varphi (x)   \coloneqq \sum_{j=0}^{n-1} \varphi \bigl( f^j(x) \bigr) 
$
for $x\in X$ and $n\in\N_0$. Note that by definition, we always have $S_0 \varphi = 0$.

Let $(X,d)$ be a metric space. For subsets $A,B\subseteq X$, we set $d(A,B) \coloneqq \inf \{d(x,y) : x\in A,\,y\in B\}$, and $d(A,x)=d(x,A) \coloneqq d(A,\{x\})$ for $x\in X$. For each subset $Y\subseteq X$, we denote the diameter of $Y$ by $\diam_d(Y) \coloneqq \sup\{d(x,y) : x, \, y\in Y\}$, the interior of $Y$ by $\inter Y$, and the max-plus characteristic function of $Y$ by $\mathbbold{0}_Y$, which maps each $x\in Y$ to $0\in\R$ and vanishes otherwise. We use the convention that $\mathbbold{0}=\mathbbold{0}_X$ when the space $X$ is clear from the context. For each $r>0$, we define $N^r_d(A)$ to be the open $r$-neighborhood $\{y\in X : d(y,A)<r\}$ of $A$, and $\overline{N}^r_d(A)$ the closed $r$-neighborhood $\{y\in X : d(y,A)\leq r\}$ of $A$. For $x\in X$, we denote the open (resp.\ closed) ball of radius $r$ centered at $x$ by $B_d(x, r)$ (resp.\ $\overline{B}_d(x,r)$). 

We set $\CCC(X)$ to be the space of continuous functions from $X$ to $\R$, $\MMM(X)$ the set of finite signed Borel measures, and $\PPP(X)$ the set of Borel probability measures on $X$. If we do not specify otherwise, we equip $\CCC(X)$ with the uniform norm $\Norm{\cdot}_{\CCC^0}$. For a continuous map $g\: X \rightarrow X$, $\MMM(X,g)$ is the set of $g$-invariant Borel probability measures on $X$. For each $x\in X$, we denote by $\delta_x$ the Dirac delta measure on $x$ given by $\delta_x(A) = 1$ if $x\in A$ and $0$ otherwise for all Borel measurable set $A \subseteq X$. 

We use $\Lip(X,d^\alpha)$ to denote the space of real-valued H\"{o}lder continuous functions on $(X,d)$ with an exponent $\alpha\in (0,1]$. For each $\psi\in\Lip(X,d^{\alpha})$, we denote
\begin{equation}   \label{e_Def|.|alpha}
\Hseminorm{d^\alpha}{\psi} \coloneqq \sup \{ \abs{\psi(x)- \psi(y)} / d(x,y)^\alpha  : x, \, y\in X, \,x\neq y \},
\end{equation}
and the H\"{o}lder norm is defined as $\Hnorm{d^\alpha}{\psi}{X} \coloneqq  \Hseminorm{d^\alpha}{\psi}  + \Norm{\psi}_{\CCC^0}$.

For a Lipschitz map $g\: (X,d)\rightarrow (X,d)$, we denote the Lipschitz constant by
\begin{equation}   \label{e_DefLipConst}
\LIP_d(g) \coloneqq \sup \{ d(g(x),g(y)) / d(x,y)  : x, \, y\in X \text{ with } x\neq y \}.
\end{equation}

\section{Preliminaries}  \label{sct_Preliminaries}

We review the definitions of expanding Thurston maps and visual metrics, and discuss some combinatorial structures and key metric properties associated to expanding Thurston maps in Subsection~\ref{subsct_ThurstonMap}. We then recall the notions of orbifolds and universal orbifold covering maps associated to expanding Thurston maps in Subsection~\ref{subsct_Orbifold}, leading to the definition of Latt\`es maps and the canonical orbifold metric in Subsection~\ref{subsct_Lattes_maps}. These two subsections are crucial in the proof of Theorem~\ref{t_Density_Lattes}, but not needed for Theorem~\ref{t_Density_Thurston}. Finally, in Subsection~\ref{subsct_SFT}, we quickly recall some notations from symbolic dynamics and show that a sufficiently high iterate of an expanding Thurston map is a factor of a one-sided subshift of finite type with a H\"older continuous factor map.

\subsection{Thurston maps} \label{subsct_ThurstonMap}
In this subsection, we go over some key concepts and results on Thurston maps, and expanding Thurston maps in particular. For a more thorough treatment of the subject, we refer to \cite{BM17}.

Let $S^2$ denote an oriented topological $2$-sphere. A continuous map $f\:S^2\rightarrow S^2$ is called a \defn{branched covering map} on $S^2$ if for each point $x\in S^2$, there exists a positive integer $d\in \N$, open neighborhoods $U$ of $x$ and $V$ of $y=f(x)$, open neighborhoods $U'$ and $V'$ of $0$ in $\wh\C$, and orientation-preserving homeomorphisms $\varphi\:U\rightarrow U'$ and $\eta\:V\rightarrow V'$ such that $\varphi(x)=0$, $\eta(y)=0$, and
\begin{equation*}
(\eta\circ f\circ\varphi^{-1})(z)=z^d
\end{equation*}
for each $z\in U'$. The positive integer $d$ above is called the \defn{local degree} of $f$ at $x$ and is denoted by $\deg_f (x)$. 

The \defn{degree} of $f$ is
\begin{equation}   \label{e_Deg=SumLocalDegree}
\deg f=\sum\limits_{x\in f^{-1}(y)} \deg_f (x)
\end{equation}
for $y\in S^2$ and is independent of $y$. If $f\:S^2\rightarrow S^2$ and $g\:S^2\rightarrow S^2$ are two branched covering maps on $S^2$, then so is $f\circ g$, and
\begin{equation} \label{e_LocalDegreeProduct}
 \deg_{f\circsmall g}(x) = \deg_g(x)\deg_f(g(x)), \qquad \text{for each } x\in S^2,
\end{equation}   
and moreover, 
\begin{equation}  \label{e_DegreeProduct}
\deg(f\circ g) =  (\deg f)( \deg g).
\end{equation}

A point $x\in S^2$ is a \defn{critical point} of $f$ if $\deg_f(x) \geq 2$. It follows immediately from the definition of the branched covering map above and the compactness of $S^2$ that there are only finitely many critical points of $f$. The set of critical points of $f$ is denoted by $\crit f$. A point $y\in S^2$ is a \defn{postcritical point} of $f$ if $y = f^n(x)$ for some $x\in\crit f$ and $n\in\N$. The set of postcritical points of $f$ is denoted by $\post f$. Note that $\post f=\post f^n$ for all $n\in\N$.

\begin{definition} [Thurston maps] \label{d_ThurstonMap}
A Thurston map is a branched covering map $f\:S^2\rightarrow S^2$ on $S^2$ with $\deg f\geq 2$ and $\card(\post f)<+\infty$.
\end{definition}

We now recall the notation for cell decompositions of $S^2$ used in \cite{BM17} and \cite{Li17}. A \defn{cell of dimension $n$} in $S^2$, $n \in \{1, \, 2\}$, is a subset $c\subseteq S^2$ that is homeomorphic to the closed unit ball $\overline{\B^n}$ in $\R^n$. We define the \defn{boundary of $c$}, denoted by $\partial c$, to be the set of points corresponding to $\partial\B^n$ under such a homeomorphism between $c$ and $\overline{\B^n}$. The \defn{interior of $c$} is defined to be $\inte (c) = c \setminus \partial c$. For each point $x\in S^2$, the set $\{x\}$ is considered as a \defn{cell of dimension $0$} in $S^2$. For a cell $c$ of dimension $0$, we adopt the convention that $\partial c=\emptyset$ and $\inte (c) =c$. 

We record the following three definitions from \cite{BM17}.

\begin{definition}[Cell decompositions]\label{d_celldecomp}
Let $\DD$ be a collection of cells in $S^2$.  We say that $\DD$ is a \defn{cell decomposition of $S^2$} if the following conditions are satisfied:

\begin{itemize}

\smallskip
\item[(i)]
the union of all cells in $\DD$ is equal to $S^2$,

\smallskip
\item[(ii)] if $c\in \DD$, then $\partial c$ is a union of cells in $\DD$,

\smallskip
\item[(iii)] for $c_1, \, c_2 \in \DD$ with $c_1 \neq c_2$, we have $\inte (c_1) \cap \inte (c_2)= \emptyset$,  

\smallskip
\item[(iv)] every point in $S^2$ has a neighborhood that meets only finitely-many cells in $\DD$.

\end{itemize}
\end{definition}

\begin{definition}[Refinements]\label{d_refine}
Let $\DD'$ and $\DD$ be two cell decompositions of $S^2$. We
say that $\DD'$ is a \defn{refinement} of $\DD$ if the following conditions are satisfied:
\begin{itemize}

\smallskip
\item[(i)] every cell $c\in \DD$ is the union of all cells $c'\in \DD'$ with $c'\subseteq c$,

\smallskip
\item[(ii)] for every cell $c'\in \DD'$ there exits a cell $c\in \DD$ with $c'\subseteq c$.

\end{itemize}
\end{definition}

\begin{definition}[Cellular maps and cellular Markov partitions]\label{d_cellular}
Let $\DD'$ and $\DD$ be two cell decompositions of  $S^2$. We say that a continuous map $f \: S^2 \rightarrow S^2$ is \defn{cellular} for  $(\DD', \DD)$ if for every cell $c\in \DD'$, the restriction $f|_c$ of $f$ to $c$ is a homeomorphism of $c$ onto a cell in $\DD$. We say that $(\DD',\DD)$ is a \defn{cellular Markov partition} for $f$ if $f$ is cellular for $(\DD',\DD)$ and $\DD'$ is a refinement of $\DD$.
\end{definition}

Let $f\:S^2 \rightarrow S^2$ be a Thurston map, and $\CC\subseteq S^2$ be a Jordan curve containing $\post f$. Then the pair $f$ and $\CC$ induces natural cell decompositions $\DD^n(f,\CC)$ of $S^2$, for $n\in\N_0$, in the following way:

By the Jordan curve theorem, the set $S^2\setminus\CC$ has two connected components. We call the closure of one of them the \defn{white $0$-tile} for $(f,\CC)$, denoted by $X^0_\w$, and the closure of the other the \defn{black $0$-tile} for $(f,\CC)$, denoted by $X^0_\b$. The set of \defn{$0$-tiles} is $\X^0(f,\CC) \coloneqq \bigl\{ X_\b^0, \, X_\w^0 \bigr\}$. The set of \defn{$0$-vertices} is $\V^0(f,\CC) \coloneqq \post f$. We set $\overline\V^0(f,\CC) \coloneqq \{ \{x\} : x\in \V^0(f,\CC) \}$. The set of \defn{$0$-edges} $\E^0(f,\CC)$ is the set of the closures of the connected components of $\CC \setminus  \post f$. Then we get a cell decomposition 
\begin{equation*}
\DD^0(f,\CC) \coloneqq \X^0(f,\CC) \cup \E^0(f,\CC) \cup \overline\V^0(f,\CC)
\end{equation*}
of $S^2$ consisting of \emph{cells of level $0$}, or \defn{$0$-cells}.

We can recursively define unique cell decompositions $\DD^n(f,\CC)$, $n\in\N$, consisting of \defn{$n$-cells} such that $f$ is cellular for $(\DD^{n+1}(f,\CC),\DD^n(f,\CC))$. We refer to \cite[Lemma~5.12]{BM17} for more details. We denote by $\X^n(f,\CC)$ the set of $n$-cells of dimension 2, called \defn{$n$-tiles}; by $\E^n(f,\CC)$ the set of $n$-cells of dimension 1, called \defn{$n$-edges}; by $\overline\V^n(f,\CC)$ the set of $n$-cells of dimension 0; and by $\V^n(f,\CC)$ the set $\bigl\{x  :  \{x\}\in \overline\V^n(f,\CC)\bigr\}$, called the set of \defn{$n$-vertices}. The \defn{$k$-skeleton}, for $k\in\{0, \, 1, \, 2\}$, of $\DD^n(f,\CC)$ is the union of all $n$-cells of dimension $k$ in this cell decomposition. 

We record Proposition~5.16 of \cite{BM17} here in order to summarize properties of the cell decompositions $\DD^n(f,\CC)$ defined above.

\begin{prop}[M.~Bonk \& D.~Meyer \cite{BM17}] \label{p_CellDecomp}
Let $k, \, n\in \N_0$, let   $f\: S^2\rightarrow S^2$ be a Thurston map,  $\CC\subseteq S^2$ be a Jordan curve with $\post f \subseteq \CC$, and   $m=\card(\post f)$. 
 
\smallskip
\begin{itemize}

\smallskip
\item[(i)] The map  $f^k$ is cellular for $\bigl( \DD^{n+k}(f,\CC), \DD^n(f,\CC) \bigr)$. In particular, if  $c$ is any $(n+k)$-cell, then $f^k(c)$ is an $n$-cell, and $f^k|_c$ is a homeomorphism of $c$ onto $f^k(c)$.

\smallskip
\item[(ii)]  Let  $c$ be  an $n$-cell.  Then $f^{-k}(c)$ is equal to the union of all 
$(n+k)$-cells $c'$ with $f^k(c')=c$.

\smallskip
\item[(iii)] The $1$-skeleton of $\DD^n(f,\CC)$ is  equal to  $f^{-n}(\CC)$. The $0$-skeleton of $\DD^n(f,\CC)$ is the set $\V^n(f,\CC)=f^{-n}(\post f )$, and we have $\V^n(f,\CC) \subseteq \V^{n+k}(f,\CC)$. 

\smallskip
\item[(iv)] $\card(\X^n(f,\CC))=2(\deg f)^n$,  $\card(\E^n(f,\CC))=m(\deg f)^n$,  and $\card (\V^n(f,\CC)) \leq m (\deg f)^n$.

\smallskip
\item[(v)] The $n$-edges are precisely the closures of the connected components of $f^{-n}(\CC)\setminus f^{-n}(\post f )$. The $n$-tiles are precisely the closures of the connected components of $S^2\setminus f^{-n}(\CC)$.

\smallskip
\item[(vi)] Every $n$-tile  is an $m$-gon, i.e., the number of $n$-edges and the number of $n$-vertices contained in its boundary are equal to $m$.  

\smallskip
\item[(vii)] Let $F\coloneqq f^k$ be an iterate of $f$ with $k \in \N$. Then $\DD^n(F,\CC) = \DD^{nk}(f,\CC)$.
\end{itemize}
\end{prop}

We also note that for each $n$-edge $e\in\E^n(f,\CC)$, $n\in\N_0$, there exist exactly two $n$-tiles $X, \, X'\in\X^n(f,\CC)$ such that $X\cap X' = e$.

From now on, if the map $f$ and the Jordan curve $\CC$ are apparent from the context, we will sometimes omit $(f,\CC)$ in the notation above.

If we fix the cell decomposition $\DD^n(f,\CC)$, $n\in\N_0$, we can define for each $v\in \V^n$ the \defn{$n$-flower of $v$} as
\begin{equation}   \label{e_Def_Flower}
W^n(v) \coloneqq \bigcup  \{\inte (c) : c\in \DD^n,\, v\in c \}.
\end{equation}
Note that flowers are open (in the standard topology on $S^2$). Let $\overline{W}^n(v)$ be the closure of $W^n(v)$.
\begin{rem}  \label{r_Flower}
For $n\in\N_0$ and $v\in\V^n$, we have 
\begin{equation*}
\overline{W}^n(v)=X_1\cup X_2\cup \cdots \cup X_m,
\end{equation*}
where $m \coloneqq 2\deg_{f^n}(v)$, and $X_1, X_2, \dots X_m$ are all the $n$-tiles that contain $v$ as a vertex (see \cite[Lemma~5.28]{BM17}). Moreover, each flower is mapped under $f$ to another flower in such a way that is similar to the map $z\mapsto z^k$ on the complex plane. More precisely, for $n\in\N_0$ and $v\in \V^{n+1}$, there exist orientation preserving homeomorphisms $\varphi\: W^{n+1}(v) \rightarrow \D$ and $\eta\: W^{n}(f(v)) \rightarrow \D$ such that $\D$ is the unit disk on $\C$, $\varphi(v)=0$, $\eta(f(v))=0$, and 
\begin{equation*}
(\eta\circ f \circ \varphi^{-1}) (z) = z^k
\end{equation*}
for all $z\in \D$, where $k \coloneqq \deg_f(v)$. Let $\overline{W}^{n+1}(v)= X_1\cup X_2\cup \cdots \cup X_m$ and $\overline{W}^n(f(v))= X'_1\cup X'_2\cup \cdots \cup X'_{m'}$, where $X_1,\, X_2,\, \dots, \, X_m$ are all the $(n+1)$-tiles that contain $v$ as a vertex, listed counterclockwise, and $X'_1, \, X'_2, \, \dots, \, X'_{m'}$ are all the $n$-tiles that contain $f(v)$ as a vertex, listed counterclockwise, and $f(X_1)=X'_1$. Then $m= m'k$, and $f(X_i)=X'_j$ if $i\equiv j \pmod{k}$, where $k=\deg_f(v)$. (See also Case~3 of the proof of Lemma~5.24 in \cite{BM17} for more details.) In particular, $W^n(v)$ is simply connected.
\end{rem}

We denote, for each $x\in S^2$ and $n\in\Z$, the \emph{$n$-bouquet of $x$}
\begin{equation}  \label{e_Def_U^n}
U^n(x) \coloneqq \bigcup \{Y^n\in \X^n  :    \text{there exists } X^n\in\X^n  
                                        \text{ with } x\in X^n, \, X^n\cap Y^n \neq \emptyset  \}  
\end{equation}
if $n\geq 0$, and set $U^n(x) \coloneqq S^2$ otherwise. 

We can now give a definition of expanding Thurston maps.

\begin{definition} [Expansion] \label{d_Expanding}
A Thurston map $f\:S^2\rightarrow S^2$ is called \defn{expanding} if there exists a metric $d$ on $S^2$ that induces the standard topology on $S^2$ and a Jordan curve $\CC\subseteq S^2$ containing $\post f$ such that 
\begin{equation*}
\lim\limits_{n\to+\infty}\max \{\diam_d(X) :  X\in \X^n(f,\CC)\}=0.
\end{equation*}
\end{definition}

\begin{rems}  \label{r_Expanding}
It is clear from Proposition~\ref{p_CellDecomp}~(vii) and Definition~\ref{d_Expanding} that if $f$ is an expanding Thurston map, so is $f^n$ for each $n\in\N$. We observe that being expanding is a topological property of a Thurston map and independent of the choice of the metric $d$ that generates the standard topology on $S^2$. By Lemma~6.2 in \cite{BM17}, it is also independent of the choice of the Jordan curve $\CC$ containing $\post f$. More precisely, if $f$ is an expanding Thurston map, then
\begin{equation*}
\lim\limits_{n\to+\infty}\max \!\bigl\{ \! \diam_{\wt{d}}(X) : X\in \X^n\bigl(f,\wt\CC \hspace{0.5mm}\bigr)\hspace{-0.3mm} \bigr\}\hspace{-0.3mm}=0,
\end{equation*}
for each metric $\wt{d}$ that generates the standard topology on $S^2$ and each Jordan curve $\wt\CC\subseteq S^2$ that contains $\post f$.
\end{rems}

P.~Ha\"{\i}ssinsky and K.~M.~Pilgrim developed a notion of expansion in a more general context for finite branched coverings between topological spaces (see \cite[Section~2.1 and Section~2.2]{HP09}). This applies to Thurston maps, and their notion of expansion is equivalent to our notion defined above in the context of Thurston maps (see \cite[Proposition~6.4]{BM17}). Such concepts of expansion are natural analogs, in the non-uniform setting, to some of the more classical notions, such as forward-expansive maps and distance-expanding maps. Our notion of expansion is not equivalent to any such classical notion in the context of Thurston maps. One topological obstruction comes from the presence of critical points for (non-homeomorphic) branched covering maps on $S^2$. In fact, as mentioned in the introduction, there are subtle connections between our notion of expansion and some classical notions of weak expansion. More precisely, one can prove that an expanding Thurston map is asymptotically $h$-expansive if and only if it has no periodic points. Moreover, such a map is never $h$-expansive. See \cite{Li15} for details.

For an expanding Thurston map $f$, we can fix a natural class of metrics $d$ on $S^2$ called \emph{visual metrics for $f$}. The construction of these metrics, inspired by Sullivan's dictionary, is closely related to the visual metrics on the boundary at infinity $\partial_\infty G$ of a Gromov hyperbolic group $G$. For the existence and properties of such metrics, see \cite[Chapter~8]{BM17}. For a visual metric $d$ for $f$, there exists a unique constant $\Lambda > 1$ called the \emph{expansion factor} of $d$ (under $f$) (see \cite[Chapter~8]{BM17} for more details). One major advantage of a visual metric $d$ is that in $(S^2,d)$, we have good quantitative control over the sizes of the cells in the cell decompositions discussed above. We summarize several results of this type (\cite[Proposition~8.4, Lemma~8.10, Lemma~8.11]{BM17}) in the lemma below.

\begin{lemma}[M.~Bonk \& D.~Meyer \cite{BM17}]   \label{l_CellBoundsBM}
Let $f\:S^2 \rightarrow S^2$ be an expanding Thurston map, and $\CC \subseteq S^2$ be a Jordan curve containing $\post f$. Let $d$ be a visual metric on $S^2$ for $f$ with expansion factor $\Lambda>1$. Then there exist constants $C\geq 1$, $C'\geq 1$, $K\geq 1$, and $n_0\in\N_0$ with the following properties:
\begin{enumerate}
\smallskip
\item[(i)] $d(\sigma,\tau) \geq C^{-1} \Lambda^{-n}$ whenever $\sigma$ and $\tau$ are disjoint $n$-cells for $n\in \N_0$.

\smallskip
\item[(ii)] $C^{-1} \Lambda^{-n} \leq \diam_d(\tau) \leq C\Lambda^{-n}$ for all $n$-edges and all $n$-tiles $\tau$ for $n\in\N_0$.

\smallskip
\item[(iii)] $B_d(x,K^{-1} \Lambda^{-n} ) \subseteq U^n(x) \subseteq B_d(x, K\Lambda^{-n})$ for $x\in S^2$ and $n\in\N_0$.

\smallskip
\item[(iv)] $U^{n+n_0} (x)\subseteq B_d(x,r) \subseteq U^{n-n_0}(x)$ where $n= \lceil -\log r / \log \Lambda \rceil$ for $r>0$ and $x\in S^2$.

\smallskip
\item[(v)] For every $n$-tile $X^n\in\X^n(f,\CC)$, $n\in\N_0$, there exists a point $p\in X^n$ such that $B_d(p,C^{-1}\Lambda^{-n}) \subseteq X^n \subseteq B_d(p,C\Lambda^{-n})$.
\end{enumerate}

Conversely, if $\wt{d}$ is a metric on $S^2$ satisfying conditions \textnormal{(i)} and \textnormal{(ii)} for some constant $C\geq 1$, then $\wt{d}$ is a visual metric with expansion factor $\Lambda>1$.
\end{lemma}

Recall that $U^n(x)$ is defined in (\ref{e_Def_U^n}).

In addition, we will need the fact that a visual metric $d$ induces the standard topology on $S^2$ (\cite[Proposition~8.3]{BM17}) and the fact that the metric space $(S^2,d)$ is linearly locally connected (\cite[Proposition~18.5]{BM17}). A metric space $(X,d)$ is \defn{linearly locally connected} if there exists a constant $L\geq 1$ such that the following conditions are satisfied:
\begin{enumerate}
\smallskip

\item  For all $z\in X$, $r > 0$, and $x, \, y\in B_d(z,r)$ with $x\neq y$, there exists a continuum $E\subseteq X$ with $x, \, y\in E$ and $E\subseteq B_d(z,rL)$.

\smallskip

\item For all $z\in X$, $r > 0$, and $x, \, y\in X \setminus B_d(z,r)$ with $x\neq y$, there exists a continuum $E\subseteq X$ with $x, \, y\in E$ and $E\subseteq X \setminus B_d(z,r/L)$.
\end{enumerate}
We call such a constant $L \geq 1$ a \defn{linear local connectivity constant of $d$}.

In fact, visual metrics serve a crucial role in connecting the dynamical arguments with geometric properties for rational expanding Thurston maps, especially Latt\`{e}s maps.

We first recall the following notions of equivalence between metric spaces.

\begin{definition}  \label{d_Quasi_Symmetry}
Consider two metric spaces $(X_1,d_1)$ and $(X_2,d_2)$. Let $g\: X_1 \rightarrow X_2$ be a homeomorphism. Then
\begin{enumerate}
\smallskip
\item[(i)] $g$ is \defn{bi-Lipschitz} if there exists a constant $C \geq 1$ such that for all $u,\, v \in X_1$,
\begin{equation*}
C^{-1} d_1( u, v ) \leq d_2 ( g(u), g(v) ) \leq C d_1 ( u, v ).
\end{equation*}

\smallskip
\item[(ii)] $g$ is \defn{bi-H\"older} if there exist constants $\alpha,\,\beta \in (0,1]$ and $C \geq 1$ such that for all $u,\, v \in X_1$,
\begin{equation*}
C^{-1} d_1( u, v )^{1/\beta} \leq d_2 ( g(u), g(v) ) \leq C d_1 ( u, v )^\alpha.
\end{equation*}

\smallskip
\item[(iii)] $g$ is a \defn{snowflake homeomorphism} if there exist constants $\alpha >0$ and $C \geq 1$ such that for all $u,\, v \in X_1$,
\begin{equation*}
C^{-1} d_1( u, v )^\alpha \leq d_2 ( g(u), g(v) ) \leq C d_1 ( u, v )^\alpha.
\end{equation*}

\smallskip
\item[(iv)] $g$ is a \defn{quasisymmetric homeomorphism} or a \defn{quasisymmetry} if there exists a homeomorphism $\eta \: [0, +\infty) \rightarrow [0, +\infty)$ such that for all pairwise distinct $u,\,v,\,w \in X_1$,
\begin{equation*}
\frac{ d_2 ( g(u), g(v) ) }{ d_2 ( g(u), g(w) ) } 
\leq \eta \biggl(  \frac{ d_1 ( u, v ) }{ d_1 ( u, w ) }  \biggr).
\end{equation*}
\end{enumerate} 
Moreover, the metric spaces $(X_1,d_1)$ and $(X_2,d_2)$ are \emph{bi-Lipschitz}, \emph{snowflake}, or \defn{quasisymmetrically equivalent} if there exists a homeomorphism from $(X_1,d_1)$ to $(X_2,d_2)$ with the corresponding property.

When $X_1 = X_2 \eqqcolon X$, then we say the metrics $d_1$ and $d_2$ are \emph{bi-Lipschitz}, \emph{bi-H\"older}, \emph{snowflake}, or \defn{quasisymmetrically equivalent} if the identity map from $(X,d_1)$ to $(X,d_2)$ has the corresponding property.
\end{definition}

\begin{theorem}[M.~Bonk \& D.~Meyer \cite{BM10, BM17}, P.~Ha\"issinsky \& K.~M.~Pilgrim \cite{HP09}]  \label{t_BM}
An expanding Thurston map is conjugate to a rational map if and only if the sphere $(S^2,d)$ equipped with a visual metric $d$ is quasisymmetrically equivalent to the Riemann sphere $\wh\C$ equipped with the chordal metric.
\end{theorem}   

See \cite[Theorem~18.1~(ii)]{BM17} for a proof. The chordal metric is recalled below.

\begin{rem}   \label{r_ChordalVisualQSEquiv}
In fact, in \cite[Lemma~18.10]{BM17}, M.~Bonk and D.~Meyer showed that for a rational expanding Thurston map $f\: \wh\C \rightarrow \wh\C$, the chordal metric $\sigma$ on the Riemann sphere $\wh\C$ is quasisymmetrically equivalent to each visual metric $d$ for $f$. Here the \emph{chordal metric} $\sigma$ on $\wh\C$ is given by
$
\sigma(z,w) =\frac{2\abs{z-w}}{\sqrt{1+\abs{z}^2} \sqrt{1+\abs{w}^2}}
$
for $z, \, w\in\C$, and $\sigma(\infty,z)=\sigma(z,\infty)= \frac{2}{\sqrt{1+\abs{z}^2}}$ for $z\in \C$. We also note that the inverse of a quasisymmetric homeomorphism is quasisymmetric (see \cite[Proposition~10.6]{He01}), and that quasisymmetric embeddings of bounded connected metric spaces are H\"{o}lder continuous (see \cite[Section~11.1 and Corollary~11.5]{He01}). Consequently the identity map between $(\wh\C,\sigma)$ and $(\wh\C,d)$ is bi-H\"older, and the class of H\"{o}lder continuous functions on $\wh\C$ equipped with the chordal metric and that on $\wh\C$ equipped with any visual metric for $f$ are the same (up to a change of the H\"{o}lder exponent). 

Moreover, since the \emph{spherical metric} on $\wh\C$ given by the length element $\mathrm{d}\sigma = 2 \abs{ \mathrm{d}z} / \bigl( 1 + \abs{z}^2 \bigr)$ is bi-Lipschitz equivalent to the chordal metric, we can replace the chordal metric by the spherical metric in the discussion above.
\end{rem}

A Jordan curve $\CC\subseteq S^2$ is \defn{$f$-invariant} if $f(\CC)\subseteq \CC$. We are interested in $f$-invariant Jordan curves that contain $\post f$, since for such a Jordan curve $\CC$, we get a cellular Markov partition $(\DD^1(f,\CC),\DD^0(f,\CC))$ for $f$. According to Example~15.11 in \cite{BM17}, such $f$-invariant Jordan curves containing $\post{f}$ need not exist. However, M.~Bonk and D.~Meyer \cite[Theorem~15.1]{BM17} proved that there exists an $f^n$-invariant Jordan curve $\CC$ containing $\post{f}$ for each sufficiently large $n$ depending on $f$. 

\begin{lemma}[M.~Bonk \& D.~Meyer \cite{BM17}]  \label{l_CexistsBM}
Let $f\:S^2\rightarrow S^2$ be an expanding Thurston map, and $\wt{\CC}\subseteq S^2$ be a Jordan curve with $\post f\subseteq \wt{\CC}$. Then there exists an integer $N(f,\wt{\CC}) \in \N$ such that for each $n\geq N(f,\wt{\CC})$ there exists an $f^n$-invariant Jordan curve $\CC$ isotopic to $\wt{\CC}$ rel.\ $\post f$.
\end{lemma}

We now discuss some metric estimates for the dynamics induced by expanding Thurston maps.

\begin{lemma}   \label{l_ETM_Lipschitz}
Let $f\:S^2 \rightarrow S^2$ be an expanding Thurston map, and $d$ be a visual metric on $S^2$ for $f$. Then $f$ is Lipschitz with respect to $d$ with Lipschitz constant $\LIP_d(f) > 1$.
\end{lemma}

\begin{proof}
It is shown in \cite[Lemma~3.12]{Li18} that $f$ is Lipschitz with respect to $d$. In order to show $\LIP_d(f) > 1$, we argue by contradiction and suppose $\LIP_d(f) \leq 1$. Fix an arbitrary Jordan curve $\CC \subseteq S^2$ containing $\post f$. Then by Proposition~\ref{p_CellDecomp}~(i), $\diam_d(X^n) \geq \diam_d ( f^n ( X^n ) )$ for each $n$-tile $X^n \in \X^n(f, \CC)$. This contradicts with Lemma~\ref{l_CellBoundsBM}~(ii).
\end{proof}

The following lemma proved in \cite[Lemma~3.13]{Li18} generalizes \cite[Lemma~15.25]{BM17}.

\begin{lemma}[M.~Bonk \& D.~Meyer \cite{BM17}, Zhiqiang~Li \cite{Li18}]   \label{l_MetricDistortion}
Let $f\:S^2 \rightarrow S^2$ be an expanding Thurston map, and $\CC \subseteq S^2$ be a Jordan curve that satisfies $\post f \subseteq \CC$ and $f^{n_\CC}(\CC)\subseteq\CC$ for some $n_\CC\in\N$. Let $d$ be a visual metric on $S^2$ for $f$ with expansion factor $\Lambda>1$. Then there exists a constant $C_0 > 1$, depending only on $f$, $d$, $\CC$, and $n_\CC$, with the following property:

If $k, \, n\in\N_0$, $X^{n+k}\in\X^{n+k}(f,\CC)$, and $x, \, y\in X^{n+k}$, then 
\begin{equation}   \label{e_MetricDistortion}
\frac{1}{C_0} d(x,y) \leq \frac{d(f^n(x),f^n(y))}{\Lambda^n}  \leq C_0 d(x,y).
\end{equation}
\end{lemma}

We recall the following key estimate, which also serves as a cornerstone for the analysis in the theory of thermodynamic formalism. See \cite[Lemma~5.1]{Li18} for a proof.

\begin{lemma} \label{l_key_estimate}
Let $f\:S^2 \rightarrow S^2$ be an expanding Thurston map, and $\CC \subseteq S^2$ be a Jordan curve that satisfies $\post f \subseteq \CC$ and $f^{n_\CC}(\CC)\subseteq\CC$ for some $n_\CC\in\N$. Let $d$ be a visual metric on $S^2$ for $f$ with expansion factor $\Lambda>1$. Let $\phi \in \Lip(S^2,d^{\alpha})$ be a real-valued H\"older continuous function with an exponent $\alpha \in (0,1]$. Then for all $n, \, m \in \N_0$ with $n\leq m$, $X^m\in\X^m(f,\CC)$, and $x,\,y\in X^m$, we have
\begin{equation*}
\abs{S_n \phi (x) - S_n \phi (y) } \leq \frac{C_0}{1- \Lambda^{-\alpha}}  \Hseminorm{d^\alpha}{\phi} d ( f^n(x), f^n(y) )^\alpha,
\end{equation*}
where $C_0 > 1$ is a constant from Lemma~\ref{l_MetricDistortion}.
\end{lemma}

\subsection{Orbifolds}   \label{subsct_Orbifold}

We consider orbifolds associated to Thurston maps. An orbifold is a space that is locally represented as a quotient of a model space by a group action (see \cite[Chapter~13]{Th80}). For the purpose of this work, we restrict ourselves to orbifolds on $S^2$. In this context, only cyclic groups can occur, so a simpler definition (than that of W.~P.~Thurston) will be used. We follow closely the setup from \cite{BM17}.

An \defn{orbifold} is a pair $\mathcal{O} = (S, \alpha)$, where $S$ is a surface and $\alpha\: S \rightarrow \wh{\N} = \N \cup \{+\infty\}$ is a map such that the set of points $p\in S$ with $\alpha(p) \neq 1$ is a discrete set in $S$, i.e., it has no limit points in $S$. We call such a function $\alpha$ a \defn{ramification function} on $S$. The set 
\begin{equation}   \label{e_Def_supp}
\supp(\alpha) \coloneqq \{p\in S : \alpha(p)\geq 2 \}
\end{equation}
is the \defn{support} of $\alpha$. We will only consider orbifolds with $S=S^2$, an oriented $2$-sphere, in this paper.

The \defn{Euler characteristic} of an orbifold $\mathcal{O}= (S^2,\alpha)$ is defined as
\begin{equation*}
\chi(\mathcal{O}) \coloneqq 2 - \sum_{x\in S^2} \biggl(1- \frac{1}{\alpha(x)} \biggr),
\end{equation*}
where we use the convention $\frac{1}{+\infty} = 0$, and note that the terms in the summation are nonzero on a finite set of points. The orbifold $\mathcal{O}$ is \defn{parabolic} if $\chi(\mathcal{O}) = 0$ and \defn{hyperbolic} if $\chi(\mathcal{O}) < 0$.

Every Thurston map $f$ has an associated orbifold $\mathcal{O}_f = (S^2,\alpha_f)$, which plays an important role in this section.

\begin{definition}  \label{d_Ramification_Fn}
Let $f\: S^2\rightarrow S^2$ be a Thurston map. The \defn{ramification function} of $f$ is the map $\alpha_f\: S^2\rightarrow\wh{\N}$ defined as
\begin{equation} \label{e_Ramification_Fn}
\alpha_f(x) \coloneqq \lcm \bigl\{\deg_{f^n}(y) :    y\in S^2,  \, n\in \N, \text{ and } f^n(y)=x \bigr\}
\end{equation}
for $x\in S^2$.
\end{definition}
Here $\wh{\N} = \N\cup \{+\infty\}$ with the order relations $<$, $\leq$, $>$, $\geq$ extended in the obvious way, and $\lcm$ denotes the least common multiple on $\wh{\N}$ defined by $\lcm(A)=+\infty$ if $A\subseteq \wh{\N}$ is not a bounded set of natural numbers, and otherwise $\lcm(A)$ is calculated in the usual way. Note that different Thurston maps can share the same ramification function; in particular, we have the following fact from \cite[Proposition~2.16]{BM17}.

\begin{prop}   \label{p_Same_Ramification_Fn}
Let $f\: S^2\rightarrow S^2$ be a Thurston map. Then $\alpha_f = \alpha_{f^n}$ for each $n\in \N$.
\end{prop}

\begin{definition}[Orbifolds associated to Thurston maps]  \label{d_Orbifold_Thurston_Maps}
Let $f\: S^2\rightarrow S^2$ be a Thurston map. The \defn{orbifold associated to $f$} is a pair $\mathcal{O}_f \coloneqq (S^2,\alpha_f)$, where $\alpha_f\: S^2\rightarrow \wh{\N}$ is the ramification function of $f$.
\end{definition}

Orbifolds associated to Thurston maps are either parabolic or hyperbolic (see \cite[Proposition~2.12]{BM17}).

\smallskip

For an orbifold $\mathcal{O} = (S^2,\alpha)$, we set
\begin{equation}  \label{e_Def_S0}
S_0^2 \coloneqq S^2 \setminus \bigl\{x\in S^2   :   \alpha(x) = +\infty \bigr\}.
\end{equation}

We record the following facts from \cite{BM17}, whose proofs can be found in \cite{BM17} and references therein (see Theorem~A.26 and Corollary~A.29 in \cite{BM17}).

\begin{theorem}  \label{t_Uni_Orb_Cover_BM}
Let $\mathcal{O}=(S^2,\alpha)$ be an orbifold that is parabolic or hyperbolic. Then the following statements are satisfied:
\begin{enumerate}
\smallskip
\item[(i)] There exists a simply connected surface $\XX$ and a branched covering map $\Theta\: \XX\rightarrow S_0^2$ such that $ \deg_\Theta(x) = \alpha(\Theta(x)) $ for each $x\in \XX$.

\smallskip
\item[(ii)] The branched covering map $\Theta$ in $\operatorname{(i)}$ is unique. More precisely, if $\wt \XX$ is a simply connected surface and $\wt\Theta \: \wt \XX \rightarrow S_0^2$ satisfies $\deg_{\wt\Theta}(y) = \alpha\bigl(\wt\Theta(x)\bigr)$ for each $y\in \wt \XX$, then for all points $x_0\in \XX$ and $\wt{x}_0 \in \wt \XX$ with $\Theta(x_0)=\wt\Theta(\wt{x}_0)$ there exists orientation-preserving homeomorphism $A\: \XX\rightarrow \wt \XX$ with $A(x_0)= \wt{x}_0$ and $\Theta=\wt\Theta\circ A$. Moreover, if $\alpha(\Theta(x_0)) = 1$, then $A$ is unique.
\end{enumerate}
\end{theorem}

\begin{definition}[Universal orbifold covering maps]  \label{d_Uni_Orb_Cover}
Let $\mathcal{O}=(S^2,\alpha)$ be an orbifold that is parabolic or hyperbolic. The map $\Theta\: \XX \rightarrow S_0^2$ from Theorem~\ref{t_Uni_Orb_Cover_BM} is called the \defn{universal orbifold covering map} of $\mathcal{O}$.
\end{definition}

\subsection{Latt\`es maps}   \label{subsct_Lattes_maps}

For a rational expanding Thurston map $f\: \wh\C  \rightarrow \wh\C$, the chordal metric $\sigma$ (see Remark~\ref{r_ChordalVisualQSEquiv} for the definition), which is bi-Lipschitz equivalent to the Euclidean metric away from the infinity, is never a visual metric for $f$ (see \cite[Lemma~8.12]{BM17}). In fact, for each visual metric $d$ for $f$, the visual sphere $(S^2,d)$ is snowflake equivalent to $\bigl(\wh\C, \sigma \bigr)$ if and only if $f$ is topologically conjugate to a Latt\`{e}s map (see \cite[Theorem~18.1~(iii)]{BM17} and Definition~\ref{d_Lattes_Map} below). 

Recall that we call two metric spaces $(X_1,d_1)$ and $(X_2,d_2)$ are \defn{bi-Lipschitz}, \defn{snowflake}, or \defn{quasisymmetrically equivalent} if there exists a homeomorphism from $(X_1,d_1)$ to $(X_2,d_2)$ with the corresponding property (see Definition~\ref{d_Quasi_Symmetry}).

We recall a version of the definition of Latt\`{e}s maps.

\begin{definition}   \label{d_Lattes_Map}
Let $f\: \wh\C  \rightarrow \wh\C$ be a rational Thurston map on the Riemann sphere $\wh\C$. If $f$ is expanding and the orbifold $\mathcal{O}_f = (S^2, \alpha_f)$ associated to $f$ is parabolic, then it is called a \defn{Latt\`{e}s map}.
\end{definition}

See \cite[Chapter~3]{BM17} and \cite{Mil06} for other definitions and more properties of Latt\`{e}s maps.

\begin{rem}  \label{r_Canonical_Orbifold_Metric}
For a Latt\`{e}s map $f\: \wh\C \rightarrow \wh\C$, the universal orbifold covering map $\Theta \: \C \rightarrow \wh\C$ of the orbifold $\mathcal{O}_f = \bigl( \wh\C, \alpha_f \bigr)$ associated to $f$ is holomorphic (see \cite[Theorem~A.26, Definition~A.27, and Corollary~A.29]{BM17}). Let $d_0$ be the Euclidean metric on $\C$. Then the \defn{canonical orbifold metric} $\omega_f$ of $f$ is the pushforward of $d_0$ by $\Theta$, more precisely,
\begin{equation*}
\omega_f(p,q)  \coloneqq \inf \bigl\{ d_0(z,w) :   z\in \Theta^{-1}(p), \, w\in \Theta^{-1}(q)  \bigr\} 
\end{equation*}
for $p, \, q\in\wh\C$ (see Section~2.5 and Appendices~A.9 and A.10 in \cite{BM17} for more details on the canonical orbifold metric). Let $\sigma$ be the chordal metric on $\wh\C$ as recalled in Remark~\ref{r_ChordalVisualQSEquiv}. By \cite[Proposition~8.5]{BM17}, $\omega_f$ is a visual metric for $f$. By \cite[Lemma~A.34]{BM17}, $\bigl( \wh\C, \omega_f \bigr)$ and $\bigl( \wh\C, \sigma \bigr)$ are bi-Lipschitz equivalent, i.e., there exists a bi-Lipschitz homeomorphism $h \:  \wh\C  \rightarrow  \wh\C$ from $\bigl( \wh\C, \omega_f \bigr)$ to $\bigl( \wh\C, \sigma \bigr)$. Moreover, by the discussion in \cite[Appendix~A.10]{BM17}, $h$ cannot be the identity map.
\end{rem}

\subsection{Symbolic dynamics for expanding Thurston maps}   \label{subsct_SFT}

In this subsection, we give a brief review of the dynamics of one-sided subshifts of finite type and discuss a symbolic model for sufficiently large iterations of expanding Thurston maps. We refer the reader to \cite{Ki98} for a beautiful introduction to symbolic dynamics. For a discussion on results on subshifts of finite type related to our context, see \cite{PP90, Ba00}.

Let $S$ be a finite nonempty set, and $A \: S\times S \rightarrow \{0, \, 1\}$ be a matrix whose entries are either $0$ or $1$. We denote the \defn{set of admissible sequences defined by $A$} by
\begin{equation*}
\Sigma_A^+ \= \{ \{x_i\}_{i\in\N_0} :  x_i \in S, \, A(x_i,x_{i+1})=1, \, \text{for each } i\in\N_0\}.
\end{equation*}
Fix a number $\theta\in(0,1)$. We equip the set $\Sigma_A^+$ with a metric $d_\theta$ given by $d_\theta(\{x_i\}_{i\in\N_0},\{y_i\}_{i\in\N_0})=\theta^N$ for $\{x_i\}_{i\in\N_0} \neq \{y_i\}_{i\in\N_0}$, where $N$ is the smallest integer with $x_N \neq y_N$. The topology on the metric space $\left(\Sigma_A^+,d_\theta \right)$ coincides with that induced from the product topology, and is therefore compact.

The \defn{left-shift operator} $\sigma_A \: \Sigma_A^+ \rightarrow \Sigma_A^+$ (defined by $A$) is given by
\begin{equation*}
\sigma_A ( \{x_i\}_{i\in\N_0} ) \= \{x_{i+1}\}_{i\in\N_0}  \qquad \text{for } \{x_i\}_{i\in\N_0} \in \Sigma_A^+.
\end{equation*}

The pair $\left(\Sigma_A^+, \sigma_A\right)$ is called the \defn{one-sided subshift of finite type} defined by $A$. The set $S$ is called the \defn{set of states} and the matrix $A\: S\times S \rightarrow \{0, \, 1\}$ is called the \defn{transition matrix}.

Let $X$ and $Y$ be topological spaces, and consider two maps $f\:X\rightarrow X$ and $g\:Y\rightarrow Y$. We say that the topological dynamical system $(X,f)$ is a \defn{factor} of the topological dynamical system $(Y,g)$ if there is a surjective continuous map $\pi\:Y\rightarrow X$ such that $\pi\circ g=f\circ\pi$. We call the map $\pi\: Y\rightarrow X$ a \emph{factor map}.

We will now consider a one-sided subshift of finite type associated to an expanding Thurston map and an invariant Jordan curve on $S^2$ containing $\post f$. Recall from the discussions in Subsection~\ref{subsct_ThurstonMap} that such an invariant Jordan curve may not exist for an expanding Thurston map, but does exist for each of its sufficiently high iterate. We will need the following technical lemma. 

\begin{lemma}   \label{l_CylinderIsTile}
Let $f\: S^2 \rightarrow S^2$ be an expanding Thurston map with a Jordan curve $\CC\subseteq S^2$ satisfying $f(\CC)\subseteq \CC$ and $\post f\subseteq \CC$. Let $\{X_i\}_{i\in\N_0}$ be a sequence of $1$-tiles in $\X^1(f,\CC)$ satisfying $f(X_i)\supseteq X_{i+1}$ for all $i\in\N_0$. Then for each $n\in\N$, we have
\begin{equation}  \label{e_CylinderIsTile}
\bigl( (f|_{X_0})^{-1} \circ (f|_{X_1})^{-1} \circ \cdots \circ (f|_{X_{n-2}})^{-1} \bigr) (X_{n-1}) =  \bigcap\limits_{i=0}^{n-1}  f^{-i} (X_i) \in \X^n(f,\CC). 
\end{equation}
Moreover, $\card \bigcap_{i\in\N_0} f^{-i}(X_i) = 1$.
\end{lemma}

\begin{proof}
Let $d$ be a visual metric on $S^2$ for $f$.

We call a sequence $\{c_i\}_{i\in\N_0}$ of subsets of $S^2$ admissible if $f(c_i)\supseteq c_{i+1}$ for all $i\in\N_0$.
  
We prove (\ref{e_CylinderIsTile}) by induction.

For $n=1$, (\ref{e_CylinderIsTile}) holds trivially for each admissible sequence of $1$-tiles $\{X_i\}_{i\in\N_0}$ in $\X^1$.

Assume that (\ref{e_CylinderIsTile}) holds for each admissible sequence of $1$-tiles $\{X_i\}_{i\in\N_0}$ in $\X^1$ and for $n=m$ for some $m\in\N$. We fix such a sequence $\{X_i\}_{i\in\N_0}$. Then $\{X_{i+1}\}_{i\in\N_0}$ is also admissible. By the induction hypothesis, we denote
\begin{equation*}
X^m \coloneqq \bigl( (f|_{X_1})^{-1} \circ (f|_{X_2})^{-1} \circ \cdots \circ (f|_{X_{m-1}})^{-1} \bigr) (X_m) = \bigcap\limits_{i=0}^{m-1} f^{-i}(X_{i+1}) \in \X^m. 
\end{equation*}
Since $f(X_0) \supseteq X_1$ and $X^m\subseteq X_1$, we get from Proposition~\ref{p_CellDecomp}~(i) and (ii) that $f$ is injective on $X_1$, and thus, $\bigcap\limits_{i=0}^{m}  f^{-i} (X_i) = X_0 \cap f^{-1}(X^m) \in \X^{m+1}$, and $(f|_{X_0})^{-1} (X^m) = X_0 \cap f^{-1}(X^m) \in \X^{m+1}$.

The induction is complete. We have established (\ref{e_CylinderIsTile}).

Note that $\bigcap\limits_{i=0}^{n-1}  f^{-i} (X_i)  \supseteq \bigcap\limits_{i=0}^{n}  f^{-i} (X_i) \in \X^{n+1}$ for each $n\in\N$. By Lemma~\ref{l_CellBoundsBM}~(ii), $\bigcap_{i\in\N_0} f^{-i}(X_i)$ is the intersection of a nested sequence of closed sets with radii convergent to zero. Thus, it contains exactly one point in $S^2$.
\end{proof}

\begin{prop}   \label{p_TileSFT}
Let $f\: S^2 \rightarrow S^2$ be an expanding Thurston map with a Jordan curve $\CC\subseteq S^2$ satisfying $f(\CC)\subseteq \CC$ and $\post f\subseteq \CC$. Let $d$ be a visual metric on $S^2$ for $f$ with expansion factor $\Lambda>1$. Fix $\theta\in(0,1)$. We set $S_{\ti} \coloneqq \X^1(f,\CC)$, and define a transition matrix $A_{\ti}\: S_{\ti}\times S_{\ti} \rightarrow \{0, \, 1\}$ by
\begin{equation*}
A_{\ti}(X,X') \= \begin{cases} 1 & \text{if } f(X)\supseteq X', \\ 0  & \text{otherwise}  \end{cases}
\end{equation*}
for $X, \, X'\in \X^1(f,\CC)$. Then $f$ is a factor of the one-sided subshift of finite type $\bigl(\Sigma_{A_{\ti}}^+, \sigma_{A_{\ti}}\bigr)$ defined by the transition matrix $A_{\ti}$, where the factor map $\pi_{\ti}\: \Sigma_{A_{\ti}}^+ \rightarrow S^2$ is a surjective H\"{o}lder continuous map defined by 
\begin{equation}   \label{e_DefTileSFTFactorMap}
\pi_{\ti} \left( \{X_i\}_{i\in\N_0} \right)= x,  \text{ where } \{x\} = \bigcap\limits_{i \in \N_0} f^{-i} (X_i).
\end{equation}
Here $\Sigma_{A_{\ti}}^+$ is equipped with the metric $d_\theta$ defined in Subsection~\ref{subsct_SFT}, and $S^2$ is equipped with the visual metric $d$.
\end{prop}

\begin{proof}
We denote by $\{X_i\}_{i\in\N_0} \in \Sigma_{A_{\ti}}^+$ an arbitrary admissible sequence.

Since $f(X_i)\supseteq X_{i+1}$ for each $i\in\N_0$, by Lemma~\ref{l_CylinderIsTile}, the map $\pi_{\ti}$ is well-defined.

Note that for each $m\in \N_0$ and each $\{X'_i\}_{i\in\N_0} \in \Sigma_{A_{\ti}}^+$ with $X_{m+1} \neq X'_{m+1}$ and $X_j=X'_j$ for each integer $j\in[0,m]$, we have $\{ \pi_{\ti}  (\{X_i\}_{i\in\N_0}  ),  \, \pi_{\ti} ( \{ X'_i \}_{i\in\N_0} ) \} \subseteq \bigcap\limits_{i=0}^{m} f^{-i}(X_i) \in \X^{m+1}$ by Lemma~\ref{l_CylinderIsTile}. Thus, it follows from Lemma~\ref{l_CellBoundsBM}~(ii) that $\pi_{\ti}$ is H\"{o}lder continuous.

To see that $\pi_{\ti}$ is surjective, we observe that for each $x\in S^2$, we can find a sequence $\bigl\{ X^j(x) \bigr\}_{j\in\N}$ of tiles such that $X^j(x) \in \X^j$, $x\in X^j(x)$, and $X^j(x) \supseteq X^{j+1}(x)$ for each $j\in\N$. Then it is clear that $\bigl\{  f^i\bigl( X^{i+1}(x) \bigr) \bigr\}_{i\in\N_0} \in \Sigma_{A_{\ti}}^+$ and $\pi_{\ti} \Bigl(  \bigl\{ f^i\bigl( X^{i+1}(x) \bigr) \bigr\}_{i\in\N_0} \Bigr) = x$.

We observe that
\begin{align*}
            \{ (f\circ \pi_{\ti}) (\{X_i\}_{i\in\N_0}) \}
  =       & f \biggl(  \bigcap\limits_{j\in\N_0} f^{-j} (X_j) \biggr)
\subseteq  \bigcap\limits_{j\in\N} f^{-(j-1)} (X_j) \\
  =       & \bigcap\limits_{i\in\N_0} f^{-i} (X_{i+1})  
  =         \{(\pi_{\ti} \circ \sigma_{A_{\ti}} ) (\{X_i\}_{i\in\N_0})  \}.
\end{align*}
Therefore, it follows that $\pi_{\ti} \circ \sigma_{A_{\ti}} = f \circ \pi_{\ti}$.
\end{proof}

\section{Little Lipschitz spaces and the locking property}   \label{sct_lock}
This section aims to show the following theorem in a general setting.

\begin{theorem}  \label{t_lock}
Let $T\: X \rightarrow X$ be a continuous map on a compact metric space $(X,d)$ and $\alpha \in (0,1)$. Then the set $\lock(X,d^\alpha)$ of functions with the locking property in $\lip(X,d^\alpha)$ is equal to the interior of $\sP(X) \cap \lip(X,d^\alpha)$ (in the induced topology of $\lip(X,d^\alpha)$ as a subspace of $\Lip(X,d^\alpha)$), and $\lock(X,d^\alpha)$ is dense in $\sP(X) \cap \lip(X,d^\alpha)$.
\end{theorem}

The strategy of the proof is a two-step perturbation argument. For every $\phi \in \sP(X) \cap \lip(X,d^\alpha)$ with a maximizing measure supported on a periodic orbit $\O$, we first perturb it by suppressing it by a scalar multiple of the distance function $d^\beta(\cdot, \O)$, then show that the resulting function $\phi_t\= \phi - t d^\beta(\cdot, \O)$ stays in $\lock(X,d^\alpha)$ after another arbitrary small perturbation $\psi \in \lip(X,d^\alpha)$. The difficulty lies in the fact that if $\beta=\alpha$, then $d^\beta(\cdot, \O)$ is no longer in $\lip(X,d^\alpha)$. On the other hand, even though $d^\beta(\cdot, \O) \in \lip(X,d^\alpha)$ for all $\beta \in (\alpha, 1]$, the Wasserstein distance type of control of J.~Bochi and the second-named author \cite[Lemma~2]{BZ15} cannot be used directly as there is a mismatch in exponents between the $d^\beta(\cdot, \O)$ and $\psi$. Our approach is to split the small arbitrary perturbation $\psi \in \lip(X,d^\alpha)$ into two parts $\psi = \psi_\beta + \tau$, where $\psi_\beta \in \Lip(X,d^\beta)$ has small $\beta$-H\"older norm and $\tau$ has arbitrarily small uniform norm. Such a splitting is possible thanks to a Lipschitz extension theorem stated in Theorem~\ref{t_Lip_extension}.

We first recall the definition of little Lipschitz functions.

\begin{definition}
Let $X$ be a compact metric space, the \emph{little Lipschitz space} $\lip (X,d)$ is defined to be the subspace of $\Lip(X,d)$ (equipped with the Lipschitz norm $\Hnorm{d}{\cdot}{X}$) consisting of those functions $\phi$ with the property that for every $\epsilon > 0$ there exists $\delta>0$ such that 
\begin{equation*}
\frac{ \abs{ \phi(x) - \phi(y) } }{ d(x,y) }  \leq \epsilon
\end{equation*}
for all $x,\,y \in X$ with $0<d(x,y) \leq \delta$. The functions in $\lip(X,d)$ are called \emph{little Lipschitz functions} or \emph{locally flat Lipschitz functions}.
\end{definition}

It is clear that a function $\phi$ is in $\lip(X,d)$ (resp.\  $\Lip(X,d)$) if and only if $\sup\{ \abs{ \phi (x) - \phi (y) } : d(x,y) \leq r  \}= o ( r )$ (resp.\ $\sup\{ \abs{ \phi (x) - \phi (y) } : d(x,y) \leq r  \}= O ( r )$) as $r \to 0$.

We consider connections between little Lipschitz spaces and big Lipschitz spaces in the following two propositions. The first one is classical (see for example, \cite[Subsection~I.1]{Sh64}).

\begin{prop}   \label{p_lip_closed}
Let $(X,d)$ be a compact metric space. Then $\lip(X,d)$ is a closed subalgebra of $\Lip(X,d)$ which contains the constant functions.
\end{prop}

We will rely on the following important result. See \cite[Corollary~3.7]{BCD87}, \cite[Corollary~1.5]{We96}, and \cite[Corollary~8.28]{We18}. 

\begin{prop}  \label{p_Lip_dense_lip_alpha}
Let $(X,d)$ be a compact metric space and let $\alpha \in (0,1)$. Then $\Lip(X,d)$ is a dense subset of $\lip(X,d^\alpha)$ (equipped with the norm $\Hnorm{d^\alpha}{\cdot}{X}$).
\end{prop}

Next, we record a basic covering theorem (see for example, \cite[Theorem~1.2]{He01}).

\begin{theorem}[Covering Theorem]  \label{t_covering_theorem}
Every family $\cF$ of balls of uniformly bounded diameter in a metric space $X$ contains a disjoint subfamily $\cG$ such that
\begin{equation}
\bigcup_{B \in \cF} B \subseteq \bigcup_{B \in \cG} 5B.
\end{equation}
Here $5B$ denotes the ball with the same center with a radius $5$ times that of $B$.
\end{theorem}

The following extension theorem for Lipschitz functions dates back to a result of T.~Botts and E.~McShane recorded in \cite[Proposition~1.4]{Sh64}. We state here the version from \cite[Proposition~3]{Ha92}. See also \cite[Lemma~3.3]{BCD87}.

\begin{theorem}[Lipschitz extension theorem]  \label{t_Lip_extension}
Let $(X,d)$ be a metric space and $\beta \in (0,1)$. Then for every finite set $F \subseteq X$, every function $\varphi$ on $F$, and every number $C > 2$, there exists a function $\psi \in \Lip(X,d)$ with the properties $\psi|_F = \varphi$ and $\Hnorm{d^\beta}{\psi}{K} \leq C \Hnorm{d^\beta}{\varphi}{F}$.\footnote{Note that the difference in the definition of H\"older norm results in the difference in the lower bound on $C$.}
\end{theorem}

We write 
\begin{equation*}
\langle \phi, \mu \rangle \= \int\! \phi \, \mathrm{d} \mu 
\qquad \text{ for } \phi \in \CCC(X) \text{ and } \mu \in \MMM(X).
\end{equation*}

Next, we record the following strengthened version of \cite[Lemma~2]{BZ15}, which follows from Lemma~2 and its proof in \cite{BZ15}.

\begin{prop} \label{p_Wasserstein}
Let $T\: X \rightarrow X$ be a continuous map on a compact metric space $(X,d)$ and $\alpha \in (0,1]$. Let $\O$ be a periodic orbit of $T$. Write $\mu_\O \= \frac{1}{\card{\O}} \sum_{x\in\O} \delta_x \in \MMM(X,T)$. Then for each $\beta \in [\alpha,1]$ there exists a constant $C_{\O,\beta} \geq 1$ such that for every $\nu \in \MMM(X,T)$, we have
\begin{equation*} 
\sup\biggl\{ \frac{ \langle \phi, \nu-\mu \rangle}{ \Hseminorm{d^\beta }{\phi} } :  \phi \in \Lip (X, d^\beta ), \, \Hseminorm{d^\beta }{\phi} \neq 0 \biggr\}
\leq C_{\O,\beta} \langle d^\beta(\cdot, \O), \nu \rangle.
\end{equation*}
Moreover, $C_{\O,\beta}$ can be chosen as $C_{\O,\beta} \= ( \diam_d(X) / \lambda_\alpha )^\beta \leq \diam_d(X) / \lambda_\alpha$, where $\lambda_\alpha \in (0, \diam_d(X))$ is a constant satisfying the condition that for all $i \in \{0,\,1,\, \dots, \, \card\O -1 \}$ and $x,\,y\in X$ with $d(x,y) < \lambda_\alpha$, the following inequality holds:
\begin{equation*}
d \bigl( T^i(x), T^i(y) \bigr) < 2^{- 1/ \alpha } \min\{ d(w,z) : w,\, z \in \O, \, w \neq z \}.
\end{equation*}
\end{prop}

Note that the number $\lambda_\alpha$ in the statement above exists due to the uniform continuity of $T$.

We are now ready to prove Theorem~\ref{t_lock}.

\begin{proof}[Proof of Theorem~\ref{t_lock}]
By definition, the set $\lock(X,d^\alpha)$ is open in $\lip(X,d^\alpha)$ and is contained in $\sP(X)$. It suffices to show that $\lock(X,d^\alpha)$ is dense in $\sP(X) \cap \lip(X,d^\alpha)$.

Consider an arbitrary potential $\phi \in \sP(X) \cap \lip(X,d^\alpha)$. Let $\O$ be a periodic orbit of $T$ such that $\mu_\O \= \frac{1}{\card{\O}} \sum_{x\in\O} \delta_x$ is a maximizing measure for $\phi$.

Fix arbitrary $t , \, \delta, \, \epsilon \in (0, 1/5)$, and $\beta \in (\alpha, 1)$.

Consider arbitrary $\psi \in \lip(X,d^\alpha)$ with 
\begin{equation}  \label{e_Pf_p_Wasserstein_psi_norm}
\Hnorm{d^\alpha}{\psi}{ X } \leq \epsilon. 
\end{equation}
By the compactness of $X$ and the covering theorem (Theorem~\ref{t_covering_theorem}), there exists a finite set $F \subseteq X$ such that 
\begin{equation}  \label{e_Pf_p_Wasserstein_5delta_covering}
\bigcup_{x\in F} B_d( x, 5 \delta) = X
\end{equation}
and the balls $B_d(x, \delta)$, $x\in F$, are pairwise disjoint. By the Lipschitz extension theorem in Theorem~\ref{t_Lip_extension} (applied to $\beta$, $F$, and $\varphi \= \psi|_F$), there exists a function $\psi_\beta \in \Lip (X,d^\beta )$ such that 
\begin{equation}  \label{e_Pf_p_Wasserstein_psi_beta}
\psi_\beta|_F = \psi|_F \quad \text{and} \quad
\Hnorm{d^\beta}{\psi_\beta}{ X }  \leq 3 \Hnorm{d^\beta}{\psi|_F}{ F }
\end{equation}
Note that by (\ref{e_Pf_p_Wasserstein_psi_norm}) and (\ref{e_Pf_p_Wasserstein_5delta_covering}), we have
\begin{equation} \label{e_Pf_p_Wasserstein_psi_F_beta_seminorm}
 \Hnorm{d^\beta}{\psi|_F}{ F } 
 \leq \epsilon + \Hseminorm{d^\alpha}{\psi|_F} ( \min\{ d(x,y) : x, \, y \in F, \, x \neq y \} )^{\alpha - \beta} 
 \leq \epsilon + \epsilon \delta^{\alpha - \beta} .
\end{equation}

Define $\tau \= \psi - \psi_\beta \in \Lip (X,d^\alpha)$. Then by (\ref{e_Pf_p_Wasserstein_psi_beta}), (\ref{e_Pf_p_Wasserstein_psi_norm}), (\ref{e_Pf_p_Wasserstein_psi_F_beta_seminorm}), and the fact that $\delta \in (0,1/5)$, we have
\begin{equation}  \label{e_Pf_p_Wasserstein_tau_unif_norm}
\norm{\tau}_{\CCC^0(X)}  
\leq (5\delta)^\alpha \Hseminorm{d^\alpha}{\psi}   + (5\delta)^\beta \Hseminorm{d^\beta}{\psi_\beta}  
\leq \epsilon (5\delta)^\alpha + 3(\epsilon + \epsilon \delta^{\alpha - \beta} )  (5\delta)^\beta 
\leq 35 \epsilon \delta^\alpha.
\end{equation}

Set $\phi_t \= \phi - t d^\beta ( \cdot, \O)$. For every $\nu \in \MMM(X,T) \setminus \{ \mu_\O \}$, by (\ref{e_Pf_p_Wasserstein_tau_unif_norm}), the fact that $\mu_\O$ is a maximizing measure for $\phi$, Proposition~\ref{p_Wasserstein}, (\ref{e_Pf_p_Wasserstein_psi_beta}), (\ref{e_Pf_p_Wasserstein_psi_F_beta_seminorm}), and (\ref{e_Pf_p_Wasserstein_epsilon})  we have
\begin{align*}
\langle \phi_t + \psi, \nu \rangle
&=       \langle \phi , \nu \rangle - t \langle d^\beta ( \cdot, \O), \nu \rangle + \langle \psi_\beta, \nu \rangle +  \langle  \tau , \nu \rangle  \\
&\leq   \langle \phi , \mu_\O \rangle + \langle \psi_\beta, \nu \rangle - t \langle d^\beta ( \cdot, \O), \nu \rangle   + 35  \epsilon \delta^\alpha  \\
&\leq   \langle \phi , \mu_\O \rangle + \langle \psi_\beta, \mu_\O \rangle + \biggl( \frac{ \diam_d(X) }{ \lambda_\alpha } \Hseminorm{d^{\beta}}{ \psi_\beta } - t \biggr)  \langle d^\beta ( \cdot, \O), \nu \rangle  + 35  \epsilon \delta^\alpha  \\
&\leq  \langle \phi_t , \mu_\O \rangle + \langle \psi, \mu_\O \rangle -  \langle \tau, \mu_\O \rangle + \biggl(  \frac{ \diam_d(X) }{ \lambda_\alpha }  \Hseminorm{d^{\beta}}{ \psi_\beta }  - t \biggr)  \langle d^\beta ( \cdot, \O), \nu \rangle  + 35  \epsilon \delta^\alpha  \\
&\leq  \langle \phi_t + \psi , \mu_\O \rangle  +  \biggl(  3(\epsilon + \epsilon \delta^{\alpha - \beta} )  \frac{ \diam_d(X) }{ \lambda_\alpha }  - t \biggr)  \langle d^\beta ( \cdot, \O), \nu \rangle + 70  \epsilon \delta^\alpha,
\end{align*}
where $\lambda_\alpha  \in ( 0, \diam_d(X))$ is a constant depending only on $T$, $\alpha$, and $\O$ from Proposition~\ref{p_Wasserstein} (applied to $T$, $\alpha$, $\O$).

Since  $t , \, \delta, \, \epsilon \in (0, 1/5)$, and $\beta \in (\alpha, 1)$ are arbitrary, for each fixed $\delta$, we can choose $\beta$ sufficiently close to $\alpha$ such that
\begin{equation}  \label{e_Pf_p_Wasserstein_beta}
\delta^{\alpha - \beta} \leq 2.
\end{equation}
On the other hand, for each fixed $t$, we can choose
\begin{equation} \label{e_Pf_p_Wasserstein_epsilon}
\epsilon \= 10^{-1} t \lambda_\alpha / \diam_d(X).
\end{equation}
Hence, 
\begin{align*}
\langle \phi_t + \psi, \nu \rangle 
& \leq  \langle \phi_t + \psi , \mu_\O \rangle  +  \biggl(  9 \epsilon   \frac{ \diam_d(X) }{ \lambda_\alpha }  - t \biggr)  \langle d^\beta ( \cdot, \O), \nu \rangle + 70  \epsilon \delta^\alpha \\
& \leq  \langle \phi_t + \psi , \mu_\O \rangle  - 10^{-1} t  \langle d^\beta ( \cdot, \O), \nu \rangle + 70  \epsilon \delta^\alpha.
\end{align*}

Since $\delta \in (0,1/5)$ is arbitrary, we conclude that $\langle \phi_t + \psi, \nu \rangle <  \langle \phi_t + \psi , \mu_\O \rangle$.

Since $\nu \in \MMM(X,T)$ is arbitrary, the measure $\mu_\O$ is the unique maximizing measure for $\phi_t + \psi$ for all $\psi \in \lip(X,d^\alpha)$ with $\Hnorm{d^\alpha}{\psi}{ X } \leq \epsilon$. On the other hand, $d^\beta(\cdot, \O) \in \Lip ( X,d^\beta ) \subseteq \lip(X,d^\alpha)$ by Proposition~\ref{p_Lip_dense_lip_alpha}. Hence, $\phi_t = \phi - t d^\beta( \cdot, \O) \in \lock( X, d^\alpha)$. Recall that $t \in (0,1/5)$ is arbitrary. Therefore, $\lock(X,d^\alpha)$ is dense in $\sP(X) \cap \lip(X,d^\alpha)$.
\end{proof}

\section{The Assumptions}      \label{sct_Assumptions}
We state below the hypotheses under which we will develop our theory in most parts of this paper. We will repeatedly refer to such assumptions in the later sections. We emphasize again that not all assumptions are used in every statement in this paper.

\begin{assumptions}
\quad

\begin{enumerate}

\smallskip

\item $f\:S^2 \rightarrow S^2$ is an expanding Thurston map.

\smallskip

\item $\CC\subseteq S^2$ is a Jordan curve containing $\post f$ with the property that there exists $n_\CC\in\N$ such that $f^{n_\CC} (\CC)\subseteq \CC$ and $f^m(\CC)\nsubseteq \CC$ for each $m\in\{1, \, 2, \, \dots, \, n_\CC-1\}$.

\smallskip

\item $d$ is a visual metric on $S^2$ for $f$ with expansion factor $\Lambda > 1$  and a linear local connectivity constant $L\geq 1$.

\smallskip

\item $\alpha\in(0,1]$.

\smallskip

\item $\phi\in \Lip(S^2,d^{\alpha})$ is a real-valued $\alpha$-H\"{o}lder continuous function with respect to the visual metric $d$.



\end{enumerate}

\end{assumptions}

Observe that by Lemma~\ref{l_CexistsBM}, for each $f$ in~(1), there exists at least one Jordan curve $\CC$ that satisfies~(2). Since for a fixed $f$, the number $n_\CC$ is uniquely determined by $\CC$ in~(2), in the remaining part of the paper, we will say that a quantity depends on $\CC$ even if it also depends on $n_\CC$.

Recall that the expansion factor $\Lambda$ of a visual metric $d$ on $S^2$ for $f$ is uniquely determined by $d$ and $f$. We will say that a quantity depends on $f$ and $d$ if it depends on $\Lambda$.


Note it follows from Remark~\ref{r_Expanding} and Lemma~\ref{l_CellBoundsBM} that a metric $d$ on $S^2$ satisfies (3) if and only if $d$ is a visual metric for $f^n$ with expansion factor $\Lambda^n>1$ and a linear local connectivity constant $L\geq 1$ for some (or each) $n\in\N$. Even though the value of $L$ is not uniquely determined by the metric $d$, in the remainder of this paper, for each visual metric $d$ on $S^2$ for $f$, we will fix a choice of linear local connectivity constant $L$. We will say that a quantity depends on the visual metric $d$ without mentioning the dependence on $L$, even though if we had not fixed a choice of $L$, it would have depended on $L$ as well.

In the discussion below, depending on the conditions we will need, we will sometimes say ``Let $f$, $\CC$, $d$, $\phi$, $\alpha$ satisfy the Assumptions in Section~\ref{sct_Assumptions}.'', and sometimes say ``Let $f$ and $d$ satisfy the Assumptions in Section~\ref{sct_Assumptions}.'', etc.

\section{The Liv\v{s}ic theorem, the Ma\~{n}\'{e} lemma, and the bilateral Ma\~{n}\'{e} lemma}  \label{sct_BouschOp}
In this section, we give a definition of Bousch operators and discuss some of their basic properties for general dynamical systems before proving in Proposition~\ref{p_calibrated_sub-action_exists} the existence of an eigenfunction for the Bousch operator, also known as a calibrated sub-action, for an expanding Thurston map. Finally, we deduce the Liv\v{s}ic theorem, the Ma\~{n}\'{e} lemma, and the bilateral Ma\~{n}\'{e} lemma in our context.

Recall that a map $T \: X \rightarrow Y$ is \defn{finite-to-one} if $\card \bigl( T^{-1} (y) \bigr) < +\infty$ for all $y\in Y$.

Let $T\: X \rightarrow X$ be a finite-to-one surjective continuous map on a compact metric space $(X,d)$, and $\psi \: X \rightarrow \R$ a real-valued continuous function. Recall that $\R^X$ denotes the set of all functions from $X$ to $\R$. The \defn{Bousch operator} $\RR_\psi \: \R^X \rightarrow \R^X$ for $T$ and $\psi$ is given by
\begin{equation}   \label{e_Def_BouschOp}
\RR_\psi (u) (x) \= \max  \bigl\{ \psi (y) + u (y) : y \in T^{-1} (x)  \bigr\},
\end{equation}
for $u\in \R^X$ and $x\in X$. We define
\begin{equation}   \label{e_Def_overline_phi}
\overline\psi  \= \psi - Q (T,\psi),
\end{equation}
where $Q(T,\psi)$  is the maximal potential energy given by
\begin{equation}  \label{e_Def_beta}
Q(T,\psi)  \= \sup_{\mu \in \MMM(X,T)} \int \! \psi \, \mathrm{d} \mu = \max_{\mu \in \MMM(X,T)} \int \! \psi \, \mathrm{d} \mu .
\end{equation}
The last identity follows from the weak$^*$-compactness of $\MMM(X, T)$. We denote the (nonempty) \defn{set of $\psi$-maximizing measures} by
\begin{equation}  \label{e_Def_max_measures}
\Mmax (T, \, \psi) \= \bigg\{ \mu \in \MMM(X,T) : \int\! \psi \, \mathrm{d} \mu = Q (T,\psi) \biggr\}.
\end{equation}

\begin{lemma}  \label{l_Bousch_Op_iterate_max}
Let $T\: X \rightarrow X$ be a finite-to-one surjective continuous map on a compact metric space $(X,d)$. Fix some $c\in\R$, $n\in\N$, $\psi \in \CCC(X)$, $u \in \R^X$, and a set $\cA \subseteq \R^X$. Then the following statements are satisfied:
\begin{itemize}

\smallskip
\item[(i)] $\RR_\psi (u + c) = c + \RR_\psi (u)$.

\smallskip
\item[(ii)] $\RR_{\overline\psi}^n (u) (x) + n Q(T,\psi)  = \RR_\psi^n (u) (x) =  \max  \{ S_n \psi (y) + u (y) : y \in T^{-n} (x)  \}$ for each $x \in X$.

\smallskip
\item[(iii)] $\RR_\psi ( \sup \{  v (\cdot) : v \in \cA \} ) (x) = \sup \{ \RR_\psi (v)(x) : v \in \cA \}$ for each $x \in X$.

\smallskip
\item[(iv)] $\lim\limits_{i\to +\infty} \RR_\psi ( u_i ) (x) = \RR_\psi \bigl( \lim\limits_{i\to+\infty}  u_i (\cdot)  \bigr) (x)$ for each $x \in X$ and each pointwise convergent sequence $\{ u_i \}_{i\in\N}$ of functions in $\R^X$.
\end{itemize}
\end{lemma}

\begin{proof}
Statement~(i) follows immediately from (\ref{e_Def_BouschOp}).

The first identity in statement~(ii) follows immediately from (\ref{e_Def_overline_phi}). We use induction to establish the second identity. The case $n=1$ follows from (\ref{e_Def_BouschOp}). Assume that statement~(ii) is verified for some $n=m \in \N$. Then by (\ref{e_Def_BouschOp}),
\begin{align*}
\RR_\psi^{m+1} (u) (x) & =  \max_{y \in T^{-1}(x)} \bigl\{ \psi (y)  + \max_{ z \in T^{-m}(y) } \{ S_m \psi(z) + u (z)  \}  \bigr\} \\
                                      & =  \max_{y \in T^{-1}(x)} \bigl\{ \max_{ z \in T^{-m}(y) } \{ S_{m+1} \psi(z) + u (z)  \}  \bigr\} \\
                                      & =  \max_{ z \in T^{-(m+1)}(x) } \{ S_{m+1} \psi(z) + u (z)  \} .
\end{align*}

Next, statement~(iii) follows from the following simple observation:
\begin{align*}
\RR_\psi ( \sup \{ v(\cdot) : v \in \cA \} ) (x) & = \max_{y \in T^{-1}(x) } \{ \psi (y) + \sup \{ v(y) : v \in \cA \} \} \\
                                                  & = \max_{y \in T^{-1}(x) } \{ \sup \{ \psi (y) +  v(y)  :  v \in \cA \} \} \\
                                                  & = \sup \bigl\{ \max_{y \in T^{-1}(x) } \{ \psi (y) +  v(y)    : v\in\cA \} \bigr\} \\
                                                  & = \sup \{ \RR_\psi ( v )(x) : v \in \cA \}.
\end{align*}

Finally, we verify statement~(iv). Let $v \: X \rightarrow \R$ be the pointwise limit of $u_i$ as $i$ tends to $+\infty$. Fix arbitrary $x\in X$ and $\epsilon > 0$. Since $T$ is finite-to-one, we can find $N \in \N$ such that for each integer $n\geq N$ and each $y \in T^{-1} (x)$, $\abs{ u_n (y) - v (y) } < \epsilon$. Fix arbitrary integer $n \geq N$. We choose $z_1, \, z_2 \in T^{-1} (x)$ satisfying $\RR_\psi (u_n) (x) = \psi (z_1) + u_n (z_1)$ and $\RR_\psi (v) (x) = \psi (z_2) + v (z_2)$. Then by (\ref{e_Def_BouschOp}),
\begin{align*}
\RR_\psi ( u_n ) (x) - \RR_\psi (v) (x)
& \leq \psi (z_1) + u_n (z_1) - \psi (z_1) - v (z_1)
    =     u_n (z_1) - v(z_1) < \epsilon  \quad\text{and} \\
\RR_\psi ( u_n ) (x) - \RR_\psi (v) (x)
& \geq \psi (z_2) + u_n (z_2) - \psi (z_2) - v (z_2)  
   =     u_n (z_2) - v(z_2) >  - \epsilon  .
\end{align*}
Statement~(iv) now follows.
\end{proof}

The statements in the following lemma are well-known, see for example, \cite[Theorem~1]{Bou01}, \cite[Theorem~4.7]{Je06}, and \cite[Lemma~2.1]{Co16}. We include a proof for the convenience of the reader.

\begin{lemma}  \label{l_Bousch_Op_normalizing_potential}
Let $T\: X \rightarrow X$ be a finite-to-one surjective continuous map on a compact metric space $(X,d)$. Fix arbitrary continuous functions $\varphi,\, u \in \CCC(X)$. Then the following statements are satisfied:
\begin{itemize}

\smallskip
\item[(i)] $\Mmax (T, \, \varphi) = \Mmax (T, \, \varphi + c + u - u \circ T)$ for each constant $c\in\R$.

\smallskip
\item[(ii)] If $\RR_{\overline\varphi} (u) = u$, then the function $\wt\varphi \= \overline\varphi  + u - u \circ T$ satisfies the following properties:

\begin{itemize}

\smallskip
\item[(a)] $Q (T, \wt \varphi ) = \max \bigl\{ \int\! \wt\varphi \, \mathrm{d} \mu : \mu \in \MMM(X, T) \bigr\} = 0$,

\smallskip
\item[(b)] $\wt\varphi (x) \leq 0$ for each $x \in X$,

\smallskip
\item[(c)] the set $\cK \= \bigcap_{j=0}^{+\infty} T^{-j} \bigl( \wt\varphi^{-1} ( 0 ) \bigr)$ is a nonempty compact $T$-forward-invariant set, and

\smallskip
\item[(d)] $\Mmax (T, \, \varphi) = \Mmax (T, \, \wt\varphi ) = \{ \mu \in \MMM( X, T ) : \supp \mu \subseteq \cK \}$.

\end{itemize}


\end{itemize}
\end{lemma}

\begin{proof}
(i) Fix a constant $c\in\R$. Denote $\psi \= \varphi + c + u - u \circ T$. For each $\mu \in \MMM( X, T )$, we have $\int \! \psi \,\mathrm{d}\mu = c + \int \! \varphi \,\mathrm{d} \mu$. Thus, $Q (T, \psi ) = c + Q (T,\varphi )$ (see (\ref{e_Def_beta})). Consequently, by (\ref{e_Def_max_measures}), for each $\mu \in \Mmax (T,\, \varphi)$,
\begin{equation*}
\int \! \psi \,\mathrm{d}\mu = c + \int \! \varphi \,\mathrm{d} \mu = c + Q (T,\varphi) = Q (T,\psi),
\end{equation*}
i.e., $\mu \in \Mmax(T, \, \psi)$. Similarly, $\Mmax(T, \, \psi) \subseteq \Mmax (T, \, \varphi)$. Statement~(i) follows.

\smallskip

(ii) Assume $u \in \CCC(X)$ satisfies $\RR_{\overline\varphi} (u) = u$.

\smallskip

(a) Since $\mu$ is $T$-invariant, 
$Q (T, \wt\varphi ) = \max\limits_{ \mu \in \MMM( X , T ) } \int \! ( \overline\varphi + u - u \circ T) \,\mathrm{d} \mu
                                         =  \max\limits_{ \mu \in \MMM( X, T ) } \int \!  \overline\varphi   \,\mathrm{d} \mu
                                         = - Q (T,\varphi) + Q (T,\varphi) = 0$.
                                         
\smallskip

(b) Since $\RR_{\overline\varphi} (u) = u$, we get from (\ref{e_Def_BouschOp}) that for each $x \in X$, $u(T(x)) = \RR_{\overline\varphi} (u) (T(x)) \geq  \overline\varphi(x) + u(x)$. Thus, by the definition of $\wt \varphi$, we have $\wt \varphi (x) \leq 0$ for each $x\in X$.

\smallskip

(c) By the definition of $\cK$, it follows immediately from the continuity of $T$ and $\wt\varphi$ that $\cK$ is compact. By the definition of $\cK$, it is also clear that $\cK$ is $T$-forward invariant. The fact that $\cK$ is nonempty will follow directly from statement~(ii)(d) below and the fact that $\Mmax(T,\varphi)$ is nonempty.

\smallskip

(d) The first identity follows from statement~(i). To establish the second identity, we first note that by (\ref{e_Def_max_measures}) and statements~(ii)(a) and~(b), every $\mu \in \MMM(X, T)$ with $\supp \mu \subseteq \cK \subseteq \wt\varphi^{-1} (0)$ is in $\Mmax  ( T, \, \wt\varphi  )$. Conversely, by statement~(ii)(a), every $\mu \in \Mmax ( T, \, \wt\varphi )$ satisfies $\int\! \wt\varphi \, \mathrm{d} \mu = 0$. By statement~(ii)(b), $\supp \mu$ is a subset of the compact set $\wt\varphi^{-1} (0)$. It now follows from the $T$-invariance of $\mu$ that $\supp \mu \subseteq \bigcap_{j=0}^{+\infty} T^{-j}  \bigl( \wt\varphi^{-1} ( 0 ) \bigr) = \cK$.
\end{proof}

\begin{lemma}  \label{l_Bousch_Op_preserve_space}
Let $f$, $\CC$, $d$, $L$, $\alpha$, $\phi$ satisfy the Assumptions in Section~\ref{sct_Assumptions}. Then there exist a constant $C_1 > 1$ depending only on $f$, $d$, $\CC$, and $\alpha$ such that for each $u \in \Lip (S^2,d^{\alpha})$ and each $n\in\N$, we have $\RR_\phi^n (u) \in \Lip (S^2, d^{\alpha})$ and 
\begin{equation}  \label{e_Bousch_Op_Holder_seminorm_bound}
 \Hseminormbig{d^\alpha }{ \RR_\phi^n (u) } 
 \leq C_1 \bigl( \Hseminorm{d^\alpha }{ \phi}  +  \Hseminorm{d^\alpha }{u} \big).
\end{equation}
\end{lemma}

\begin{proof}
Fix an arbitrary function $u\in \Lip( S^2,d^{\alpha})$ and $n\in\N$. Let $X^0$ be either the black $0$-tile $X^0_\b$ or the white $0$-tile $X^0_\w$ in $\X^0$. For each $X^n \in \X^n$ with $f^n(X^n) = X^0$, by Proposition~\ref{p_CellDecomp}~(i), $(f^n)|_{X^n}$ is a homeomorphism of $X^n$ onto $X^0$. So for $x,\,y \in X^0$, there exist unique points $x',\,y' \in X^n$ with $x' \in f^{-n} (x)$ and $y' \in f^{-n} (y)$. Recall Lemma~\ref{l_Bousch_Op_iterate_max}~(ii). By assuming that $x'$ is a preimage of $x$ under $f^n$ with the property that 
\begin{equation*}
\RR_\phi^n (u ) (x) = S_n \phi (x') + u(x'),
\end{equation*}
we get from Lemma~\ref{l_Bousch_Op_iterate_max}~(ii), Lemma~\ref{l_key_estimate}, and Lemma~\ref{l_MetricDistortion} that
\begin{align*}
\RR_\phi^n (u) (x) - \RR_\phi^n (u) (y) 
& \leq S_n \phi (x') + u (x') - S_n \phi (y') - u (y') \\
& \leq C_0 ( 1 - \Lambda ^{-\alpha} )^{-1} \Hseminorm{d^\alpha }{\phi}  d(x,y)^\alpha + \Hseminorm{d^\alpha }{u}  d(x',y')^\alpha \\
& \leq C_0 ( 1 - \Lambda ^{-\alpha} )^{-1} \Hseminorm{d^\alpha }{\phi}  d(x,y)^\alpha + \Hseminorm{d^\alpha }{u}  C_0^{\alpha} \Lambda^{- n \alpha}  d(x,y)^\alpha \\
& \leq 2 C_0 ( 1 - \Lambda ^{-\alpha} )^{-1} \bigl(  \Hseminorm{d^\alpha }{\phi} + \Hseminorm{d^\alpha }{u} \bigr)  d(x,y)^\alpha,
\end{align*}
where $\Lambda > 1$ is the expansion factor of $d$ under $f$, and $C_0 > 1$ is a constant depending only on $f$, $\CC$, and $d$ from Lemma~\ref{l_MetricDistortion}. Similarly, by assuming that $y'$ is a preimage of $y$ under $f^n$ with the property that
$
\RR_\phi^n (u ) (y) =  S_n \psi (y') + u(y')$,
we get 
$\RR_\phi^n (u) (x) - \RR_\phi^n (u) (y)  \geq -  \bigl( \Hseminorm{d^\alpha }{ S_n \phi } + \Hseminorm{d^\alpha }{u} \bigr) C_0^{\alpha} \Lambda^{ - n \alpha }  d(x,y)^\alpha$. Hence, 
\begin{equation} \label{e_Pf_l_Bousch_Op_preserve_space_1}
\Absbig{ \RR_\phi^n (u) (x) - \RR_\phi^n (u) (y) } 
\leq 2 C_0 ( 1 - \Lambda ^{-\alpha} )^{-1}  \bigl( \Hseminorm{d^\alpha }{ \phi } + \Hseminorm{d^\alpha }{u} \bigr) d(x,y)^\alpha.
\end{equation}

Next, we consider arbitrary $x \in X^0_\w$ and $y \in X^0_\b$. Since the metric space $(S^2, d)$ is linearly locally connected with a linear local connectivity constant $L \geq 1$, there exists a continuum $E \subseteq S^2$ with $x,\,y \in E$ and $E \subseteq B_d ( x, L d(x,y) )$ (see Subsection~\ref{subsct_ThurstonMap}). We can then fix a point $z \in \CC \cap E$. Then $x,\, z \in X^0_\w$ and $y,\, z \in X^0_\b$. Thus, we have
\begin{align}    \label{e_Pf_l_Bousch_Op_preserve_space_2}
\Absbig{ \RR_\phi^n (u) (x) - \RR_\phi^n (u) (y) }
& \leq \Absbig{ \RR_\phi^n (u) (x) - \RR_\phi^n (u) (z) }  +  \Absbig{ \RR_\phi^n (u) (y) - \RR_\phi^n (u) (z) }   \notag  \\
& \leq ( d(x, z)^\alpha  +  d(y,z)^\alpha )  2 C_0 ( 1 - \Lambda ^{-\alpha} )^{-1}  \bigl( \Hseminorm{d^\alpha }{ \phi } + \Hseminorm{d^\alpha }{u} \bigr) \\
& \leq 2 ( \diam_d(E) )^\alpha 2 C_0 ( 1 - \Lambda ^{-\alpha} )^{-1}  \bigl( \Hseminorm{d^\alpha }{ \phi } + \Hseminorm{d^\alpha }{u} \bigr)   \notag \\
& \leq 8 L^\alpha  C_0 ( 1 - \Lambda ^{-\alpha} )^{-1} \bigl( \Hseminorm{d^\alpha }{ \phi } + \Hseminorm{d^\alpha }{u} \bigr) d(x,y)^\alpha \notag .
\end{align}  

Finally $\RR_\phi^n (u) \in \Lip (S^2, d^{\alpha})$ follows from (\ref{e_Pf_l_Bousch_Op_preserve_space_1}) and (\ref{e_Pf_l_Bousch_Op_preserve_space_2}). On the other hand, by choosing $C_1 \= 8 L^\alpha C_0 ( 1 - \Lambda ^{-\alpha} )^{-1} >1$, we get (\ref{e_Bousch_Op_Holder_seminorm_bound}).
\end{proof}

Inspired by the construction of eigenfunctions of the Ruelle--Perron--Frobenius operators (c.f.\ \cite[Theorem~5.16]{Li18}), we find a fixed point $u_\phi$ of the Bousch operator $\RR_{\overline\phi}$ (also known as a \defn{calibrated sub-action} for $f$ and $\phi$).

\begin{prop}  \label{p_calibrated_sub-action_exists}
Let $f$, $\CC$, $d$, $\alpha$, $\phi$ satisfy the Assumptions in Section~\ref{sct_Assumptions}. Then the function $u_\phi \: S^2 \rightarrow \R$ given by
\begin{equation}  \label{e_calibrated_sub-action_exists}
u_\phi (x) \= \limsup_{n\to+\infty} \RR_{\overline\phi}^n (\mathbbold{0}) (x), \qquad x \in S^2,
\end{equation}
satisfies the following properties:
\begin{enumerate}
\smallskip
\item[(i)] $\abs{ u_\phi (x) } \leq 2  C_1 \Hseminorm{d^\alpha }{ \phi } \diam_d(S^2)^\alpha$ for each $x \in S^2$,

\smallskip
\item[(ii)] $u_\phi \in \Lip (S^2,d^{\alpha})$ with $\Hseminorm{d^\alpha }{u_\phi} \leq C_1 \Hseminorm{d^\alpha }{\phi}$,

\smallskip
\item[(iii)] $\RR_{\overline\phi} (u_\phi) = u_\phi$.
\end{enumerate}
Here $C_1 >1$ is a constant depending only on $f$, $d$, $\CC$, and $\alpha$ from Lemma~\ref{l_Bousch_Op_preserve_space}.
\end{prop}

\begin{proof}
We write $D \=C_1 \Hseminorm{d^\alpha }{ \phi } \diam_d(S^2)^\alpha$ in this proof.

To establish statement~(i), we fix a point $x\in S^2$ and a measure $\mu \in \Mmax( f , \, \phi )$. Recall from Lemma~\ref{l_Bousch_Op_iterate_max}~(ii), for each $n\in\N$,
\begin{equation}  \label{e_Pf_p_calibrated_sub-action_exists_Ln0}
\RR_{\overline\phi}^n (\mathbbold{0}) (x) =  \max \bigl\{ S_n \overline\phi (y) : y \in f^{-n} (x) \bigr\}.
\end{equation}
To show that $u_\phi(x) \neq - \infty$, we choose, for each $n\in\N$, a point $y_n \in S^2$ on which $S_n \overline\phi$ attains its maximum value. Then for each $n\in \N$, by Lemma~\ref{l_Bousch_Op_preserve_space}, (\ref{e_Def_overline_phi}), and (\ref{e_Pf_p_calibrated_sub-action_exists_Ln0}),
\begin{align*}
&\sup \bigl\{ \RR_{\overline\phi}^m( \mathbbold{0} ) (x) : m\in\N , \, m\geq n \bigr\} \\
&\qquad \geq  \RR_{\overline\phi}^n( \mathbbold{0} ) (x) \\
&\qquad \geq  \RR_{\overline\phi}^n( \mathbbold{0} ) ( f^n (y_n) ) - \Hseminormbig{d^\alpha}{\RR_{\overline\phi}^n( \mathbbold{0} )} d(x, f^n(y_n) )^\alpha \\
&\qquad \geq \max \bigl\{ S_n \overline\phi (z) : z \in f^{-n} ( f^n(y_n) ) \bigr\} - C_1 \Hseminorm{d^\alpha }{ \phi } \diam_d(S^2)^\alpha \\
&\qquad \geq S_n \overline\phi (y_n) - D \\
&\qquad \geq \int \! S_n \overline\phi \,\mathrm{d}\mu  - D \\
&\qquad = \int \! S_n  \phi \,\mathrm{d}\mu - n Q(f, \phi ) - D \\
&\qquad =  - D 
\end{align*}
Hence, $u_\phi (x) = \limsup_{n\to+\infty} \RR_{\overline\phi}^n ( \mathbbold{0} ) (x) \geq - D$.

Next, we will show that $u_\phi (x) \leq 2D$. If $\phi$ is a constant, then by (\ref{e_Pf_p_calibrated_sub-action_exists_Ln0}) and (\ref{e_calibrated_sub-action_exists}), we get that $u_\phi = \mathbbold{0}$ and statement~(i) holds. On the other hand, if $\phi$ is not a constant, then $D > 0$. 

We will establish the following claim.

\smallskip
\emph{Claim.} There exists a point $x_0 \in S^2$ such that $u_\phi (x_0 ) \leq 0$.

\smallskip

Assuming the claim and that $\phi$ is not a constant, by (\ref{e_calibrated_sub-action_exists}) we can fix an integer $N \in \N$ with the property that for each integer $n \geq N$, $\sup \bigl\{ \RR_{\overline\phi}^m( \mathbbold{0} ) (x_0) : m\in \N , \, m\geq n \bigr\} \leq D$. Then by Lemma~\ref{l_Bousch_Op_preserve_space},
\begin{align*}
&\sup \bigl\{ \RR_{\overline\phi}^m( \mathbbold{0} ) (x) : m\in \N , \, m\geq n \bigr\} \\
&\qquad \leq  \sup \bigl\{ \RR_{\overline\phi}^m( \mathbbold{0} ) (x_0) + \Hseminormbig{d^\alpha }{\RR_{\overline\phi}^m( \mathbbold{0} )} d(x, x_0 )^\alpha : m\in \N , \, m\geq n \bigr\}  \\
&\qquad \leq  \sup \bigl\{ \RR_{\overline\phi}^m( \mathbbold{0} ) (x_0) + C_1 \Hseminorm{d^\alpha }{\phi} \diam_d( S^2 )^\alpha : m\in \N , \, m\geq n \bigr\} \\  
&\qquad \leq  \sup \bigl\{ \RR_{\overline\phi}^m( \mathbbold{0} ) (x_0) : m\in \N , \, m\geq n \bigr\} + D \\
&\qquad \leq 2D.
\end{align*}
Thus, $u_\phi (x) = \limsup_{n\to+\infty} \RR_{\overline\phi}^n ( \mathbbold{0} ) (x) \leq 2D$. Since $x\in S^2$ is arbitrary, statement~(i) is verified.

\smallskip
To establish the claim above, we argue by contradiction and suppose that $u_\phi(w) > 0$ for all $w\in S^2$. Thus, for each $w \in S^2$, there exists an integer $n_w \in \N$ with the property that
\begin{equation*}
 \max \bigl\{ S_{n_w} \overline\phi (z) : z \in f^{-n_w} (w) \bigr\} \geq u_\phi (w) / 2.
\end{equation*}
It follows from the continuity of $\phi$ and $f$, Proposition~\ref{p_CellDecomp}~(i), and Lemma~\ref{l_CellBoundsBM}~(ii) that for each $w\in S^2$, there exists a number $\delta_w > 0$ such that for each $y \in B_d(w, \delta_w)$,
\begin{equation*}
 \max \bigl\{ S_{n_w} \overline\phi (z) : z \in f^{-n_w} (y) \bigr\} \geq u_\phi (w) / 3.
\end{equation*}
By compactness, we choose finitely-many points $x_1,\,x_2,\,\cdots,\,x_k \in S^2$ such that $\bigcup_{i=1}^k B(x_i, \delta_{x_i} ) = S^2$. Denote
\begin{equation}   \label{e_Pf_p_calibrated_sub-action_exists_c}
c \= \min\{ u(x_i) / (3 n_{x_i} ) : i \in \{ 1, \, 2 , \, \cdots, \, k \} \} > 0.
\end{equation}
For each $y \in S^2 \setminus \{ x_1, \, x_2, \, \cdots, \, x_k \}$, set $n_y \= n_{x_i}$, where $x_i$ is an arbitrarily chosen point from $\{x_1, \, x_2, \, \cdots, \, x_k\}$ satisfying $y \in B_d(x_i, \delta_{x_i} )$. Then for each $y\in S^2$, 
\begin{equation} \label{e_Pf_p_calibrated_sub-action_exists_ny}
 \max \bigl\{ S_{n_y} \overline\phi (z) : z \in f^{-n_y} (y) \bigr\} \geq n_y c .
\end{equation}

Fix an arbitrary $z_0 \in S^2$. By (\ref{e_Pf_p_calibrated_sub-action_exists_ny}), we can recursively choose $z_i \in S^2$, $i\in\N$, with the following two properties: (a) $z_i \in f^{ - n_{z_{i-1}} } ( z_{i-1} )$ and (b) $S_{n_{z_{i-1}} } \overline\phi (z_i) \geq n_{z_{i-1}} c$.

On the other hand, consider a sequence $\{\nu_i\}_{i\in\N}$ of probability measures given by 
\begin{equation*}
\nu_i \= \frac{1}{m_i} \sum_{j=0}^{m_i - 1} \delta_{f^j ( z_i )},
\end{equation*}
where $m_i \= n_{z_0} + n_{z_1} +\cdots + n_{z_{i-1}}$ and $\delta_{f^j ( z_i )}$ is the Dirac delta measure at $f^j ( z_i )$. By Alaoglu's theorem, there exists a subsequence $\nu_{i_1},\,\nu_{i_2},\,\cdots,\,\nu_{i_l},\,\cdots$ of $\{ \nu_i\}_{i\in\N}$ that converges in the weak$^*$ topology to a probability measure $\nu \in \PPP(S^2)$. We then deduce the $f$-invariance of $\nu$  from the observation that for each continuous function $w \in \CCC(S^2)$ and each $i\in\N$,
\begin{equation*}
\Absbigg{ \int\! w \,\mathrm{d} \nu_i  - \int\! w \,\mathrm{d} f_*(\nu_i) }  \leq \frac{2}{m_i} \norm{w}_{\CCC^0}
\end{equation*}
Hence, by (\ref{e_Def_overline_phi}), (\ref{e_Def_beta}), property~(b) above, and (\ref{e_Pf_p_calibrated_sub-action_exists_c}),
\begin{align*}
0 & \geq \int\!  \overline\phi \,\mathrm{d} \nu 
   =     \lim_{l\to+\infty} \frac{1}{ m_{i_l} } \sum_{j = 0}^{m_{i_l} - 1} \overline\phi \bigl( f^j ( z_{i_l}  ) \bigr) 
   =     \lim_{l\to+\infty} \frac{1}{ m_{i_l} } \sum_{j = 0}^{i_l - 1} S_{n_{z_j}} \overline\phi (  z_{j+1}  ) \\
& \geq \lim_{l\to+\infty} \frac{1}{ m_{i_l} } \sum_{j = 0}^{i_l - 1}  n_{z_j} c
 =  c > 0.
\end{align*}
This is a contradiction. The claim is, therefore, established.

\smallskip

Next, we verify statement~(ii). Fix an arbitrary pair of distinct points $x, \, y \in S^2$ and an arbitrary number $\epsilon>0$. By (\ref{e_calibrated_sub-action_exists}), we can find an integer $N\in\N$ such that the following inequalities hold:
\begin{align*}
\Absbig{ \RR_{\overline\phi}^N ( \mathbbold{0} ) (x) - u_\phi (x) } & < \epsilon \qquad \text{and} \\
\sup\bigl\{ \RR_{\overline\phi}^n ( \mathbbold{0} ) (y) : n\in\N, \, n \geq N \bigr\} - u_\phi(y)& < \epsilon.
\end{align*}
Then by Lemma~\ref{l_Bousch_Op_preserve_space},
\begin{align*}
u_\phi (x) - u_\phi (y) 
&\leq \RR_{\overline\phi}^N ( \mathbbold{0} ) (x) + \epsilon + \epsilon - \sup\bigl\{ \RR_{\overline\phi}^n ( \mathbbold{0} ) (y) : n\in\N, \, n \geq N \bigr\}  \\
&\leq \RR_{\overline\phi}^N ( \mathbbold{0} ) (x) - \RR_{\overline\phi}^N ( \mathbbold{0} ) (y) + 2\epsilon   \\
&\leq  \Hseminormbig{d^\alpha }{\RR_{\overline\phi}^N( \mathbbold{0} )} d(x, y )^\alpha + 2\epsilon   \\
&\leq C_1 \Hseminorm{d^\alpha }{\phi} d(x, y )^\alpha + 2\epsilon.
\end{align*}
Similarly, we can find an integer $M\in\N$ such that the following inequalities hold:
\begin{align*}
\Absbig{ \RR_{\overline\phi}^M ( \mathbbold{0} ) (y) - u_\phi (y) } & < \epsilon \qquad \text{and} \\
\sup\bigl\{ \RR_{\overline\phi}^m ( \mathbbold{0} ) (x) : m\in\N, \, m \geq M \bigr\} - u_\phi(x)& < \epsilon.
\end{align*}
Then by Lemma~\ref{l_Bousch_Op_preserve_space},
\begin{align*}
u_\phi (x) - u_\phi (y) 
&\geq \sup\bigl\{ \RR_{\overline\phi}^m ( \mathbbold{0} ) (x) : m\in\N, \, m \geq M \bigr\}  - \epsilon - \RR_{\overline\phi}^M ( \mathbbold{0} ) (y) - \epsilon   \\
&\geq \RR_{\overline\phi}^M ( \mathbbold{0} ) (x) - \RR_{\overline\phi}^M ( \mathbbold{0} ) (y) - 2\epsilon   \\
&\geq - \Hseminormbig{d^\alpha }{\RR_{\overline\phi}^M( \mathbbold{0} )} d(x, y )^\alpha - 2\epsilon   \\
&\geq - C_1 \Hseminorm{d^\alpha }{\phi} d(x, y )^\alpha - 2\epsilon.
\end{align*}
Since $\epsilon>0$ is arbitrary, statement~(ii) now follows.

\smallskip

Finally, by Lemma~\ref{l_Bousch_Op_iterate_max}, for each $x\in S^2$,
\begin{align*}
\RR_{\overline\phi} (u_\phi) (x) 
& = \RR_{\overline\phi} \bigl( \lim_{n\to+\infty} \sup\bigl\{ \RR_{\overline\phi}^m ( \mathbbold{0} )(x) : m\in\N, \, m\geq n \bigr\} \bigr) \\
& =  \lim_{n\to+\infty} \sup\bigl\{ \RR_{\overline\phi}^{m+1} ( \mathbbold{0} )(x) : m\in\N, \, m\geq n \bigr\}  \\                      
& = u_\phi (x).
\end{align*}
Statement~(iii) is, therefore, verified.
\end{proof}

We now establish the Ma\~{n}\'{e} and bilateral Ma\~{n}\'{e} lemmas for expanding Thurston maps.

\begin{proof}[Proof of Theorem~\ref{t_mane}]
We observe that it suffices to investigate the case when $f$ is an expanding Thurston map on the $S^2$ equipped with a visual metric $d$ since the other case when $f$ is a postcritically-finite rational map with no periodic critical points on the Riemann sphere equipped with the chordal or spherical metric follows from the former case and the bi-H\"olderness of a quasisymmetry (see Theorem~\ref{t_BM} and Remark~\ref{r_ChordalVisualQSEquiv}). So from now on, we consider $f$ in the former case.

We first verify the Ma\~{n}\'{e} lemma and then use it to deduce the bilateral Ma\~{n}\'{e} lemma using a similar argument as in \cite{Bou02}. 

Assume that $\phi \in \Lip(S^2,d^\alpha)$ for some $\alpha \in (0,1]$. Then it follows immediately from Proposition~\ref{p_calibrated_sub-action_exists}, (\ref{e_Def_overline_phi}), and Lemma~\ref{l_Bousch_Op_normalizing_potential}~(ii) that exists a function $w \in \Lip(S^2,d^\alpha)$ such that $\phi (x) - w(x) + ( w \circ f )(x) \leq Q (f,\phi)$ for all $x\in X$. The Ma\~{n}\'{e} lemma follows.

We now proceed to verify the bilateral Ma\~{n}\'{e} lemma. By the Ma\~{n}\'{e} lemma applied to $-\phi$ and to $\phi$, there exist $w_1,\, w_2 \in \Lip(S^2,d^\alpha)$ such that 
\begin{align*}
-\phi (x) - w_1(x) + ( w_1 \circ f )(x) & \leq Q (f,-\phi) \qquad \text{and} \\
\phi (x) - w_2(x) + ( w_2 \circ f )(x) & \leq Q (f,\phi)
\end{align*}
for all $x\in X$. We set $\psi \= \phi - w_1 + w_1 \circ f$ and $w_3 \= w_1 + w+2$. Then it is easy to check that $Q(f,\psi) = Q(f,\phi)$ and $Q(f,-\psi) = Q(f,- \phi)$. Thus, we have that $-Q(f,-\psi) \leq \psi (x)$ and $\psi(x) - w_3(x) + ( w_3 \circ f) (x) \leq Q(f,\psi)$ for all $x\in S^2$. By \cite[Theorem~1]{Bou02}, this guarantees that there exists a function $v \in \Lip(S^2,d^\alpha)$ such that $-Q(f,-\psi) \leq \psi(x) - v(x) + (v \circ f) (x) \leq Q(f,\psi)$ for all $x\in S^2$. Finally, by setting $u\= v + w_1 \in \Lip(S^2,d^\alpha)$, we get that 
\begin{equation*}
-Q(f,-\phi) = -Q(f,-\psi) \leq \phi(x) - u(x) + (u \circ f) (x) \leq Q(f,\psi) = Q(f,\phi)
\end{equation*}
for all $x\in S^2$, establishing the 
bilateral Ma\~{n}\'{e} lemma.
\end{proof}

To see the Liv\v{s}ic theorem for expanding Thurston maps as a consequence of the bilateral Ma\~{n}\'{e} lemma for these maps, we first need to verify the following lemma.

\begin{lemma}   \label{l_periodic_measure_dense}
Let $f \: S^2 \rightarrow S^2$ be an expanding Thurston map. Then the set 
\begin{equation*}
\{ \mu \in \MMM(S^2, f) : \mu \text{ is supported on a periodic orbit of } f\}
\end{equation*}
is dense in $\MMM(S^2, f)$ (in the weak$^*$ topology).
\end{lemma}

\begin{proof}
It follows from a theorem of K.~Sigmund (\cite[Theorem~1]{Sig74}) that it suffices to verify that $f$ has the specification property in the sense of K.~Sigmund (see the definition in \cite[Section~2]{Sig74}). Such specification property is closed under factors (\cite[Proposition~1]{Sig74}), and the full shift on the space of bi-infinite sequences of finitely-many symbols has this property (see \cite[Section~2]{Sig74}). Thus, it suffices to see that $f$ is a factor of such a full shift. By Theorem~9.1 in \cite{BM17}, $f$ is a factor of the full shift on the space of one-sided infinite sequences of $\deg f$ symbols, which in turn is a factor of the full shift on the space of bi-infinite sequences. The lemma now follows.
\end{proof}

\begin{proof}[Proof of Theorem~\ref{t_livsic}]
It is trivially true that (ii) implies (i). To see that (i) implies (ii), we assume that $S_n \phi (x) = 0$ for each periodic point $x \in S^2$ of period $n \in \N$. Then by Lemma~\ref{l_periodic_measure_dense} and (\ref{e_Def_beta}), we get $Q(f,\phi) = 0$. Statement~(ii) follows now from the bilateral Ma\~{n}\'{e} lemma (Theorem~\ref{t_mane}~(ii)).
\end{proof}

\begin{rem}   \label{r_livsic}
Note that as can be seen in the proofs of Theorems~\ref{t_mane} and~\ref{t_livsic}, if $\phi \in \Lip( X, d^\alpha)$ for some visual metric $d$ and exponent $\alpha \in (0,1]$, then the corresponding function $u$ can be chosen from $\Lip( X, d^\alpha)$.
\end{rem}

\section{Uniform local expansion away from critical points}  \label{sct_ULEAFC}
In this section, we formulate and establish in Lemma~\ref{l_Loc_Inject_Expansion_Away_from_Crit} the \emph{uniform local expansion} property of expanding Thurston maps away from critical points, which is crucial in the quantitative analysis in Sections~\ref{sct_Closing_Lemmas} and~\ref{sct_Holder_Density}. The proof relies on an interplay between the combinatorial objects and the visual metrics. Instead of tiles or flowers, we first link the dynamics and metric geometry using ``quasi-round'' bouquets.

Let $f \: S^2 \rightarrow S^2$ be an expanding Thurston map and $\CC\subseteq S^2$ be a Jordan curve with $\post f \subseteq \CC$. Then it follows from Remark~\ref{r_Flower} and Proposition~\ref{p_CellDecomp} that flowers iterate nicely under $f$, or more precisely, 
\begin{equation}  \label{e_Flower_iterate}
 f ( W^n (x) ) = W^{n-1} ( f(x) )
\end{equation}
for each $n\in\N$ and each $x \in \V^n(f,\CC)$. Compared with flowers, bouquets $U^n(x)$ defined in (\ref{e_Def_U^n}) serve a better role in linking the combinatorial structures induced by $f$ and $\CC$ to the geometry of visual metrics. We therefore establish a similar result to (\ref{e_Flower_iterate}) for $U^n(x)$ in the following lemma.

\begin{lemma}  \label{l_Un}
Let $f \: S^2 \rightarrow S^2$ be an expanding Thurston map and $\CC\subseteq S^2$ be a Jordan curve containing $\post f$. For all $x\in S^2$ and $n\in \Z$, we have
\begin{equation}  \label{e_l_Un}
f(U^n(x)) = U^{n-1} (f(x)).
\end{equation}
Here $U^m (x)$ is defined in (\ref{e_Def_U^n}) using $m$-tiles in the cell decompositions $\DD^m(f,\CC)$.
\end{lemma}

\begin{proof}
Fix arbitrary $x \in S^2$ and $n\in\Z$. If $n<0$, then $U^n(x) = S^2$ by definition. It is also clear that $U^0(x) = S^2$. Thus, we can assume, without loss of generality, that $n \geq 1$.

It follows quickly from (\ref{e_Def_U^n}), Proposition~\ref{p_CellDecomp}~(i), Definition~\ref{d_cellular}, and Definition~\ref{d_celldecomp} that $f( U^n(x)) \subseteq U^{n-1}(x)$. So it suffices to show that $f( U^n(x)) \supseteq U^{n-1}(x)$.

We define, for all $y\in S^2$ and $m\in\N_0$, 
\begin{equation}  \label{e_Pf_l_Un_Yn}
Y^m(y) \= \bigcup \{ X^m \in \X^m : y \in X^m \}.
\end{equation}

\smallskip

We claim that $Y^{n-1} (f(x))  = \bigcup \{ f(X^n) : X^n \in \X^n, \, x \in X^n \} = f ( Y^n(x) )$.

\smallskip

Indeed, we establish the claim by observing that it suffices to verify the first identity by discussing the following three cases.

\smallskip

\emph{Case~1.} $x \in \inte(X^n)$ for some $n$-tile $X^n \in \X^n$. Then the first identity in this case follows from Proposition~\ref{p_CellDecomp}~(i) and Definition~\ref{d_celldecomp}~(iii).

\smallskip

\emph{Case~2.} $x \in \inte(e^n)$ for some $n$-edge $e^n \in \E^n$. Then the first identity in this case follows from Proposition~\ref{p_CellDecomp}~(i) and Definition~\ref{d_celldecomp}~(iii).

\smallskip

\emph{Case~3.} $x \in \V^n$, i.e., $x$ is some $n$-vertex. Then by Remark~\ref{r_Flower}, $Y^n(x) = \overline{W}^n(x)$, where $\overline{W}^n(x)$ is the closure of the $n$-flower $W^n(x)$ defined in (\ref{e_Def_Flower}). It also follows from Remark~\ref{r_Flower} that $Y^{n-1} ( f(x) ) = \overline{W}^{n-1} ( f(x) ) = \bigcup \{ f(X^n) : X^n \in \X^n, \, x \in X^n \}$.

\smallskip

The claim is now established. 

Note that it follows from (\ref{e_Def_U^n}) that 
\begin{equation}    \label{e_Pf_l_Un_UY}
U^m (x) = \{ X^m \in \X^m : X^m \cap Y^m(x) \neq \emptyset \}
\end{equation}
for each $m\in \N_0$. So for each $y \in Y^n(x)$, we know that $y$ is in the (topological) interior $\inter(U^n(x))$ of $U^n(x)$. Since $f$ is a branched covering map, it is open, i.e., it sends open sets to open sets (see for example, \cite[Appendix~A.6]{BM17} and \cite[Lemma~2.1.2]{HP09}). It follows that $f(y) \in \inter( f( U^n(x) ) )$. Combined with the claim above, we get
\begin{equation*}
Y^{n-1} ( f(x) ) = f ( Y^n(x) ) \subseteq \inter ( f ( U^n(x) ) ).
\end{equation*}

Thus, for each $X^{n-1}_0 \in \X^{n-1}$ with $X^{n-1}_0 \cap Y^{n-1} ( f(x) ) \neq \emptyset$, we have $\inte \bigl( X^{n-1}_0 \bigr) \cap f (U^n(x) ) \neq \emptyset$. Since 
$
f( U^n (x) ) = \bigcup \{ f(X^n) : X^n \in \X^n, \, X^n \subseteq U^n(x) \}
$
and by Proposition~\ref{p_CellDecomp}~(i), Definition~\ref{d_cellular}, and Definition~\ref{d_celldecomp}, $f(X^n) \in \X^{n-1}$ for each $X^n \in \X^n$, we conclude that $\inte \bigl( X^{n-1}_0 \bigr) \cap f \bigl( X^n_0 \bigr) \neq \emptyset$ for some $n$-tile $X^n_0 \subseteq U^n(x)$. This would contradict with Definition~\ref{d_celldecomp} unless $X^{n-1}_0 = f \bigl( X^n_0 \bigr)$. Hence, $X^{n-1}_0 \subseteq f ( U^n(x) )$. 

Therefore, it follows from (\ref{e_Pf_l_Un_UY}) that $U^{n-1}(x) \subseteq f ( U^n(x) )$, and the proof is complete.
\end{proof}

Deducing from Lemma~\ref{l_Un}, we are able to strengthen the uniform local injectivity property of expanding Thurston maps proved in \cite[Lemma~5.5]{Li15} to the uniform local expansion property.

\begin{lemma}[Uniform local injectivity and expansion away from the critical points]   \label{l_Loc_Inject_Expansion_Away_from_Crit}
Let $f$, $d$, $\Lambda$ satisfy the Assumptions in Section~\ref{sct_Assumptions}. Let $\CC\subseteq S^2$ be a Jordan curve with $\post f \subseteq \CC$. Then there exist numbers $C_2 \in (0,1)$, $\delta_1 \in (0,1]$, and a function $\iota \: (0,\delta_1]\rightarrow (0,+\infty)$ with the following properties:
\begin{enumerate}
\smallskip
\item[(i)] $\lim\limits_{\delta \to 0} \iota (\delta) = 0$.

\smallskip
\item[(ii)] For each $\delta \leq \delta_1$, the map $f$ restricted to any open ball of radius $\delta$ centered outside the $\iota(\delta)$-neighborhood of $\crit f$ is injective, i.e., $f|_{B_d(x,\delta)}$ is injective for each $x \in S^2 \setminus N_d^{\iota(\delta)}(\crit f)$.

\smallskip
\item[(iii)] For all $\delta \in (0, \delta_1]$, $x,\,y \in S^2$, and $n\in\N$ the following statement holds:

\smallskip

If for each $j \in \{ 0, \, 1, \, \dots, \, n-1 \}$, $f^j(x) \in S^2 \setminus N_d^{\iota(\delta)} ( \crit f )$ and $d \bigl( f^j(x) , f^j (y) \bigr) < C_2 \delta$, then 
\begin{equation} \label{e_Unif_local_expansion}
C_2 \Lambda^n d(x, y) 
\leq  d( f^n (x), f^n(y) )
\leq C_2^{-1} \Lambda^n d(x, y).
\end{equation}
\end{enumerate}
\end{lemma}

\begin{proof}
By \cite[Lemma~5.5]{Li15}, there exists a number $\delta_1 \in (0,1]$ and a function $\iota\:(0,\delta_1] \rightarrow (0,+\infty)$ such that statements~(i) and~(ii) hold. It remains to establish statement~(iii).

Define
\begin{equation}  \label{e_C2}
C_2  \= \min \bigl\{ K^{-2} \Lambda^{-1}, \, K^{-1} \Lambda^{-2} \delta_1^{-1} \bigr\} \in (0,1), 
\end{equation}
where $K \geq 1$ is a constant depending only on $f$, $\CC$, and $d$ from Lemma~\ref{l_CellBoundsBM} (applied to $f$, $\CC$, and $d$).

Fix arbitrary $\delta \in (0, \delta_1]$, $x,\,y\in S^2$, and $n\in\N$. Assume that $f^j(x) \in S^2 \setminus N_d^{\iota(\delta)} ( \crit f )$ and $d \bigl( f^j(x) , f^j (y) \bigr) < C_2 \delta$ for each $j \in \{0,\, 1, \, \dots, \, n-1\}$.  

Fix an arbitrary $j \in \{0,\, 1, \, \dots, \, n-1\}$. Let $m_j \in \Z$ be the largest integer with $f^j (y) \in U^{m_j} \bigl( f^j(x) \bigr)$. Thus, $f^j(y) \notin U^{m_j + 1} \bigl( f^j(x) \bigr)$. By Lemma~\ref{l_CellBoundsBM}~(iii), such an integer $m_j$ exists and
\begin{equation}   \label{e_Pf_l_Loc_Inject_Expansion_Away_from_Crit_1}
K^{-1} \Lambda^{ - m_j  - 1} 
\leq d \bigl( f^j(x), f^j(y) \bigr)
<    K \Lambda^{ - m_j }.
\end{equation}
Thus,
\begin{equation}   \label{e_Pf_l_Loc_Inject_Expansion_Away_from_Crit_2}
K^{-1} \Lambda^{ - m_j  - 1}  \leq    d \bigl( f^j(x), f^j(y) \bigr)    < C_2 \delta \leq C_2 \delta_1 \leq K^{-1} \Lambda^{-2},
\end{equation}
the last inequality follows from (\ref{e_C2}). So $- m_j - 1 < - 2$, i.e., $m_j \geq 2$. 

It also follows from (\ref{e_Pf_l_Loc_Inject_Expansion_Away_from_Crit_2}) and (\ref{e_C2}) that 
\begin{equation}   \label{e_Pf_l_Loc_Inject_Expansion_Away_from_Crit_3}
\delta \geq C_2^{-1} K^{-1} \Lambda^{ - m_j - 1} \geq K \Lambda^{- m_j}.
\end{equation}
Thus, by (\ref{e_Pf_l_Loc_Inject_Expansion_Away_from_Crit_3}) and Lemma~\ref{l_CellBoundsBM}~(iii), $U^m \bigl( f^j(x) \bigr) \subseteq B_d \bigl( f^j(x), K \Lambda^{ - m } \bigr) \subseteq B_d \bigl( f^j(x), \delta \bigr)$ for each integer $m \geq m_j$.

By statement~(ii), $f$ is injective on $B_d \bigl( f^j(x), \delta \bigr)$. Thus, we get from Lemma~\ref{l_Un} and our choice of $m_j$ that 
\begin{align}
f^{j+1}(y) &\in      f \bigl( U^{m_j}       \bigl( f^j(x) \bigr) \bigr) = U^{m_j -1} \bigl( f^{j+1}(x) \bigr)   \qquad \text{ and} \label{e_Pf_l_Loc_Inject_Expansion_Away_from_Crit_4} \\
f^{j+1}(y) &\notin f \bigl( U^{m + 1} \bigl( f^j(x) \bigr) \bigr) = U^m \bigl( f^{j+1}(x) \bigr) \qquad \text{ for each integer } m \geq m_j.                                   \label{e_Pf_l_Loc_Inject_Expansion_Away_from_Crit_5} 
\end{align}
Hence, $m_{j+1} = m_j - 1$ for all $j\in \{ 0, \, 1, \, \dots, \, n-2 \}$ and consequently $m_{n-1} = m_0 - n + 1$.

By (\ref{e_Pf_l_Loc_Inject_Expansion_Away_from_Crit_4}), (\ref{e_Pf_l_Loc_Inject_Expansion_Away_from_Crit_5}) (both with $j \= n-1$ and $m \= m_j$), and Lemma~\ref{l_CellBoundsBM}~(iii),
\begin{equation}     \label{e_Pf_l_Loc_Inject_Expansion_Away_from_Crit_6} 
K^{-1} \Lambda^{ - (m_0 - n + 1) } 
\leq d( f^n(x), f^n(y) )
\leq K \Lambda^{ - (m_0 - n ) }.
\end{equation}
By (\ref{e_Pf_l_Loc_Inject_Expansion_Away_from_Crit_1}) (with $j\=0$), we get 
\begin{equation}     \label{e_Pf_l_Loc_Inject_Expansion_Away_from_Crit_7} 
K^{-1} \Lambda^{ - m_0 - 1 } 
\leq d( x, y )
\leq K \Lambda^{ - m_0}.
\end{equation}
Therefore, it follows from (\ref{e_C2}), (\ref{e_Pf_l_Loc_Inject_Expansion_Away_from_Crit_6}), and (\ref{e_Pf_l_Loc_Inject_Expansion_Away_from_Crit_7}) that 
\begin{equation*}
C_2 \Lambda^n \leq K^{-2}\Lambda^{n-1}
\leq \frac{ d( f^n(x), f^n(y) ) }{ d(x,y) }
\leq K^2 \Lambda^{n+1} \leq C_2^{-1} \Lambda^n.
\end{equation*}
This completes the proof.
\end{proof}

\section{Fine closing lemmas}  \label{sct_Closing_Lemmas}

The main goal of this section is to establish in Lemma~\ref{l_bound_by_gap} a local closing lemma that produces, for a nonempty compact forward-invariant set $\cK$ disjoint from critical points, a periodic orbit $\O$ close to $\cK$ in terms of its $(r,\theta)$-gap defined below. This result relies on two other forms of closing lemmas, namely, a local Anosov closing lemma (Subsection~\ref{subsct_Local_Anosov_Closing_Lemma}) and a (global) Bressaud--Quas closing lemma (Subsection~\ref{subsct_BQ}). Even though a global version of the Anosov closing lemma for expanding Thurston maps is available in \cite[Lemma~8.6]{Li18}, it only holds for sufficiently high iterations of the map and sufficiently long periodic pseudo-orbits. It is crucial in our proof of Lemma~\ref{l_bound_by_gap} to be able to close periodic pseudo-orbits of arbitrary length for the map itself. Therefore, we formulate and prove a local version of the Anosov closing lemma (Lemma~\ref{l_Local_Closing_Lemma}) from scratch. It closes periodic pseudo-orbits away from critical points to avoid the more complicated combinatorics near critical points. The proof relies on considerations over combinatorial structures like tiles, flowers, and bouquets. In Subsection~\ref{subsct_BQ}, we define the Bressaud--Quas shadowing property for general dynamical systems and prove that it can be passed on to related systems via the factor relation and iterations. We establish this property for expanding Thurston maps in Theorem~\ref{t_BQ_for_ETM}. Lemma~\ref{l_bound_by_gap} is then proved in Subsection~\ref{subsct_Local_Closing_Lemma_ETM_proof}.

\subsection{Local closing lemma away from critical points}   \label{subsct_Local_Closing_Lemma_ETM}

Let $T\: X\rightarrow X$ be a map on a compact metric space $(X,d)$. For each periodic orbit $\O$ of $T$, its \defn{gap} is defined as
\begin{equation}  \label{e_Def_gap}
\Delta (\O) = \Delta^d (\O) \= \min\{ d(x, y) : x,y \in \O, \, x \neq y \}.
\end{equation}
Here we adopt the convention that $\min \emptyset = +\infty$. For positive numbers $r, \, \theta \in \R$, the \defn{$(r,\theta)$-gap} of $\O$ is
\begin{equation}  \label{e_Def_gap_r_theta}
\Delta_{r,\,\theta} (\O) = \Delta^d_{r,\,\theta} (\O) \= \min\{ r, \, \theta \cdot \Delta (\O) \}.
\end{equation}
We often omit the superscript $d$ if it does not cause confusion.

\begin{lemma}  \label{l_bound_by_gap}
Let $f$, $\CC$, $d$ satisfy the Assumptions in Section~\ref{sct_Assumptions}. Let $\cK \subseteq S^2$ be a nonempty compact $f$-forward-invariant set disjoint from $\crit f$. Fix some real numbers $r>0$, $\theta>0$, $\alpha \in (0,1]$, and $\tau > 0$. Then there exists a periodic orbit $\O$ of $f$ with
\begin{equation} \label{e_l_bound_by_gap}
\sum_{x\in\O} d( x, \cK )^\alpha \leq \tau \cdot ( \Delta_{r,\,\theta} ( \O ) )^{\alpha}.
\end{equation}
\end{lemma}

\subsection{Local Anosov closing lemma}    \label{subsct_Local_Anosov_Closing_Lemma}

In order to establish our local Anosov closing lemma for expanding Thurston maps, we first provide a mechanism to locate periodic points using flowers in Lemma~\ref{l_Local_Fixed_Pt_Flower}, then relate flowers to bouquets in Lemma~\ref{l_U_in_W} before establishing a version of local Anosov closing lemma for sufficiently long periodic pseudo-orbits in Lemma~\ref{l_Local_Closing_Lemma_l} using tiles, flowers, and bouquets. Finally, in Lemma~\ref{l_Local_Closing_Lemma}, our local Anosov closing lemma is proved for periodic pseudo-orbits of all lengths.

We first demonstrate in the following lemma a mechanism to produce a periodic point. The proof relies on the expansion property of expanding Thurston maps. Note that the closure of a flower may not necessarily be simply connected. As a result, it takes extra care to locate periodic points from the combinatorial structures.

\begin{lemma}     \label{l_Local_Fixed_Pt_Flower}
Let $f\:S^2 \rightarrow S^2$ be an expanding Thurston map and $\CC \subseteq S^2$ be a Jordan curve containing $\post f$. Then for all $m,\, n \in \N$ with $m \geq n$ and each $m$-vertex $v^m \in \V^m (f, \CC)$ with $\oW^m ( v^m ) \subseteq W^{m-n} ( f^n (v^m) )$, if $f^n$ restricted to $W^m(v^m)$ is injective, then there exists $x \in \oW^m (v^m)$ such that $f^n(x) = x$.
\end{lemma}

Recall that $\oW^m( v^m )$ denotes the closure of $W^m ( v^m )$.

\begin{proof}
Fix $m,\,n \in \N$ and $v^m \in \V^m$ as in the statement of this lemma. Assume that $f^n$ restricted to $W^m (v^m)$ is injective. By Proposition~\ref{p_CellDecomp}~(iii), $v^{m-j} \= f^j (v^m) \in \V^{m-j}$ for each $j\in \{1,\,2,\,\dots,\, n \}$. 

Fix an arbitrary integer $i \in \{ 1, \, 2, \, \dots, \, n \}$. Recall from Remark~\ref{r_Flower} that $f^i (W^m (v^m) ) = W^{ m - i } \bigl( v^{m-i} \bigr)$. Then $f$ maps $W^{m-i+1} \bigl( v^{m-i+1} \bigr)$ injectively to $W^{m-i} \bigl( v^{m-i} \bigr)$. It follows from Remark~\ref{r_Flower} again that $f$ restricted to $W^{m-i+1} \bigl( v^{m-i+1} \bigr)$ has an continuous inverse denoted as $g_i \: W^{m-i} \bigl( v^{m-i} \bigr) \rightarrow W^{m-i+1} \bigl( v^{m-i+1} \bigr)$. Define a map $g \: W^{m-n} ( v^{m-n} ) \rightarrow W^m ( v^m )$ to be $g \= g_n \circ g_{n-1} \circ \cdots \circ g_2 \circ g_1$. Then $g$ is the inverse of $f^n|_{W^m ( v^m )}$ and $g$ is a homeomorphism.

Next, we recursively construct $v_j \in \V^{m+jn}$ satisfying that for each $j \in \N_0$,
\begin{align}
\oW^{m+jn} ( v_j ) & \subseteq  W^{m+(j-1)n} ( v_{j-1} )   \subseteq W^{m-n} ( v^{m-n} )  \quad \text{and} \\
W^{m+jn} ( v_j ) & = g \bigl( W^{m+(j-1)n} ( v_{j-1} ) \bigr).     \label{e_Pf_l_Local_Fixed_Pt_Flower_gW}
\end{align}

\smallskip

We set $v_{-1} \= v^{m-n}$ and $v_0 \= v^m$. Then the base step (i.e., $j=0$) has already been verified above.

For the recursive step, we assume that such $v_j \in \V^{m+jn}$ has been constructed for some $j\in\N_0$. Then by Proposition~\ref{p_CellDecomp}~(ii), $v_{j+1} \= g (v_j) \in \V^{m+(j+1)n}$. In particular, $v_{j+1} \in W^{m+jn} ( v_j )$. We observe that it follows from (\ref{e_Def_Flower}), Remark~\ref{r_Flower}, Proposition~\ref{p_CellDecomp}~(i),~(ii), and Definition~\ref{d_celldecomp}~(iii) that for all $l,\,k\in \N_0$ and $v\in \V^k$, $f^{-l} \bigl( W^k (v) \bigr) = \bigcup_{v' \in f^{-l} (v) } W^{k+l} (v')$, and that for distinct vertices $v_1,\, v_2 \in f^{-l} (v)$ we have $W^{k+l} ( v_1 ) \cap W^{k+l} ( v_2 ) = \emptyset$. Hence, $g \bigl( W^{m+jn} ( v_j ) \bigr) = W^{m+(j+1)n} ( v_{j+1} )$ and $W^{m+(j+1)n} ( v_{j+1} ) \subseteq W^{m+jn} ( v_j )$. Moreover, since $g$ is a homeomorphism and $\oW^{m+jn}(v_j) \subseteq W^{m+(j-1)n}(v_{j-1}) \subseteq W^{m-n}(v^{m-n})$, we get $\oW^{m+(j+1)n} ( v_{j+1} )  \subseteq W^{m+jn} ( v_j )\subseteq W^{m-n}(v^{m-n})$. 

The recursive construction is complete.

\smallskip

We have constructed a nested sequence $\bigl\{ \oW^{m+jn} ( v_j )  \bigr\}_{ j \in \N_0}$ of closed sets. By Lemma~\ref{l_CellBoundsBM}~(ii), the intersection $\bigcap_{ j \in \N_0} \oW^{m+jn} ( v_j )$ contains exactly one point, say $x$. It follows from (\ref{e_Pf_l_Local_Fixed_Pt_Flower_gW}) that $g(x) = x$. Therefore, $f^n(x) = x$.
\end{proof}

We relate flowers and bouquets of similar levels in the following lemma from \cite[Lemma~8.5]{Li18}, which will be crucial in the proof of Lemma~\ref{l_Local_Closing_Lemma_l} below.

\begin{lemma}     \label{l_U_in_W}
Let $f\:S^2 \rightarrow S^2$ be an expanding Thurston map and $\CC \subseteq S^2$ be a Jordan curve containing $\post f$. Then there exists a number $\kappa \in \N_0$ such that the following statement holds:

For each $x \in S^2$, each $n \in \N_0$, and each $n$-tile $X^n \in \X^n(f,\CC)$, if $x\in X^n$, then there exists an $n$-vertex $v^n \in \V^n (f, \CC) \cap X^n$ with $U^{ n + \kappa } (x) \subseteq W^n (v^n)$.
\end{lemma}

We establish below a form of \emph{local Anosov closing lemma away from critical points}. Due to the combinatorial structures used in the proof, we are only able to close sufficiently long pseudo-orbits.

\begin{lemma}     \label{l_Local_Closing_Lemma_l}
Let $f$, $d$, $\Lambda$ satisfy the Assumptions in Section~\ref{sct_Assumptions}. Let $\CC\subseteq S^2$ be a Jordan curve with $\post f \subseteq \CC$. Then there exist numbers $M_0 \in \N$ and $\beta_0 > 1$ such that for each $\eta \in (0,1)$ there exists a number $\delta_2 \in (0,1)$ with the following property:

For each $\delta \in  (0, \delta_2 ]$, if $x\in S^2$ and $l\in\N$ satisfy $l \geq M_0$, $d \bigl( x,f^l(x) \bigr) \leq \delta$, and $d \bigl( f^i (x) , \crit f \bigr) \geq \eta$ for all $i\in\{ 0, \, 1, \, \dots, \, l \}$, then there exists $y\in S^2$ such that $f^l(y)=y$ and $d(f^i(x),f^i(y))\leq \beta_0\delta\Lambda^{-(l-i)} \leq \eta / 2$ for each $i\in\{0,\,1,\,\dots,\,l\}$.
\end{lemma}

\begin{proof}
Fix $f$, $\CC$, $d$, and $\eta$ as in the statement of this lemma.

By Lemma~\ref{l_Loc_Inject_Expansion_Away_from_Crit}, we can fix a positive number $\delta_2$ satisfying
\begin{equation}   \label{e_Pf_l_Local_Closing_Lemma_l_delta}
\iota \bigl( 4K^2 \Lambda^{\kappa +1} \delta_2 \bigr) < \eta / 2
\quad \text{and} \quad
\delta_2 < \min \{ \delta_1, \,  \eta   \}\big/  \bigl( 8 K^2 \Lambda^{\kappa +1} \bigr) .
\end{equation}
In particular,
\begin{equation}   \label{e_Pf_l_Local_Closing_Lemma_I_delta2}
4K^2 \Lambda^{\kappa +1} \delta_2 < \delta_1.
\end{equation}
Here $K\geq 1$, $\kappa \in \N_0$, and $\delta_1 \in (0,1]$ are constants depending only on $f$, $\CC$, and $d$ from Lemma~\ref{l_CellBoundsBM}, Lemma~\ref{l_U_in_W}, and Lemma~\ref{l_Loc_Inject_Expansion_Away_from_Crit}, respectively, and $\iota \: [0, \delta_1] \rightarrow (0, +\infty)$ is a function depending only on $f$, $\CC$, and $d$ from Lemma~\ref{l_Loc_Inject_Expansion_Away_from_Crit}. Define
\begin{equation}  \label{e_Pf_l_Local_Closing_Lemma_l_M}
M_0 \= \bigl\lfloor \log_{\Lambda} \bigl( 4 K^2 \bigr) \bigr\rfloor + \kappa \in \N_0.
\end{equation}

Recall from Lemma~\ref{l_CellBoundsBM}~(iii) that for each $z\in S^2$ and each $n\in \N_0$, we have
\begin{equation} \label{e_Pf_l_Local_Closing_Lemma_l_U}
B_d \bigl( z, K^{-1} \Lambda^{-n} \bigr)  \subseteq U^n(z) \subseteq B_d ( z, K \Lambda^{-n} ).
\end{equation}

Fix arbitrary $\delta \in (0, \delta_2]$, $x\in S^2$, and $l\in\N$ satisfying $l\geq M_0$, $d \bigl( x, f^l(x) \bigr) \leq \delta$, and $d \bigl( f^i (x) , \crit f \bigr) \geq \eta$ for all $i\in\{ 0, \, 1, \, \dots, \, l \}$. Set
\begin{equation}  \label{e_Pf_l_Local_Closing_Lemma_l_N}
N \= \lfloor  - \log_{\Lambda} ( 2 K \delta) \rfloor - \kappa.
\end{equation}
Note that it follows from $\delta \in (0, \delta_2]$, $K\geq 1$, $\eta \in (0,1)$, and (\ref{e_Pf_l_Local_Closing_Lemma_l_delta}) that $N\in\N_0$.

We first construct an $(N+l)$-flower $W^{N+l} \bigl( v^{N+l} \bigr)$ containing $x$ whose closure is contained in its image under $f^l$ as detailed below. Compare Figure~\ref{f_1} through Figure~\ref{f_3} for the construction.

\begin{figure}
    \centering
    \begin{overpic}
    [width=8cm, 
    tics=20]{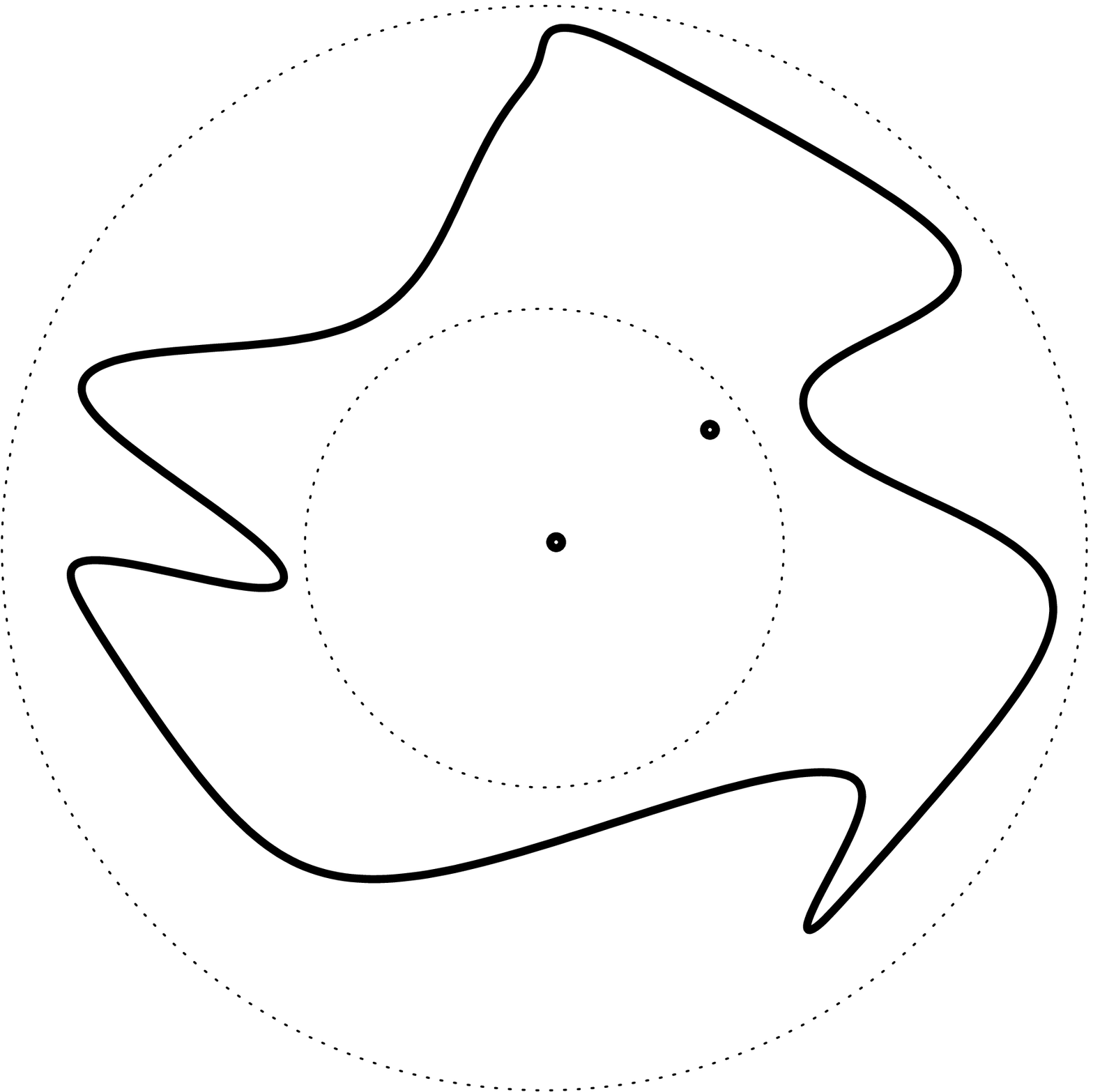}
    \put(138,128){$x$}    
    \put(100,100){$f^l(x)$}
    \put(104,40){$U^{N+\kappa} \bigl( f^l (x) \bigr)$}
    \end{overpic}
    \caption{Closing a periodic pseudo-orbit I: $U^{N+\kappa} \bigl( f^l (x) \bigr)$.}
    \label{f_1}
\end{figure}

\begin{figure}
    \centering
    \begin{overpic}
    [width=10cm, 
    tics=20]{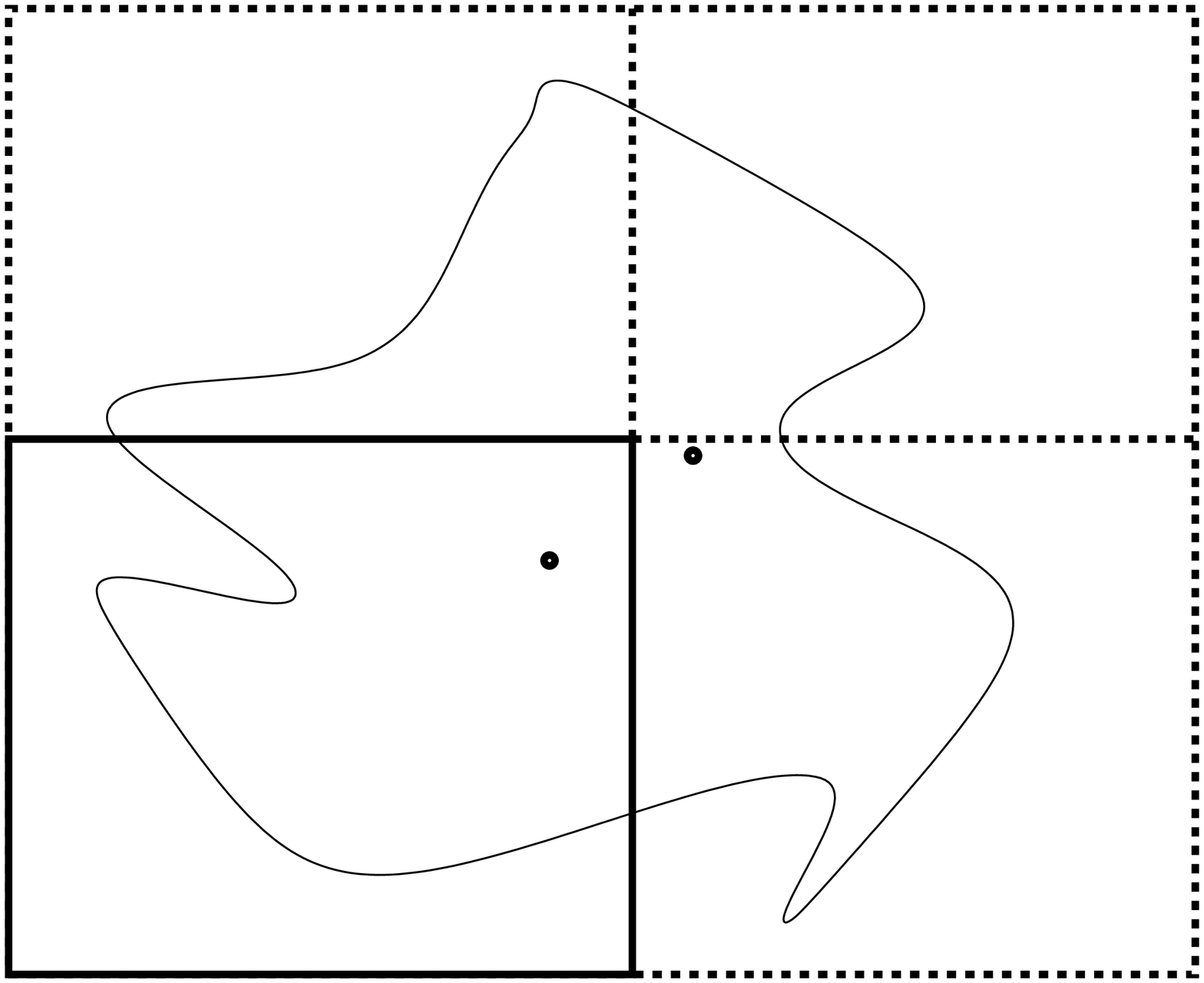}
    \put(158,112){$x$}    
    \put(114,85){$f^l(x)$}
    \put(65,36){$U^{N+\kappa} \bigl( f^l (x) \bigr)$}
    \put(10,10){$X^N$} 
    \put(151,131){$v^N$} 
    \put(233,212){$W^N(v^N)$}             
    \end{overpic}
    \caption{Closing a periodic pseudo-orbit II: $X^N$.}
    \label{f_2}
\end{figure}

\begin{figure}
    \centering
    \begin{overpic}
    [width=10cm, 
    tics=20]{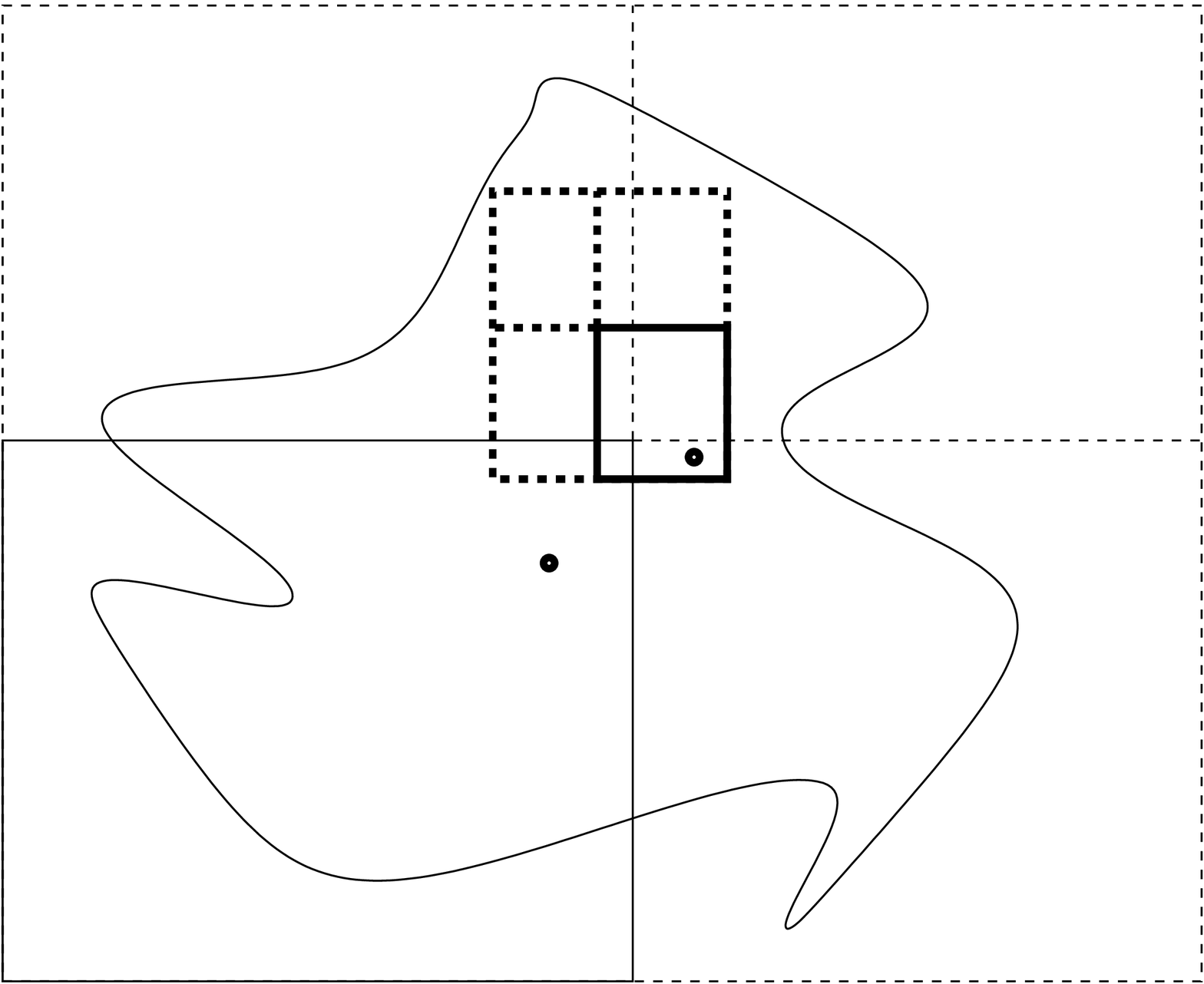}
    \put(158,112){$x$}    
    \put(114,85){$f^l(x)$}
    \put(65,36){$U^{N+\kappa} \bigl( f^l (x) \bigr)$}
    \put(10,10){$X^N$} 
    \put(151,131){$v^N$} 
    \put(233,212){$W^N(v^N)$} 
    \put(173,163){$W^{N+l}(v^{N+l})$}     
    \put(168,108){$X^{N+l}$}           
    \put(142,158){$v^{N+l}$}     
    \end{overpic}
    \caption{Closing a periodic pseudo-orbit III: $X^{N+l}$.}
    \label{f_3}
\end{figure}

\begin{figure}
    \centering
    \begin{overpic}
    [width=10cm, 
    tics=20]{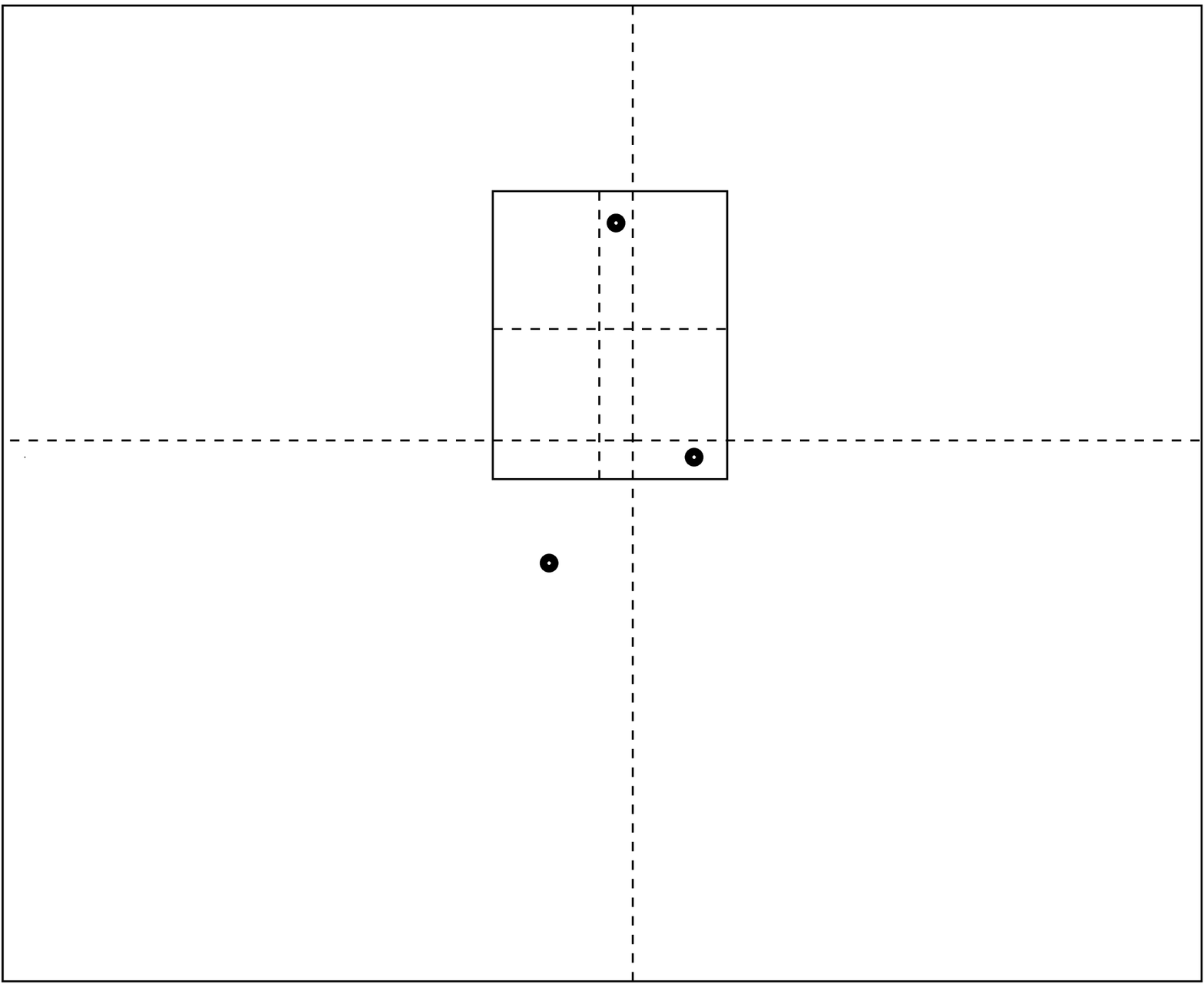}
    \put(158,112){$x$}    
    \put(114,85){$f^l(x)$}
    \put(151,131){$v^N$} 
    \put(233,212){$W^N(v^N)$} 
    \put(173,163){$W^{N+l}(v^{N+l})$}          
    \put(142,158){$v^{N+l}$}   
    \put(132,178){$y$}     
    \end{overpic}
    \caption{Closing a periodic pseudo-orbit IV: $y$.}
    \label{f_4}
\end{figure}

Let $X^N \in \X^N$ be an $N$-tile containing $f^l(x)$. By Lemma~\ref{l_U_in_W}, there exists an $N$-vertex $v^N \in \V^N \cap X^N$ such that $U^{N+\kappa} \bigl( f^l(x) \bigr) \subseteq W^N (v^N)$ (see Figure~\ref{f_2}). By Proposition~\ref{p_CellDecomp}, there exist $X^{N+l} \in \X^{N+l}$ and $v^{N+l} \in \V^{N+l} \cap X^{N+l}$ such that $x \in X^{N+l}$, $f^l \bigl( X^{N+l} \bigr) = X^N$, and $f^l \bigl( v^{N+l} \bigr) = v^N$ (see Figure~\ref{f_3}). Then $x \in W^{N+l} \bigl( v^{N+l} \bigr)$ by Proposition~\ref{p_CellDecomp}~(i). Since $d \bigl( x, f^l(x) \bigr) \leq \delta$, $l \geq M_0$, and $W^{N+l}( v^{N+l} ) \subseteq U^{N+l}(x)$, we get from (\ref{e_Pf_l_Local_Closing_Lemma_l_U}), (\ref{e_Pf_l_Local_Closing_Lemma_l_N}), and (\ref{e_Pf_l_Local_Closing_Lemma_l_M}) that if $z \in W^{N+l} \bigl( v^{N+l} \bigr)$, then
\begin{align}  \label{e_Pf_l_Local_Closing_Lemma_l_W_in_U}
d \bigl( f^l(x) , z \bigr)
& \leq d \bigl( f^l(x) , x \bigr) + d ( x, z ) 
   \leq \delta + 2 K \Lambda^{-(N+l)}  \notag \\ 
& \leq \frac{\Lambda^{-(N+\kappa)} }{2K}  + \frac{ 2 K \Lambda^{-(N+\kappa)} }{ 4 K^2 }
   \leq K^{-1} \Lambda^{ - ( N + \kappa ) }.
\end{align}
Thus, by (\ref{e_Pf_l_Local_Closing_Lemma_l_U}), (\ref{e_Pf_l_Local_Closing_Lemma_l_W_in_U}), $f^l \bigl( v^{N+l} \bigr) = v^N$, and (\ref{e_Flower_iterate}), we get (see Figure~\ref{f_3})
\begin{equation}  \label{e_Pf_l_Local_Closing_Lemma_l_W}
\oW^{N+l} \bigl( v^{N+l} \bigr)
\subseteq U^{ N + \kappa } \bigl( f^l(x) \bigr) 
\subseteq W^N \bigl( v^N \bigr)
= f^l \bigl(  W^{N+l} \bigl( v^{N+l} \bigr)   \bigr).
\end{equation}

Next, we claim that $f^l$ is injective on $W^{N+l} \bigl( v^{N+l} \bigr)$.

\smallskip

Indeed, we consider an arbitrary integer $i\in \{0, \, 1, \, \dots, \, l \}$. Since $x \in  W^{N+l} \bigl( v^{N+l} \bigr)$ and consequently $f^i(x) \in f^i \bigl( W^{N+l} \bigl( v^{N+l} \bigr) \bigr)$, we get from Remark~\ref{r_Flower}, (\ref{e_Def_U^n}), and Proposition~\ref{p_CellDecomp} that 
\begin{equation*}
f^i \bigl( W^{N+l} \bigl( v^{N+l} \bigr) \bigr) \subseteq U^{ N+l-i } \bigl( f^i(x) \bigr).
\end{equation*}
By (\ref{e_Pf_l_Local_Closing_Lemma_l_U}), (\ref{e_Pf_l_Local_Closing_Lemma_l_N}), and $\delta \in (0, \delta_2)$, we get
\begin{align}
U^{ N+l-i } \bigl( f^i (x) \bigr) 
& \subseteq B_d \bigl( f^i(x), K \Lambda^{-N-l+i} \bigr)  \notag \\
& \subseteq B_d \bigl( f^i(x), 2K^2 \delta \Lambda^{ \kappa - l + 1  + i} \bigr)
   \subseteq B_d \bigl( f^i(x), 2K^2 \delta_2 \Lambda^{ \kappa + 1 } \bigr).   \label{e_Pf_l_Local_Closing_Lemma_l_U_radius}
\end{align}
Thus, by $d \bigl( f^i(x), \crit f \bigr) \geq \eta$, (\ref{e_Pf_l_Local_Closing_Lemma_l_delta}), (\ref{e_Pf_l_Local_Closing_Lemma_I_delta2}), and Lemma~\ref{l_Loc_Inject_Expansion_Away_from_Crit}~(ii), $f$ is injective on $U^{N+l-i} \bigl( f^i (x) \bigr) \supseteq f^i \bigl( W^{N+l} \bigl( v^{N+l} \bigr) \bigr)$. The claim now follows.

\smallskip

Hence, by (\ref{e_Pf_l_Local_Closing_Lemma_l_W}), the claim above, Lemma~\ref{l_Local_Fixed_Pt_Flower}, and $x \in W^{N+l} \bigl( v^{N+l} \bigr)$, there exists $y \in \oW^{N+l} ( v^{N+l} )    \subseteq U^{N+l} (x)$ such that $f^l(y) = y$. See Figure~\ref{f_4}.

We set 
\begin{equation}  \label{e_Pf_l_Local_Closing_Lemma_l_beta}
\beta_0 \= 2K^2 \Lambda^{\kappa +1} \geq 1.
\end{equation}

It suffices now to verify that $ d \bigl( f^i (x) , f^i (y) \bigr) \leq \beta_0 \delta \Lambda^{ - (l-i) } \leq \eta / 2$ for each $i\in\{ 0,\,1,\,\dots,\, l\}$. Indeed, by Remark~\ref{r_Flower}, (\ref{e_Def_U^n}), the fact that $x,\,y \in \oW^{N+l} \bigl( v^{N+l} \bigr)$, and Lemma~\ref{l_Un}, we get
\begin{equation*}
f^i (y)
\in f^i \bigl( U^{N+l} (x) \bigr)
=  U^{N+l-i} \bigl( f^i (x) \bigr)
\end{equation*}
for each $i \in \{0,\,1,\,\dots,\,l\}$. Thus, by (\ref{e_Pf_l_Local_Closing_Lemma_l_U_radius}), (\ref{e_Pf_l_Local_Closing_Lemma_l_beta}), and (\ref{e_Pf_l_Local_Closing_Lemma_l_delta}),
\begin{equation*}
d \bigl( f^i(x), f^i(y) \bigr)
\leq 2 K^2 \delta \Lambda^{ \kappa + 1 - (l - i) } 
= \beta_0 \delta \Lambda^{ - (l-i) } 
\leq \beta_0  \delta_2
\leq \eta / 2.
\end{equation*}

The proof is now complete.
\end{proof}

\begin{lemma}   \label{l_expansion_upper_bound}
Let $f$, $d$, $\Lambda$ satisfy the Assumptions in Section~\ref{sct_Assumptions}. Let $\CC\subseteq S^2$ be a Jordan curve with $\post f \subseteq \CC$. For all $x,\,y \in S^2$ and $n\in\N_0$, we have 
\begin{equation*}
d( f^n (x), f^n(y) )  \leq K^2 \Lambda^{n+1} d(x, y),
\end{equation*}
where $K \geq 1$ is a constant depending only on $f$, $\CC$, and $d$ from Lemma~\ref{l_CellBoundsBM}.
\end{lemma}

\begin{proof}

Fix arbitrary $x,\,y \in S^2$ and $n\in\N_0$. It suffices to consider $n \geq 1$. Recall $U^0(z) = S^2$ for all $z\in S^2$. For each $j \in \N_0$, denote by $m_j \in \N_0$ the largest integer with $f^j (y) \in U^{m_j} \bigl( f^j(x) \bigr)$. Thus, $f^j(y) \notin U^{m_j + 1} \bigl( f^j(x) \bigr)$. By Lemma~\ref{l_CellBoundsBM}~(iii), such an integer $m_j$ exists and
\begin{equation*}
K^{-1} \Lambda^{ - m_j  - 1} 
\leq d \bigl( f^j(x), f^j(y) \bigr)
<    K \Lambda^{ - m_j }.
\end{equation*}
By Lemma~\ref{l_Un}, for each $j\in\N_0$, $f^{j+1}(y) \in f \bigl( U^{m_j} \bigl( f^j (y) \bigr) \bigr) =  U^{m_j - 1} \bigl( f^{j+1} (y) \bigr)$. So $m_{j+1} \geq m_j -1$. Therefore, $d( f^n (x), f^n(y) )  
< K \Lambda^{- m_n} 
\leq  K \Lambda^{ - m_0 + n} 
\leq K^2 \Lambda^{n+1} \bigl( K^{-1} \Lambda^{ - m_0 - 1 } \bigr) 
\leq K^2 \Lambda^{n+1} d(x, y)$.
\end{proof}

In order to close pseudo-orbits of period smaller than $M_0$, we need a local Anosov closing lemma (away from critical points) using Lemma~\ref{l_Local_Closing_Lemma_l} above and the uniform local expansion property (see Lemma~\ref{l_Loc_Inject_Expansion_Away_from_Crit}).

\begin{lemma}[Local Anosov Closing Lemma]     \label{l_Local_Closing_Lemma}
Let $f$, $d$, $\Lambda$ satisfy the Assumptions in Section~\ref{sct_Assumptions}. Let $\CC \subseteq S^2$ be a Jordan curve containing $\post f$. Fix a real number $\eta > 0$. Then there exists a constant $\delta_3 \in (0,1)$ such that the following statement holds:

For each $\delta \in  (0, \delta_3 ]$, if $x\in S^2$ and $l\in\N$ satisfy $d \bigl( x,f^l(x) \bigr) \leq \delta$ and $d \bigl( f^i (x) , \crit f \bigr) \geq \eta$ for all $i\in\{ 0, \, 1, \, \dots, \, l \}$, then there exists $y\in S^2$ such that $f^l(y)=y$ and $d(f^i(x),f^i(y))\leq \beta_0\delta\Lambda^{-(l-i)} \leq \eta / 2$ for each $i\in\{0,\,1,\,\dots,\,l\}$. Here $\beta_0 > 1$ is a constant from Lemma~\ref{l_Local_Closing_Lemma_l} depending only on $f$, $\CC$, and $d$.
\end{lemma}

\begin{proof}
Fix $f$, $\CC$, $d$, and $\eta$ as in the statement of this lemma. 

By Lemma~\ref{l_Loc_Inject_Expansion_Away_from_Crit}, we can fix a positive number $\delta' < \min \{ \delta_2 ,\, \delta_1,\, \eta \}$ satisfying
\begin{equation}  \label{e_Pf_l_Local_Closing_Lemma_half_distance}
\iota (\delta') < \eta / 2.
\end{equation}
Here the constant $\delta_1 \in (0,1]$ and the function $\iota\: [0,\delta_1] \rightarrow (0, +\infty)$ are from Lemma~\ref{l_Loc_Inject_Expansion_Away_from_Crit} depending only on $f$, $\CC$, and $d$; and the constant $\delta_2 \in (0,1)$ is from Lemma~\ref{l_Local_Closing_Lemma_l} depending only on $f$, $\CC$, $d$, and $\eta$. Define
\begin{equation}  \label{e_Pf_l_Local_Closing_Lemma_delta}
\delta_3 \= \min \biggl\{ \delta' \frac{ \Lambda^{- M_0^2 - 1 } }{4 M_0} C_2 K^{-2} , \, \frac{ \eta }{ 8 \beta_0 } , \, \frac{ \delta' }{ 4 \beta_0 } \biggr\} < \delta_2.
\end{equation}
Here $C_2 \in (0,1)$, $K \geq 1$, and $M_0 \in \N$ are constants depending only on $f$, $\CC$, and $d$ from Lemma~\ref{l_Loc_Inject_Expansion_Away_from_Crit}, Lemma~\ref{l_CellBoundsBM}, and Lemma~\ref{l_Local_Closing_Lemma_l}, respectively.

Fix arbitrary $\delta \in ( 0, \delta_3 )$, $x \in S^2$, and $l \in \N$. Assume that $d(x , f^l(x) ) \leq \delta$ and $d \bigl( f^i (x) , \crit f \bigr) \geq \eta$ for all $i\in\{ 0, \, 1, \, \dots, \, l \}$. Write $z \= f^l(x)$.

The case when $l \geq M_0$ follows immediately from Lemma~\ref{l_Local_Closing_Lemma_l}. Thus, without loss of generality, we can assume that $l < M_0$ and $M_0 \geq 2$. 

Since $d(x, z) \leq \delta < \delta_3$, we get from Lemma~\ref{l_expansion_upper_bound} and (\ref{e_Pf_l_Local_Closing_Lemma_delta}) that for each $i \in \{ 0, \, 1, \, \dots , \, l M_0 \}$, 
\begin{equation}   \label{e_Pf_l_Local_Closing_Lemma_xz_distance_bound}
d \bigl( f^i(x), f^i(z) \bigr)
\leq K^2 \Lambda^{i+1} d(x,z)
<     K^2 \Lambda^{i+1} \delta_3
\leq \frac{ C_2 \delta' \Lambda^{ - M_0^2 + i } }{ 4 M_0 }
\leq \frac{ C_2 \delta' }{ 4 M_0 }.
\end{equation}
Hence, by the triangular inequality, for each $j \in \{ 1, \, 2, \, \dots, \, M_0 \}$ and $k \in \{ 0, \, 1, \, \dots, \, l \}$,
\begin{equation}  \label{e_Pf_l_Local_Closing_Lemma_xk}
d \bigl( f^k (x), f^{ l j + k }(x) \bigr)
\leq  \sum_{m=0}^{ j - 1 }  d \bigl( f^{ l m + k } (x), f^{ l m + k } (z) \bigr)
\leq \frac{ j C_2 \delta' }{ 4 M_0 }
\leq \frac{ C_2 \delta' }{4}.
\end{equation} 
In particular, since $d\bigl( f^k (x), \crit f \bigr) \geq \eta$, $\delta' \leq \eta$, and $C_2 \in (0,1)$, we get
\begin{equation}  \label{e_Pf_l_Local_Closing_Lemma_x_orbit_Ml_crit_f}
d \bigl( f^i (x), \crit f \bigr) > 3\eta / 4 
\qquad \text{ for all } i \in \{ 0, \, 1, \, \dots , \, lM_0 + l \}.
\end{equation}
On the other hand, it follows from (\ref{e_Pf_l_Local_Closing_Lemma_xk}) that $d \bigl( x, f^{lM_0} (x) \bigr) \leq \delta'$.
Since $\delta' < \delta_2$, by Lemma~\ref{l_Local_Closing_Lemma_l}, there exists $y \in S^2$ such that $f^{ l M_0 } (y) = y$ and
\begin{equation}  \label{e_Pf_l_Local_Closing_Lemma_y_shadow_x}
d \bigl( f^i (x), f^i(y) \bigr) 
\leq \beta_0 \delta \Lambda^{ - ( l M_0 - i ) }
\end{equation}
for each $i \in \{0, \, 1, \, \dots, \, l M_0 \}$. In particular when $i \in \{0, \, 1, \, \dots, \, l \}$ we have $d \bigl( f^i (x), f^i(y) \bigr) \leq \beta_0 \delta \Lambda^{ - ( l - i ) }$.

It now suffices to show that $f^l(y) = y$. We argue by contradiction and assume that $f^l(y) \neq y$. Then by (\ref{e_Pf_l_Local_Closing_Lemma_y_shadow_x}), $\delta \in (0, \delta_3)$, and (\ref{e_Pf_l_Local_Closing_Lemma_delta}), for each $i \in \{0, \, 1, \, \dots, \, l M_0 \}$, we have
\begin{equation}   \label{e_Pf_l_Local_Closing_Lemma_y_shadow_K}
d \bigl( f^i (x), f^i(y) \bigr) 
\leq \beta_0 \delta
\leq \eta / 8.
\end{equation}
We now get from (\ref{e_Pf_l_Local_Closing_Lemma_half_distance}), (\ref{e_Pf_l_Local_Closing_Lemma_x_orbit_Ml_crit_f}), and (\ref{e_Pf_l_Local_Closing_Lemma_y_shadow_K}) that $f^i(y) \in S^2 \setminus N_d^{\iota( \delta' )} (\crit f )$ for each $i \in \{ 0, \, 1, \, \dots, \, l M_0 \}$.

Next, we verify that for each $i \in \{ 0, \, 1, \, \dots, \, l M_0 \}$ we have 
\begin{equation*}
d \bigl( f^i (y), f^{l+i} (y) \bigr) \leq C_2 \delta'.
\end{equation*}
We establish this bound by discussing in the following two cases:

\smallskip

\emph{Case~1.} We assume $i \leq l M_0 - l$. Then by (\ref{e_Pf_l_Local_Closing_Lemma_y_shadow_K}), (\ref{e_Pf_l_Local_Closing_Lemma_xz_distance_bound}), $\delta \in (0, \delta_3)$, and (\ref{e_Pf_l_Local_Closing_Lemma_delta}), we get
\begin{align*}
   d\bigl( f^i( y ), f^{l+i} (y) \bigr)   
& \leq d\bigl( f^i( y ), f^i (x) \bigr)  +  d\bigl( f^i( x ), f^{l+i} (x) \bigr)  +  d\bigl( f^{l+i}( x ), f^{l+i} (y) \bigr)  \\
& \leq \beta_0 \delta_3 + \frac{ C_2 \delta' }{4} + \beta_0 \delta_3
   \leq \frac{ C_2 \delta' }{4} + \frac{ C_2 \delta' }{4} + \frac{ C_2 \delta' }{4}
   \leq  C_2 \delta'.
\end{align*}

\smallskip

\emph{Case~2.} We assume $l M_0 - l < i \leq l M_0$. Then $n \= l+i - l M_0 \in [1, l]$. By (\ref{e_Pf_l_Local_Closing_Lemma_y_shadow_K}), (\ref{e_Pf_l_Local_Closing_Lemma_xz_distance_bound}), (\ref{e_Pf_l_Local_Closing_Lemma_xk}), $\delta \in (0, \delta_3)$, (\ref{e_Pf_l_Local_Closing_Lemma_delta}), and the fact that $f^{l M_0} (y) = y$, we get
\begin{align*}
            d\bigl( f^i( y ), f^{l+i} (y) \bigr)  
& \leq d\bigl( f^i( y ), f^i (x) \bigr)  +  d\bigl( f^i( x ), f^{l+i} (x) \bigr)  +   d\bigl( f^{l+i}( x ), f^n (x) \bigr) + d( f^n( x ), f^n (y) )   \\
& \leq \beta_0 \delta_3 + \frac{ C_2 \delta' }{4} + \frac{ C_2 \delta' }{4} + \beta_0 \delta_3
   \leq \frac{ C_2 \delta' }{4} + \frac{ C_2 \delta' }{4} + \frac{ C_2 \delta' }{4} + \frac{ C_2 \delta' }{4} 
   =  C_2 \delta'.
\end{align*}

\smallskip

Finally, we recall again that $f^{l M_0} (y) = y$. Thus, we have $f^i(y) \in S^2 \setminus N_d^{\iota( \delta' )} (\crit f )$ and $d \bigl( f^i (y), f^{l+i} (y) \bigr) \leq C_2 \delta'$ for all $i\in\N_0$.

Hence, we can apply the uniform local expansion property from Lemma~\ref{l_Loc_Inject_Expansion_Away_from_Crit}~(iii) to get
\begin{equation*}
d \bigl( y, f^l(y) \bigr)
     = d \bigl( f^{ l M_0 k } (y), f^{ l M_0 k + l} (y) \bigr) 
\geq C_2 \Lambda^{ l M_0 k } d \bigl( y, f^l(y) \bigr)
\end{equation*}
for all $k \in \N$. This can only be true when $y = f^l (y)$, a contradiction. 

We have now concluded that $y = f^l (y)$.
\end{proof}

\subsection{Bressaud--Quas closing lemma}   \label{subsct_BQ}

\begin{definition}   \label{d_BQ}
Let $T\: X \rightarrow X$ be a map on a metric space $(X,d)$. We say that the dynamical system $(X,T)$ has the \emph{Bressaud--Quas shadowing property} if for every nonempty compact $T$-forward-invariant set $\cK$ and every number $\kappa > 0$, there exists a number $\epsilon_0 > 0$ such that for each $\epsilon \in (0, \epsilon_0)$, we can find a periodic orbit $\O$ of $T$ of period $p < (1/\epsilon)^{\kappa}$ contained entirely in the $\epsilon$-neighborhood $N_d^{\epsilon} (\cK)$ of $\cK$.
\end{definition}

We first prove some properties of the Bressaud--Quas shadowing property.

\begin{prop}  \label{p_BQ_factor}
Let $f\: X \rightarrow X$ and $g \: Y \rightarrow Y$ be maps on compact metric spaces $(X,d_X)$ and $(Y,d_Y)$, respectively. Assume that $(X,f)$ is a factor of $(Y,g)$ with a H\"older continuous factor map $\pi\:Y \rightarrow X$. If $(Y,g)$ has the Bressaud--Quas shadowing property, then so does $(X,f)$.
\end{prop}

\begin{proof}
Assume that $\pi$ is $\beta$-H\"older continuous, $\beta \in (0,1]$. Denote 
\begin{equation*}
\Hseminorm{\beta}{\pi} \= \sup \bigl\{ \abs{ d_X (\pi (x), \pi (y) ) } / d_Y (x,y)^\beta : x,\,y \in Y , \, x \neq y \bigr\}.
\end{equation*}

We fix a nonempty compact $f$-forward-invariant set $\cK' \subseteq X$. Denote $\cK \= \pi^{-1} (\cK')$. Then $\cK$ is nonempty compact and satisfies
$\pi  ( g (\cK) ) =  f ( \pi (\cK) ) = f (\cK') \subseteq \cK'$. Thus, $\cK$ is $g$-forward-invariant.

Fix an arbitrary $\kappa' > 0$. Let $\epsilon_0>0$  be the constant from the Bressaud--Quas shadowing property applied to $(Y, g)$, $\cK$, and $\kappa \=  \kappa' \beta /2$. Without loss of generality, we assume that $\epsilon_0 \leq \Hseminorm{\beta}{\pi}^{ - 2 / \beta }$.

Set $\epsilon'_0 \= \Hseminorm{\beta}{\pi} \cdot (\epsilon_0)^{\beta} \leq \Hseminorm{\beta}{\pi}^{-1}$. Fix an arbitrary $\epsilon' \in (0, \epsilon'_0)$ and write $\epsilon \= \bigl( \epsilon' / \Hseminorm{\beta}{\pi} \bigr)^{ 1 / \beta }$. Then $0 < \epsilon < \bigl( \epsilon'_0 / \Hseminorm{\beta}{\pi} \bigr)^{ 1 / \beta } = \epsilon_0$. Thus, by the Bressaud--Quas shadowing property of $(Y, g)$ (applied to $\cK$, $\kappa$, and $\epsilon$) from the hypothesis, there exists a periodic orbit $\O$ of $g$ of period $p < (1/\epsilon)^{\kappa} = ( \Hseminorm{\beta}{\pi} / \epsilon' )^{ \kappa' / 2 } \leq (1/\epsilon')^{\kappa'}$ contained in the $\epsilon$-neighborhood of $\cK$. The last inequality follows from $\Hseminorm{\beta}{\pi} \leq 1 / \epsilon'_0 < 1 / \epsilon'$. Then $\O' \= \pi( \O )$ is a periodic orbit of $f$ of period $p' \leq p < (1/ \epsilon' )^{\kappa'}$, contained in $N_d^{\Hseminorm{\beta}{\pi} \epsilon^{\beta} } (\cK') = N_d^{\epsilon'}(\cK')$.
\end{proof}

\begin{prop}  \label{p_BQ_iterate}
Let $f\: X \rightarrow X$ be a map on a compact metric space $(X,d)$ and $n\in\N$. Assume that $f$ is Lipschitz with respect to $d$. Denote $F \= f^n$. Then $(X,f)$ has the Bressaud--Quas shadowing property if and only if $(X, F)$ has the Bressaud--Quas shadowing property.
\end{prop}

\begin{proof}
The forward implication is straightforward. To show the backward implication, we fix a nonempty compact $f$-forward-invariant set $\cK \subseteq S^2$. Then $F(\cK) \subseteq \cK$. By the Bressaud--Quas shadowing property, for each $\kappa > 0$, there exists $\epsilon_0 (F, \kappa ) > 0$ such that for each $\epsilon \in ( 0, \epsilon_0 (F, \kappa) )$, we can find a periodic orbit $\O_F = \bigl\{x,\,F(x), \, \dots,\, F^{p-1}(x) \bigr\}$ of $F$ satisfying $p \= \card \O_F < ( 1 / \epsilon )^{\kappa}$ and $\O_F \subseteq N_d^{\epsilon} (\cK)$.

Fix an arbitrary number $\kappa > 0$. Set
\begin{equation}   \label{p_BQ_iterate_eps0}
\epsilon_0(f,\kappa) \= \min \bigl\{  \epsilon_0( F, \kappa / 4 ), \, ( \LIP_d(f) )^{-n}, \, n^{ - 2 / \kappa}  \bigr\} \in (0,1].
\end{equation}

Fix an arbitrary number $\epsilon \in ( 0 , \epsilon_0 ( f, \kappa ))$. Then we also have $\epsilon^2  \in ( 0, \epsilon_0 (F, \kappa / 4 ) )$. There exists a periodic orbit $\O_F$ of $F$ satisfying
\begin{equation}   \label{p_BQ_iterate_O_F}
\card \O_F < \bigl( 1 / \epsilon^2 \bigr)^{\kappa / 4} = ( 1 / \epsilon )^{\kappa / 2}
\qquad \text{and} \qquad
\O_F \subseteq N_d^{\epsilon^2} ( \cK ).
\end{equation}
Define $\O_f \= \bigcup_{i=1}^{n} f^i ( \O_F )$. Then $\O_f$ is a periodic orbit of $f$, and by (\ref{p_BQ_iterate_O_F}) and (\ref{p_BQ_iterate_eps0}), $\card \O_f \leq n \cdot \card \O_F < n (1/\epsilon)^{\kappa / 2} < (1/\epsilon)^\kappa$. On the other hand, it follows from (\ref{p_BQ_iterate_O_F}), Lemma~\ref{l_ETM_Lipschitz}, and (\ref{p_BQ_iterate_eps0}) that $\O_f \subseteq N_d^{ \epsilon^2 ( \LIP_d (f) ) ^{n} } ( \cK ) \subseteq N_d^{\epsilon} (\cK)$.
\end{proof}

\begin{theorem}[Bressaud--Quas \cite{BQ07}]   \label{t_BQ}
Let $\bigl(\Sigma_A^+, \sigma_A \bigr)$ be a one-sided subshift of finite type defined by a transition matrix $A$. Equip $\Sigma_A^+$ with a metric $d=d_\theta$, $\theta \in (0,1)$, as given in Subsection~\ref{subsct_SFT}. Then $\bigl(\Sigma_A^+, \sigma_A \bigr)$ has the Bressaud--Quas shadowing property.
\end{theorem}

X.~Bressaud and A.~Quas first established a form of the above closing lemma for the one-sided full shifts in \cite[Theorem~4]{BQ07}. A Bressaud--Quas Closing Lemma in the current form for hyperbolic homeomorphisms is formulated in \cite[Theorem~7.3]{BG19} and proved in \cite[Appendix~A.6]{BG19}. A proof of Theorem~\ref{t_BQ} above can be obtained verbatim the same from the proof in \cite[Appendix~A.6]{BG19} if we (1) replace $(Y,d)$ by $\bigl(\Sigma_A^+, \sigma_A \bigr)$, $f$ and $T$ by the shift map $\sigma_A$, ``expansive'' by ``forward-expansive'', (2) update the reference for Lipschitz shadowing lemma to \cite[Corollary~4.2.4]{PU10} that applies to one-sided subshift of finite type, (3) observe that (A.27) follows from \cite[Section~1.4]{Ki98} in our context, and finally (4) observe that in our context, for sufficiently small $\varepsilon>0$, each periodic $(\varepsilon,d_n,\sigma_A)$-pseudo-orbit is a periodic $(\varepsilon_n,d,\sigma_A)$-pseudo-orbit where $\varepsilon_n \= \theta^n \epsilon$ and $n\in\N$.

Combining Lemma~\ref{l_CexistsBM}, Proposition~\ref{p_TileSFT}, Lemma~\ref{l_ETM_Lipschitz}, Theorem~\ref{t_BQ}, Proposition~\ref{p_BQ_factor}, and Propostion~\ref{p_BQ_iterate}, we immediately conclude that expanding Thurston maps have the Bressaud--Quas shadowing property.

\begin{theorem}[Bressaud--Quas closing lemma for expanding Thurston maps]  \label{t_BQ_for_ETM}
Let $f$, $\CC$, $d$ satisfy the Assumptions in Section~\ref{sct_Assumptions}. Then $(S^2,f)$ has the Bressaud--Quas shadowing property. More precisely, for every nonempty compact $f$-forward-invariant set $\cK \subseteq S^2$ and every $\kappa > 0$, there exists a number $\epsilon_0 > 0$ such that for each $\epsilon \in (0, \epsilon_0)$, we can find a periodic orbit $\O$ of period $p < (1/\epsilon)^{\kappa}$ contained entirely in the $\epsilon$-neighborhood $N_d^{\epsilon} (\cK)$ of $\cK$.
\end{theorem}

\subsection{Proof of Lemma~\ref{l_bound_by_gap}}    \label{subsct_Local_Closing_Lemma_ETM_proof}

We combine our Local Anosov Closing Lemma and Bressaud--Quas closing lemma to establish Lemma~\ref{l_bound_by_gap}.

\begin{proof}[Proof of Lemma~\ref{l_bound_by_gap}]
Fix the set $\cK$, real numbers $r>0$, $\theta>0$, $\alpha \in (0,1]$, and $\tau > 0$ as in the statement of this lemma.

We will construct recursively a finite sequence $\O_0,\,\O_1,\,\dots,\,\O_m$ of periodic orbits, the last of which satisfies (\ref{e_l_bound_by_gap}). For each $i\in \{ 0,\, 1,\, \dots,\, m\}$, we denote
\begin{equation}  \label{e_Pf_l_bound_by_gap_Sigma_Delta_i}
\Sigma_i \= \sum_{x\in\O_i} d(x,\cK)^\alpha \text{ and } 
\Delta_i \=  \Delta_{r,\,\theta} ( \O_i ).
\end{equation}

The following constants defined below will be needed in the proof:
\begin{align}
\eta        &\=   \min\{ d ( \cK, x ) / 2 : x \in \crit f \} > 0,      \label{e_Pf_l_bound_by_gap_eta} \\
D            &\=   1 +  \frac{  \beta_0^\alpha \theta^{-\alpha} }{1- \Lambda^{-\alpha} } \cdot \frac{1}{\tau} >1,      \label{e_Pf_l_bound_by_gap_D} \\
c             &\=   1 / ( 2 + 2\log_2(D) ) \in (0, 1/2),    \label{e_Pf_l_bound_by_gap_c} \\
\kappa    &\= c \alpha > 0,  \label{e_Pf_l_bound_by_gap_kappa} \\
\epsilon  &\= \min \bigl\{  1, \, \epsilon_0 / 2, \,\eta^2, \, \tau^{ 2 / \alpha } \min \bigl\{ r^2, \,  \theta^2 \delta_3^2    \bigr\}  \bigr\} \in (0,1],  \label{e_Pf_l_bound_by_gap_epsilon}
\end{align}
where $\delta_3 \in (0,1)$ is a constant depending only on $f$, $\CC$, $d$, and $\cK$ from Lemma~\ref{l_Local_Closing_Lemma} (applied to $f$, $\CC$, $d$, and $\eta$), $\beta_0>1$ is a constant depending only on $f$, $\CC$, and $d$ from Lemma~\ref{l_Local_Closing_Lemma_l}, and $\epsilon_0>0$ is a constant from Theorem~\ref{t_BQ_for_ETM} applied to $f$, $\CC$, $d$, $\cK$, and $\kappa$. Thus, $\epsilon_0$ depends only on $f$, $\CC$, $d$, $\cK$, and $\alpha$.

Applying Theorem~\ref{t_BQ_for_ETM} to $f$, $\CC$, $d$, and $\cK$ with $\kappa$ and $\epsilon$ defined above, we can find a periodic orbit $\O_0$ of $f$ of period
\begin{equation}    \label{e_Pf_l_bound_by_gap_p_0}
p_0 < \epsilon^{ - c \alpha} 
\end{equation}
contained in $N_d^\epsilon (\cK)$. If $\O_0$ satisfies $\Sigma_0 \leq \tau \Delta_0^\alpha$, then $\O_0$ is the orbit $\O$ we are looking for. So, without loss of generality, we may assume that 
\begin{equation}   \label{e_Pf_l_bound_by_gap_O_0}
\Sigma_0 > \tau \Delta_0^\alpha.
\end{equation}
Note that by (\ref{e_Pf_l_bound_by_gap_O_0}), (\ref{e_Pf_l_bound_by_gap_Sigma_Delta_i}), (\ref{e_Pf_l_bound_by_gap_p_0}), (\ref{e_Pf_l_bound_by_gap_epsilon}), and (\ref{e_Pf_l_bound_by_gap_c}), we have
\begin{equation}   \label{e_Pf_l_bound_by_gap_Delta_0}
\tau \Delta_0^\alpha < \Sigma_0 \leq p_0 \epsilon^\alpha < \epsilon^{ - c \alpha + \alpha}  \leq \epsilon^{\alpha / 2}   \leq \eta^\alpha.
\end{equation}

\smallskip

\emph{Base step.} We have found a periodic orbit $\O_0$ with $\tau \Delta_0^\alpha < \Sigma_0   \leq \epsilon^{\alpha / 2} \leq \eta^\alpha$.

\smallskip

\emph{Recursive step.} Assume that we have found a periodic orbit $\O_{k-1}$, for some $k\in\N$, such that $p_{k-1} \= \card \O_{k-1} \leq 2^{-k+1} p_0$, $\Sigma_{k-1} \leq D^{k-1} \Sigma_0$, and $\tau \Delta_{k-1}^\alpha < \Sigma_{k-1} < \epsilon^{\alpha / 2} \leq \eta^\alpha$.

Then by the recursion hypothesis and (\ref{e_Pf_l_bound_by_gap_epsilon}),
\begin{equation}  \label{e_Pf_l_bound_by_gap_Delta_k-1}
 \Delta_{k-1} < \tau^{-1/\alpha} \epsilon^{1/2}  \leq \min\{ r, \, \theta \delta_3 \}.
\end{equation}
It follows that $\Delta_{k-1} = \theta \cdot \Delta ( \O_{k-1} )$ and $p_{k-1}  = \card \O_{k-1}$ is at least $2$, since otherwise we would have $\Delta_{k-1} = r$ (see (\ref{e_Def_gap_r_theta})). We choose distinct points $x,\,x' \in \O_{k-1}$ with the properties that $d(x,x') = \Delta ( \O_{k-1} ) = \theta^{-1} \Delta_{k-1}$ and $x' = f^n(x)$ for some positive integer $n \leq p_{k-1} / 2$.

By (\ref{e_Pf_l_bound_by_gap_Delta_k-1}), $d(x,x') = \theta^{-1} \Delta_{k-1} \leq \delta_3$. On the other hand, $ \min\{ d( \O_{k-1}, z ) : z\in \crit f \} \geq \eta$ by $\Sigma_{k-1} \leq \eta^\alpha$, (\ref{e_Pf_l_bound_by_gap_Sigma_Delta_i}), and (\ref{e_Pf_l_bound_by_gap_eta}). Then by Lemma~\ref{l_Local_Closing_Lemma} and the recursion hypothesis, we get a periodic point $y$ of $f$ of period $p_k \leq n \leq p_{k-1} / 2 \leq 2^{-k} p_0$. Let $\O_k$ be the orbit of $y$ under $f$. Finally, to verify the recursion hypothesis for $\O_k$, we get from (\ref{e_Pf_l_bound_by_gap_Sigma_Delta_i}), $\alpha \in (0,1]$, Lemma~\ref{l_Local_Closing_Lemma}, $d(x,x') =\theta^{-1} \Delta_{k-1}$, the recursion hypothesis for $\O_{k-1}$, (\ref{e_Def_gap_r_theta}), (\ref{e_Pf_l_bound_by_gap_Delta_k-1}), and (\ref{e_Pf_l_bound_by_gap_D}),
\begin{align*}
\Sigma_k
& \leq  \sum_{i=0}^{n-1}  d \bigl( f^i(y), \cK \bigr)^\alpha
    \leq  \sum_{i=0}^{n-1}    \bigl( d \bigl( f^i(x), \cK \bigr)^\alpha +   d \bigl( f^i(x), f^i(y) \bigr)^\alpha \bigr) \\
& \leq  \Sigma_{k-1} +  \beta_0^{\alpha} d(x,x')^\alpha \sum_{i=0}^{n-1} \Lambda^{- \alpha (n-i) } 
    \leq  \Sigma_{k-1} + \frac{ \beta_0^{\alpha} \theta^{-\alpha} }{ 1- \Lambda^{ - \alpha } } \Delta_{k-1}^\alpha \\
& \leq \biggl( 1 + \frac{ \beta_0^{\alpha} \theta^{-\alpha} }{ 1- \Lambda^{ - \alpha } } \cdot \frac{1}{\tau} \biggr)  \Sigma_{k-1} 
      =   D  \Sigma_{k-1} 
    \leq D^{k}  \Sigma_0.
\end{align*}
Since $k \leq \log_2 (p_0)$, we get from (\ref{e_Pf_l_bound_by_gap_Delta_0}), (\ref{e_Pf_l_bound_by_gap_p_0}), (\ref{e_Pf_l_bound_by_gap_c}), and (\ref{e_Pf_l_bound_by_gap_epsilon}),
\begin{equation}   \label{e_Pf_l_bound_by_gap_Sigma_k}
\Sigma_k 
\leq  \Sigma_0  p_0^{\log_2(D)} 
\leq \epsilon^\alpha  p_0^{1 + \log_2(D)} 
<  \epsilon^{ ( 1 - c ( 1 + \log_2(D) )) \alpha}
= \epsilon^{\alpha / 2}
\leq \eta^\alpha.
\end{equation}

Finally, if $\Sigma_k \leq \tau \Delta_k^\alpha$, then $\O_k$ is the periodic orbit we are looking for. So without loss of generality, we can assume that $\tau \Delta_k^\alpha < \Sigma_k$, and this completes the recursive step.

\smallskip

If this procedure continues until an orbit $\O_m$ which consists of exactly one fixed point of $f$, then by (\ref{e_Pf_l_bound_by_gap_Sigma_k}), (\ref{e_Pf_l_bound_by_gap_epsilon}), (\ref{e_Def_gap_r_theta}), and (\ref{e_Def_gap}),
\begin{equation*}
\Sigma_m < \epsilon^{\alpha / 2}  \leq \tau r^\alpha = \tau \Delta_m^\alpha.
\end{equation*}
Hence, this orbit $\O_m$ is what we are looking for.
\end{proof}

\section{Proofs of the first parts of Theorems~\ref{t_Density_Thurston} and~\ref{t_Density_Lattes}}  \label{sct_Holder_Density}

In this section, we first give a proof of Theorem~\ref{t_Density_Lattes}~(i), then describe the modifications needed on this proof to establish Theorem~\ref{t_Density_Thurston}~(i). The remaining parts of these theorems will be proved in Section~\ref{sct_lip}.

In the proof of Theorem~\ref{t_Density_Lattes}~(i), we show that for an arbitrary $\alpha$-H\"older continuous potential $\varphi \in \Lip(\wh\C, \sigma^{\alpha})$ with respect to the chordal metric $\sigma$, any perturbation of the form $\varphi' = \varphi - \epsilon \sigma (\cdot, \O)^\alpha$, with $\epsilon>0$ sufficiently small, belongs to $\sP^\alpha(\wh\C,\sigma)$, where $\O$ is some special periodic orbit produced by our local closing lemma away from critical points. The quantitative analysis is, however, carried out in the canonical orbifold metric $d$ on the related potentials $\wt\varphi$ and $\psi$, which are $\alpha$-H\"older continuous with respect to $d$ but not with respect to $\sigma$. In the case of Latt\`es maps, the canonical orbifold metric is also a visual metric. The technical parts are (1) to quantitatively avoid critical points $\crit f$ where the combinatorics are more involved in order to apply our local closing lemma as well as uniform local expansion property and (2) to quantitatively avoid postcritical points $\post f$ where the conversion between $d$ and $\sigma$ is more involved. In fact, by applying various properties of the canonical orbifold metric and considering the orbifold ramification function, we get that the canonical orbifold metric and the chordal metric are ``locally comparable away from postcritical points''. We know that the identity map on $\wh\C$ between these two metrics is never bi-Lipschitz (see \cite[Appendix~A.10]{BM17}).

\begin{proof}[Proof of Theorem~\ref{t_Density_Lattes}~(i)]
J.~Bochi and the second-named author demonstrated in \cite[Proposition~1]{BZ15} that for every continuous map $T\: X \rightarrow X$ on a compact metric space $(X,d)$, the set $\Lock (X,d)$ is equal to the interior of $\sP(X) \cap \Lip(X,d)$ and $\Lock (X,d)$ is dense in $\sP(X) \cap \Lip(X,d)$ (with respect to the Lipschitz norm). Thus, it suffices to show that, in our setting, $\Lock (\wh\C, \sigma^\alpha)$ is dense in $\sP( \wh\C ) \cap \Lip(\wh\C, \sigma^\alpha)$.

Recall that Latt\`es maps are, in particular, (rational) expanding Thurston maps. Let $d$ be the canonical orbifold metric of $f$ on $\wh\C$. Then $d$ is a visual metric (see Remark~\ref{r_Canonical_Orbifold_Metric}). Let $\Lambda>1$ be the expansion factor of $d$ (under $f$). Let  $\CC \subseteq S^2$ be a Jordan curve containing $\post f$ satisfying the condition that $f^{n_\CC}(\CC) \subseteq \CC$ for some integer $n_\CC \in \N$ (see Lemma~\ref{l_CexistsBM}). Here $n_\CC$ is assumed to be the smallest of such integers associated to $\CC$.

In this proof, for each periodic orbit $\O$ of $f$ we define a measure $\mu_{\O}$ supported on $\O$ as
\begin{equation}   \label{e_Pf_t_Density_Lattes_delta_measure}
\mu_{\O} \= \frac{1}{ \card \O} \sum_{x\in \O} \delta_x \in \MMM(\wh\C, f).
\end{equation}

By inequality (A.43) in \cite[Appendix~A.10]{BM17}, there exist constants $C_3 \geq 1$ and $\eta \in (0,1)$ such that
\begin{equation}   \label{e_Pf_t_Density_Lattes_metrics_compare_global}
C_3^{-1} \sigma (x,y) \leq d(x,y) \leq C_3 \sigma(x,y)^\eta
\end{equation}
for all $x,\,y \in \wh\C$. Thus,
\begin{equation}  \label{e_Pf_t_Density_Lattes_Holder_spaces_compare_global}
\Lip( \wh\C, \sigma^{\alpha}) 
\subseteq \Lip ( \wh\C, d^{\alpha} ) 
\subseteq \Lip ( \wh\C, \sigma^{\eta \alpha} ) .
\end{equation}

Fix an arbitrary $\varphi \in \Lip(\wh\C, \sigma^{\alpha}) \subseteq \Lip(\wh\C, d^{\alpha})$ with no $\varphi$-maximizing measure in $\Mmax(f, \,\varphi)$ supported on a periodic orbit of $f$.

Let $u_\varphi \in \Lip (\wh\C, d^{\alpha})$ be a calibrated sub-action for $\varphi$ and $f$ (i.e., a fixed point of $\RR_{\overline\varphi}$) from Proposition~\ref{p_calibrated_sub-action_exists}. Since $f$ is Lipschitz with respect to $d$ (see Lemma~\ref{l_ETM_Lipschitz}), define
\begin{equation}   \label{e_Pf_t_Density_Lattes_tphi}
\wt{\varphi} \= \varphi - Q(f,\varphi) + u_\varphi - u_\varphi  \circ f \in \Lip(\wh\C , d^{\alpha}) ,
\end{equation}
where $Q(f,\varphi)$ is defined in (\ref{e_Def_beta}). Then 
\begin{equation}    \label{e_Pf_t_Density_Lattes_wtvarphi}
\wt\varphi (x) \leq 0 \qquad \text{ for all }x\in\wh\C,
\end{equation}
and the set 
\begin{equation} \label{e_Pf_t_Density_Lattes_K}
\cK \= \bigcap_{j=0}^{+\infty} f^{-j} \bigl( \wt\varphi^{-1} ( 0 ) \bigr)
\end{equation}
is a nonempty compact $f$-forward-invariant set (see Lemma~\ref{l_Bousch_Op_normalizing_potential}~(ii)).

If $\cK$ contains a periodic orbit $\O$ of $f$, then by Lemma~\ref{l_Bousch_Op_normalizing_potential}, the measure $\mu_{\O}$ defined in (\ref{e_Pf_t_Density_Lattes_delta_measure}) satisfies $\int\! \wt\varphi \, \mathrm{d}\mu_{\O} = 0 = Q(f, \wt\varphi )$, and consequently $\mu_{\O} \in \Mmax( f, \, \wt\varphi ) = \Mmax(f, \,\varphi)$. Thus, from our assumption on $\varphi$, we get that $\cK$ does not contain any periodic orbit of $f$. In particular, $\cK \cap (\crit f \cup \post f) = \emptyset$.

Recall from Subsection~\ref{subsct_ThurstonMap} that $\crit f$ and $\post f$ are finite sets. Denote a constant $\iota$ depending only on $f$, $d$, and $\varphi$ as the following:
\begin{equation}   \label{e_Pf_t_Density_Lattes_iota}
\iota \= \min \{ d(x,\cK) : x \in \crit f \cup \post f \} >0.
\end{equation}

Note that the ramification function $\alpha_f (z) = 1$ for all $z\in\widehat{\C}\setminus \post f$ (see Definition~\ref{d_Ramification_Fn}). Recall the notion of \emph{singular conformal metrics} from \cite[Appendix~A.1]{BM17}. By Proposition~A.33 and the discussion proceeding it in \cite[Appendix~A.10]{BM17}, the following statements hold:
\begin{enumerate}
\smallskip
\item[(1)] The canonical orbifold metric $d$ is a singular conformal metric with a conformal factor $\rho$ that is continuous and positive everywhere except at the points in $\supp(\alpha_f) \subseteq \post f$ (see (\ref{e_Def_supp})).

\smallskip
\item[(2)] $d(z_1,z_2) = \inf\limits_{\gamma} \int_\gamma \! \rho \,\mathrm{d} \sigma$, where the infimum is taken over all $\sigma$-rectifiable paths $\gamma$ in $\widehat{\C}$ joining $z_1$ and $z_2$.  

\smallskip
\item[(3)] For each $z\in\widehat{\C} \setminus \supp(\alpha_f)$, there exists a neighborhood $U_z \subseteq \widehat{\C}$ containing $z$ and a constant $C_z \geq 1$ such that $C_z^{-1} \leq \rho(u) \leq C_z$ for all $u \in U_z$. 
\end{enumerate}

It follows from the above statements, (\ref{e_Pf_t_Density_Lattes_metrics_compare_global}), and a compactness argument that there exists a constant $C_4 > 1$ depending only on $f$, $d$, and $\cK$ such that
\begin{equation}  \label{e_Pf_t_Density_Lattes_metrics_compare_away_crit}
C_4^{-1} \sigma (x, y) \leq d(x, y) \leq C_4 \sigma (x, y) \quad \text{ for all } x,\,y \in \wh\C \setminus N_d^{\iota / 2} (\post f).
\end{equation}

\smallskip
\emph{Claim~1.} There exists $\delta_0 \in (0,1]$ depending only on $f$, $\CC$, and $d$ such that for all $n\in\N$ and $x,\,y \in \wh\C$, if for each $j\in\{0,\,1,\,\dots,\,n-1\}$ we have $f^j(x) \notin N_d^{\iota / 2} ( \crit f )$ and $d \bigl( f^j(x), f^j(y) \bigr) < C_2 \delta_0$, then for each $i\in\{0,\,1,\,\dots,\,n\}$, we have
\begin{equation*} 
C_2 \Lambda^i d(x,y) \leq d \bigl( f^i(x), f^i(y) \bigr) \leq C_2^{-1} \Lambda^i d(x,y),
\end{equation*}
where $C_2 \in (0, 1)$ is a constant depending only on $f$, $\CC$, and $d$ from Lemma~\ref{l_Loc_Inject_Expansion_Away_from_Crit}.

\smallskip

Claim~1 follows immediately from Lemma~\ref{l_Loc_Inject_Expansion_Away_from_Crit}.

\smallskip

Fix an arbitrary positive real number 
\begin{equation} \label{e_Pf_t_Density_Lattes_epsilon}
\epsilon < \min \{ \iota / 8, \, 1\} .
\end{equation}
Define constants
\begin{align}
C_5    &\= \frac{ ( \Hseminorm{d^\alpha }{\wt\varphi}  + C_3 )  C_4^\alpha  }{ C_2^4 ( 1 - \Lambda^{-\alpha} ) } ,  \label{e_C12}  \\
\lambda   &\= C_2^2 \min \biggl\{ \frac{1}{ 8 \LIP_d(f) }, \,  \frac{ C_2 \delta_0 }{ \iota }   \biggr\} < \frac18  , \label{e_Pf_t_Density_Lattes_lambda}  \\
\tau          &\= \min\bigg\{ 1, \, \frac{ \epsilon }{ \Hseminorm{d^\alpha }{\wt\varphi} C_4^{\alpha}   },  \,  \frac{ \epsilon C_4^{-\alpha} }{  ( 1 + C_5 \epsilon^{-1} )   \Hseminorm{d^\alpha }{\wt\varphi} } \biggr\} \leq 1 , \label{e_Pf_t_Density_Lattes_tau}   \\
\theta       &\=  \min \biggl\{   \frac{ C_2^2 }{3 \LIP_d(f)}, \, \frac13 \Lambda^{-1} \biggr\} < \frac13  . \label{e_Pf_t_Density_Lattes_theta} 
\end{align}

By Lemma~\ref{l_bound_by_gap}, there exists a periodic orbit $\O$ of $f$ of a period $p \= \card \O$ satisfying
\begin{equation}   \label{e_Pf_t_Density_Lattes_bound_by_gap}
\sum_{x\in\O} d( x, \cK )^\alpha \leq \tau \cdot ( \Delta_{\lambda \iota,\,\theta} ( \O ) )^{\alpha}.
\end{equation}
By (\ref{e_Pf_t_Density_Lattes_iota}), (\ref{e_Pf_t_Density_Lattes_bound_by_gap}), (\ref{e_Pf_t_Density_Lattes_tau}), (\ref{e_Def_gap_r_theta}), and (\ref{e_Pf_t_Density_Lattes_lambda}), we have
\begin{equation}    \label{e_Pf_t_Density_Lattes_O_crit_distance}
\min \{ d (x, \crit f \cup \post f) : x \in \O \} 
\geq \iota - \max_{x\in \O } d(x,\cK ) 
\geq \iota - \lambda \iota 
\geq \frac{7}{8} \iota. 
\end{equation}

Define potentials
\begin{align}
\varphi' &\= \varphi - \epsilon \sigma (\cdot, \O)^\alpha \in \Lip(\wh\C, \sigma^{\alpha}) \subseteq \Lip( \wh\C, d^{\alpha} )  \text{ and}  \label{e_Pf_t_Density_Lattes_varphi'} \\
\psi            &\= \wt\varphi -  \epsilon  \sigma (\cdot, \O)^\alpha  =  \varphi' - Q(f,\varphi) + u_\varphi - u_\varphi \circ f \in  \Lip( \wh\C, d^{\alpha} ). \label{e_Pf_t_Density_Lattes_psi} 
\end{align}
Here (\ref{e_Pf_t_Density_Lattes_varphi'}) follows from (\ref{e_Pf_t_Density_Lattes_Holder_spaces_compare_global}) and the hypothesis that $\varphi \in \Lip( \wh\C, \sigma^{\alpha})$, whereas (\ref{e_Pf_t_Density_Lattes_psi}) follows from (\ref{e_Pf_t_Density_Lattes_varphi'}), the fact that $f$ is Lipschitz with respect to $d$, and $u_\varphi \in \Lip( \wh\C, d^{\alpha})$.

Note that by Lemma~\ref{l_Bousch_Op_normalizing_potential}~(i) and (\ref{e_Pf_t_Density_Lattes_psi}),
\begin{equation}  \label{e_Pf_t_Density_Lattes_Mmax_identical}
\Mmax (f, \, \varphi')  = \Mmax(f, \, \psi).
\end{equation} 

\smallskip
\emph{Claim~2.} The measure $\mu_{\O}$ defined in (\ref{e_Pf_t_Density_Lattes_delta_measure}) is in $\Mmax (f, \, \psi)$, i.e., $Q(f,\psi) = \gamma$, where
\begin{equation}  \label{e_Pf_t_Density_Lattes_gamma}
\gamma \= \int\! \psi \, \mathrm{d}\mu_\O 
= \frac{1}{ p } \sum_{x\in\O} \psi (x) 
= \frac{1}{ p } \sum_{x\in\O} \wt\varphi (x)
< 0.
\end{equation}
We observe that the last equality follows from (\ref{e_Pf_t_Density_Lattes_psi}), whereas the last inequality follows from (\ref{e_Pf_t_Density_Lattes_wtvarphi}), (\ref{e_Pf_t_Density_Lattes_K}), and our assumption that $\cK$ contains no periodic orbit of $f$.

\smallskip

Assuming that Claim~2 holds. Then Claim~2 and (\ref{e_Pf_t_Density_Lattes_psi}) imply $\psi \in \sP^\alpha(\wh\C,d)$. Thus, by (\ref{e_Pf_t_Density_Lattes_Mmax_identical}) and (\ref{e_Pf_t_Density_Lattes_varphi'}), $\varphi' \in \sP^\alpha(\wh\C,\sigma)$. On the other hand, by (\ref{e_Pf_t_Density_Lattes_varphi'}),
\begin{equation*}
\Hnorm{\sigma^\alpha}{\varphi - \varphi'}{ \wh\C } 
= \Hnorm{\sigma^\alpha}{ \epsilon \sigma ( \cdot, \O )^\alpha }{ \wh\C }  
\leq \epsilon ( 1 + \diam_{\sigma} (\wh\C)^\alpha ).
\end{equation*}
Since $\epsilon$ from (\ref{e_Pf_t_Density_Lattes_epsilon}) can be chosen arbitrarily small and $\varphi' \in \Lip( \wh\C, \sigma^{\alpha})$ (see (\ref{e_Pf_t_Density_Lattes_varphi'})), we finally conclude that
$\varphi$ is in the closure of $\sP^\alpha(\wh\C,\sigma)$ in $\Lip(\wh\C, \sigma^{\alpha})$ with respect to the H\"older norm.

\smallskip
Hence, it suffices to establish Claim~2. By the definition of $\gamma$ in (\ref{e_Pf_t_Density_Lattes_gamma}), it suffices to show that $Q(f,\psi) \leq \gamma$. Fix an arbitrary point $y \in S^2$. We analyze below the value of $\psi(y)$ according to the location of $y$.

We first observe that by (\ref{e_Pf_t_Density_Lattes_gamma}) and (\ref{e_Pf_t_Density_Lattes_bound_by_gap}), we have
\begin{equation}  \label{e_Pf_t_Density_Lattes_p_gamma}
p \abs{\gamma} 
\leq \sum_{x\in\O} \abs{ \wt\varphi (x) - 0 }
\leq \sum_{x\in\O} \Hseminorm{d^\alpha }{\wt\varphi}  d(x,\cK)^\alpha
\leq \Hseminorm{d^\alpha }{\wt\varphi}  \tau \cdot ( \Delta_{\lambda \iota,\,\theta} ( \O ) )^{\alpha}.
\end{equation}

By (\ref{e_Pf_t_Density_Lattes_p_gamma}), (\ref{e_Pf_t_Density_Lattes_tau}), (\ref{e_Def_gap_r_theta}), (\ref{e_Pf_t_Density_Lattes_lambda}), and $\LIP_d(f)>1$ (see Lemma~\ref{l_ETM_Lipschitz}), we have 
\begin{align}  \label{e_Pf_t_Density_Lattes_rho}
\rho  \=&  C_4 \epsilon^{-1/\alpha} \abs{\gamma}^{1/\alpha}  
     \leq   C_4 ( \Hseminorm{d^\alpha }{\wt\varphi}  \tau / \epsilon )^{1/\alpha}  \Delta_{\lambda \iota, \, \theta} ( \O )  \\
     \leq  &  \Delta_{\lambda \iota,\, \theta} ( \O )  \leq \lambda \iota \leq C_2^2 \min\{ C_2 \delta_0, \, \iota / 8 \} < \iota / 8. \notag
\end{align}

Let $U \= \overline{N}_d^{\rho} (\O) \subseteq \overline{N}_d^{\iota / 8} (\O)$. 

If $y \notin U \cup N_d^{\iota / 2} ( \crit f \cup \post f )$, then by (\ref{e_Pf_t_Density_Lattes_metrics_compare_away_crit}),  $\sigma (y,\O) \geq C_4^{-1} d(y, \O) > C_4^{-1} \rho$, and consequently by (\ref{e_Pf_t_Density_Lattes_psi}), (\ref{e_Pf_t_Density_Lattes_wtvarphi}), (\ref{e_Pf_t_Density_Lattes_rho}), and (\ref{e_Pf_t_Density_Lattes_gamma}), we have
\begin{equation}  \label{e_Pf_t_Density_Lattes_outside_U}
\psi(y) 
=      \wt\varphi(y) - \epsilon  \sigma (y, \O )^\alpha 
\leq   -  \epsilon \sigma (y, \O )^\alpha
<       -  \epsilon C_4^{-\alpha} \rho^\alpha = - \abs{\gamma} = \gamma.
\end{equation}

If $y \in N_d^{\iota / 2} ( \crit f \cup \post f)$, then we will choose a point $y' \notin U \cup N_d^{\iota / 2} ( \crit f \cup \post f )$ satisfying $\sigma (y', \O ) < \sigma (y, \O )$ in the following way. We first fix a point $z\in \O$ with $d(y,z) = d(y, \O)$. Let $\Gamma$ be the geodesic arc (with respect to the chordal metric $\sigma$) connecting $y$ and $z$ with the minimal length. By the definition of $\iota$ in  (\ref{e_Pf_t_Density_Lattes_iota}) and the fact that  $U \subseteq \overline{N}_d^{\iota / 8} (\O)$, it is clear that we can choose a point $y'\in \Gamma$ satisfying $y' \notin U \cup N_d^{\iota / 2} ( \crit f  \cup \post f )$. Then $\sigma (y', \O ) < \sigma (y, \O )$. Hence, it follows from (\ref{e_Pf_t_Density_Lattes_outside_U}) that 
\begin{equation}  \label{e_Pf_t_Density_Lattes_near_crit_f}
\psi(y) 
=      \wt\varphi(y) - \epsilon  \sigma (y, \O )^\alpha 
\leq  -  \epsilon \sigma (y, \O )^\alpha
<      -  \epsilon \sigma (y', \O )^\alpha
<      \gamma.
\end{equation}

Recall a characterization of $Q(f,\psi)$ (see \cite[Equation~(30)]{Boc19}):
\begin{equation*}
Q(f, \psi )
= \sup_{x\in \wh\C}  \inf_{n\in\N} \frac{ S_n \psi (x) }{n}.
\end{equation*}
This follows immediately from a lemma of Y.~Peres \cite[Lemma~2]{Pe88}. See \cite[Theorem~A.3]{Mo13} for a proof of a generalization of the above characterization.

Thus, in order to show that $Q(f,\psi) \leq \gamma$, by the estimates (\ref{e_Pf_t_Density_Lattes_outside_U}) for $y \notin U \cup N_d^{\iota / 2} ( \crit f \cup \post f )$, (\ref{e_Pf_t_Density_Lattes_near_crit_f}) for $y \in N_d^{\iota / 2} ( \crit f \cup \post f )$, and the equality $\frac{1}{p} S_p \psi (y)  = \gamma$ for $y\in \O$, it suffices to prove that if $y\in U \setminus \O$ then there exists $N \in \N$ such that $S_{N} \psi (y) \leq N \gamma$.

To this end, we assume now that $y\in U \setminus \O$. So $0< d( y, \O ) \leq \rho$. By (\ref{e_Pf_t_Density_Lattes_rho}), (\ref{e_Def_gap_r_theta}), and (\ref{e_Pf_t_Density_Lattes_theta}), $\rho \leq \Delta_{\lambda \iota,\, \theta} (\O) \leq \frac13 \Delta(\O)$. So by (\ref{e_Def_gap}), there is a unique point $z \in \O$ which is closest to $y$ among points in the periodic orbit $\O$.

Let $N \in \N \cup \{ +\infty \}$ be the smallest positive integer satisfying 
\begin{equation}  \label{e_Pf_t_Density_Lattes_N}
d \bigl( f^{N-1}(y), f^{N-1}(z) \bigr) \geq C_2^{-2} \Delta_{\lambda \iota , \, \theta} (\O),
\end{equation}
or set $N = +\infty$ if such a positive integer does not exist. Note that $N\geq 2$ since $d(y, \O) \leq \Delta_{\lambda \iota,\, \theta} (\O)$. We use the convention that $N-1 = N$ when $N = +\infty$. Then by (\ref{e_Def_gap_r_theta}) and (\ref{e_Pf_t_Density_Lattes_lambda}), for all $j\in \N_0$ with $j < N-1$,
\begin{equation}   \label{e_Pf_t_Density_Lattes_N_less} 
 d \bigl( f^j (y) , \O \bigr) 
\leq   d \bigl( f^j (y) , f^j(z) \bigr) 
<         \frac{ \Delta_{\lambda \iota, \, \theta} (\O) }{C_2^2}
\leq     \frac{ \lambda \iota  }{C_2^2}
\leq   \min \biggl\{ C_2 \delta_0,\, \frac{ \iota }{ 8 \LIP_d(f) }   \biggr\} .  
\end{equation} 
In particular, $d \bigl( f^{j+1} (y), \O \bigr) \leq d \bigl( f^{j+1}(y), f^{j+1}(z) \bigr) \leq \iota / 8$. Thus, by (\ref{e_Pf_t_Density_Lattes_iota}), (\ref{e_Pf_t_Density_Lattes_bound_by_gap}), (\ref{e_Pf_t_Density_Lattes_tau}), (\ref{e_Def_gap_r_theta}), and Lemma~\ref{l_ETM_Lipschitz}, for all $i\in \N_0$ with $i < N$, we get that for each $x \in \crit f \cup \post f$,
\begin{align*}   
d \bigl( x, f^i(y) \bigr) 
&\geq  d(x, \cK) - d \bigl(  f^i(y) , \cK \bigr)   
\geq \iota - \min\bigl\{  d \bigl( f^i(y) , o \bigr) + d ( o, \cK )  :  o \in \O \bigr\}  \\
&\geq \iota - d \bigl( f^i(y) , \O \bigr) - \max \{ d ( o, \cK ) :  o \in \O \} 
\geq \frac{ 7\iota }{ 8 } - \Delta_{\lambda \iota, \, \theta} (\O)
\geq \frac{ 7\iota }{ 8 }  - \lambda \iota,
\end{align*} 
and therefore,
\begin{equation}   \label{e_Pf_t_Density_Lattes_away_from_crit}
d \bigl( \crit f \cup \post f, f^i(y) \bigr)  > \frac{ \iota }{ 2 }   .
\end{equation}
It follows from Claim~1, (\ref{e_Pf_t_Density_Lattes_N_less}), and (\ref{e_Pf_t_Density_Lattes_away_from_crit}) that for all $i\in \N_0$,
\begin{equation}  \label{e_Pf_t_Density_Lattes_yz_orbit_distance}
C_2 \Lambda^i d(y,z) \leq d \bigl( f^i(y), f^i(z) \bigr) \leq C_2^{-1} \Lambda^i d(y,z)    \quad \text{ provided } i < N.
\end{equation}
Now, by the definition of $N$ and the lower bound in (\ref{e_Pf_t_Density_Lattes_yz_orbit_distance}), we know that $N$ is finite. On the other hand, by (\ref{e_Def_gap_r_theta}), (\ref{e_Pf_t_Density_Lattes_theta}), Lemma~\ref{l_ETM_Lipschitz}, and the definition of $N$ in (\ref{e_Pf_t_Density_Lattes_N}), we get that for each $j \in \{ 0, \, 1, \, \dots, \, N-2 \}$,
\begin{equation*}
d \bigl( f^j(y), f^j(z) \bigr) < \frac{ \Delta_{\lambda \iota, \, \theta} (\O) }{ C_2^2 }
\leq  \frac{  \theta \Delta(\O)  }{ C_2^2 }
\leq  \frac{ \Delta ( \O )  }{3 \LIP_d(f)}
\leq \frac{ \Delta ( \O )  }{3 }  .
\end{equation*}
In particular, $d \bigl( f^{N-1}(y), f^{N-1}(z) \bigr) \leq \LIP_d(f) \cdot d \bigl( f^{N-2}(y), f^{N-2}(z) \bigr) \leq \Delta(\O) / 3$. Thus, by (\ref{e_Def_gap}) and (\ref{e_Pf_t_Density_Lattes_yz_orbit_distance}), for each $i\in\{0, \, 1, \, \dots, \, N - 1 \}$,
\begin{equation}  \label{e_Pf_t_Density_Lattes_y_orbit_distance_from_O}
 d \bigl( f^i(y), \O \bigr)  = d \bigl( f^i(y), f^i(z) \bigr) \in \bigl[  C_2 \Lambda^i d(y,z)   ,   \,  C_2^{-1} \Lambda^i d(y,z) \bigr].
\end{equation}

We now proceed to show that $S_N \psi (y) \leq N \gamma$.

Let $n \in \N$ be the smallest positive integer satisfying 
\begin{equation}  \label{e_Pf_t_Density_Lattes_n_property}
d(f^n(y), f^n(z)) > \rho  C_2^{-2}. 
\end{equation}
Such an integer $n$ exists and satisfies $n\leq N-1$ due to $d(y,z) \leq \rho < \rho  C_2^{-2} \leq C_2^{-2} \Delta_{\lambda \iota, \, \theta} (\O)$ (see (\ref{e_Pf_t_Density_Lattes_rho})) and the definition of $N$ above. Moreover, we have
\begin{equation}  \label{e_Pf_t_Density_Lattes_n}
d \bigl( f^{n-1}(y), f^{n-1}(z) \bigr) \leq  \rho C_2^{-2} .
\end{equation}

We will estimate two parts in the sum 
\begin{equation}  \label{e_Pf_t_Density_Lattes_sum12}
S_N ( \gamma - \psi ) (y) = S_n ( \gamma - \psi ) (y) + S_{N-n} ( \gamma - \psi ) ( f^n(y) )  \eqqcolon \operatorname{I}  +  \operatorname{II}
\end{equation}
separately.

For each $j\in \N$ satisfying $j \in [n , N-1]$, by (\ref{e_Pf_t_Density_Lattes_y_orbit_distance_from_O}), (\ref{e_Pf_t_Density_Lattes_n_property}), and the fact that $\Lambda > 1$, we have
\begin{equation*}
d \bigl( f^j(y), \O \bigr) 
= d \bigl( f^j(y) , f^j(z) \bigr) 
\geq C_2^2 \Lambda^{j-n} d(f^n(y), f^n(z)) 
>\rho.
\end{equation*}
Thus, $f^j (y) \notin U$, and by (\ref{e_Pf_t_Density_Lattes_outside_U}), $\gamma - \psi \bigl( f^j(y) \bigr) >0$ for each $j\in \{n, \, n+1, \, \dots, \, N-1 \}$. By (\ref{e_Pf_t_Density_Lattes_away_from_crit}) and (\ref{e_Pf_t_Density_Lattes_O_crit_distance}), we have $f^{N-1}(y), \, f^{N-1}(z) \notin N^{\iota/2}_d (\post f)$. Hence, by (\ref{e_Pf_t_Density_Lattes_psi}), (\ref{e_Pf_t_Density_Lattes_wtvarphi}), (\ref{e_Pf_t_Density_Lattes_metrics_compare_away_crit}), (\ref{e_Pf_t_Density_Lattes_y_orbit_distance_from_O}), (\ref{e_Pf_t_Density_Lattes_N}), and $C_2 \in (0,1)$, we have
\begin{align*} \label{e_Pf_t_Density_Lattes_bound_II}
\operatorname{II}
& \geq  \gamma - \psi \bigl( f^{N-1} (y) \bigr)\\
&  =      \gamma - \wt\varphi \bigl( f^{N-1} (y) \bigr)  + \epsilon  \sigma \bigl( f^{N-1}(y), \O \bigr)^\alpha \notag \\
& \geq  \gamma + \epsilon C_4^{-\alpha}  d \bigl( f^{N-1}(y), f^{N-1} (z) \bigr)^\alpha \notag\\
& \geq  \gamma + \epsilon C_4^{-\alpha}  ( \Delta_{\lambda \iota, \, \theta} (\O) )^\alpha. \notag
\end{align*}

To estimate $\operatorname{I}$, we write
\begin{equation}   \label{e_Pf_t_Density_Lattes_sum34}
\operatorname{I} = ( n \gamma - S_n \psi(z) ) + ( S_n \psi(z) - S_n \psi(y) ) \eqqcolon \operatorname{III} + \operatorname{IV}
\end{equation}
and bound each part below.

We write $n = pq + r$ for $q,\,r\in\N_0$ with $0 \leq r \leq p - 1$. Then by (\ref{e_Pf_t_Density_Lattes_psi}), (\ref{e_Pf_t_Density_Lattes_wtvarphi}), and (\ref{e_Pf_t_Density_Lattes_gamma}), we have $S_n \psi (z) \leq S_n \wt\varphi (z) = S_{pq} \wt\varphi (z) + S_r \wt\varphi (z) \leq pq\gamma$. Thus, considering $\gamma<0$ (see (\ref{e_Pf_t_Density_Lattes_gamma})), we get
\begin{equation*}    \label{e_Pf_t_Density_Lattes_bound_III}
\operatorname{III} \geq r \gamma \geq (p-1) \gamma.
\end{equation*}
Next, by (\ref{e_Pf_t_Density_Lattes_psi}), (\ref{e_Pf_t_Density_Lattes_metrics_compare_global}), (\ref{e_Pf_t_Density_Lattes_y_orbit_distance_from_O}), (\ref{e_Pf_t_Density_Lattes_epsilon}), (\ref{e_Pf_t_Density_Lattes_n}), (\ref{e_Pf_t_Density_Lattes_rho}), and (\ref{e_C12}), we have
\begin{align*}
\abs{\operatorname{IV}}
&  \leq \sum_{j=0}^{n-1}  \Absbig{  \psi \bigl( f^j (z) \bigr) - \psi \bigl( f^j (y) \bigr) }  \\
&  \leq \sum_{j=0}^{n-1}  \bigl( \Absbig{  \wt\varphi \bigl( f^j (z) \bigr) - \wt\varphi \bigl( f^j (y) \bigr) } + \epsilon \sigma \bigl( f^j(y), \O \bigr)^\alpha \bigr)  \\
&  \leq \sum_{j=0}^{n-1}  \bigl( \Absbig{  \wt\varphi \bigl( f^j (z) \bigr) - \wt\varphi \bigl( f^j (y) \bigr) } + \epsilon C_3 d \bigl( f^j(y), f^j(z) \bigr)^\alpha \bigr)  \\ 
&  \leq \sum_{j=0}^{n-1}  ( \Hseminorm{d^\alpha }{\wt\varphi} + \epsilon C_3 ) d \bigl( f^j(y), f^j(z) \bigr)^\alpha  \\
&  \leq \sum_{j=0}^{n-1}  ( \Hseminorm{d^\alpha }{\wt\varphi} +  C_3 ) C_2^{- 2 \alpha} \Lambda^{-(n-1-j) \alpha} d \bigl( f^{n-1}(y), f^{n-1}(z) \bigr)^\alpha  \\
&  \leq \sum_{j=0}^{n-1}  ( \Hseminorm{d^\alpha }{\wt\varphi} +  C_3 ) C_2^{- 2 \alpha} \Lambda^{-(n-1-j) \alpha} \rho^\alpha C_2^{-2\alpha}  \\
&  \leq \frac{ \Hseminorm{d^\alpha }{\wt\varphi} +  C_3 }{ C_2^4 (1-\Lambda^{-\alpha} )} \cdot \rho^\alpha   \\
&  \leq \frac{ \Hseminorm{d^\alpha }{\wt\varphi} +  C_3}{ C_2^4 (1-\Lambda^{-\alpha} )} \cdot  C_4^\alpha \epsilon^{-1} \abs{\gamma} \\
&  \leq C_5   \epsilon^{-1} \abs{\gamma} .
\end{align*}

Combining the above estimates for $\operatorname{II}$, $\operatorname{III}$, and $\operatorname{IV}$, we get from (\ref{e_Pf_t_Density_Lattes_sum12}), (\ref{e_Pf_t_Density_Lattes_sum34}), (\ref{e_Pf_t_Density_Lattes_p_gamma}), and (\ref{e_Pf_t_Density_Lattes_tau}) the final estimate
\begin{align*}
N \gamma -  S_N\psi (y) 
&       = \operatorname{II} + \operatorname{III} + \operatorname{IV}  \\
& \geq   \gamma + \epsilon C_4^{-\alpha}  ( \Delta_{\lambda \iota, \, \theta} (\O) )^\alpha - (p-1) \abs{\gamma} - C_5 \epsilon^{-1} \abs{\gamma} \\
& \geq    \epsilon C_4^{-\alpha}  ( \Delta_{\lambda \iota, \, \theta} (\O) )^\alpha -  \bigl( 1 + C_5 \epsilon^{-1} \bigr) p \abs{\gamma}\\
& \geq   \bigl(  \epsilon C_4^{-\alpha}    -  \bigl( 1 + C_5 \epsilon^{-1} \bigr) \Hseminorm{d^\alpha }{\wt\varphi} \tau  \bigr)      ( \Delta_{\lambda \iota, \, \theta} (\O) )^\alpha\\
& \geq 0.
\end{align*}

\smallskip

Claim~2 is now established.
\end{proof}

The proof of Theorem~\ref{t_Density_Thurston}~(i) is verbatim the same as that of Theorem~\ref{t_Density_Lattes}~(i) once we replace the chordal metric $\sigma$ by the visual metric $d$ (while the counterparts in the proof of Theorem~\ref{t_Density_Thurston}~(i) to some inequalities and inclusions in the proof of Theorem~\ref{t_Density_Lattes}~(i) become vacuously true as a result).

Combining Theorem~\ref{t_Density_Thurston}~(i) with the fact that for a rational expanding Thurston map $f \: \wh\C \rightarrow \wh\C$ on the Riemann sphere $\wh\C$, the identity map between the chordal sphere $(\wh\C, \sigma)$ and a visual sphere $(\wh\C,d)$, where $d$ is any visual metric, is a quasisymmetry (see \cite[Lemma~18.10]{BM17}), and is therefore bi-H\"older (see Definition~\ref{d_Quasi_Symmetry} and Remark~\ref{r_ChordalVisualQSEquiv}), we get the following result, which serves as a supplement to Theorem~\ref{t_Density_Rational_little}. We recall that a subspace $Y$ of a topological space $X$ is dense in another subspace $Z$ of $X$ if $Z$ is a subset of the closure of $Y$ in $X$ \cite[Definition~2.14]{Gr67}.

\begin{theorem}[Density of periodic maximization for Misiurewicz--Thurston rational maps]  \label{t_Density_Rational}
Let $f\: \wh\C \rightarrow \wh\C$ be a Misiurewicz--Thurston rational map (i.e., a postcritically-finite rational map without periodic critical points). Let $\sigma$ and $d$ be the chordal metric and a visual metric, respectively, on the Riemann sphere $\wh\C$. Fix a number $\alpha \in (0,1]$. Then the set $\sP(\wh\C) \cap \Lip ( \wh\C, \sigma^{\alpha} )$ is dense in $\Lip(\wh\C, d^{\alpha})$ with respect to the $\alpha$-H\"older norm. Moreover, there exists a number $\beta \in (0, \alpha)$ such that the set $\sP(\wh\C) \cap \Lip ( \wh\C, \sigma^{\beta} )$ is dense in $\Lip(\wh\C, \sigma^{\alpha})$ with respect to the $\beta$-H\"older norm.
\end{theorem}

\section{Proofs of genericity of the locking property in little Lipschitz spaces}   \label{sct_lip}
In this section, we prove the remaining results from Section~\ref{sct_Introduction}.  We denote by $\cl_{d}(\sF)$ the closure of a subset $\sF \subseteq \Lip(X, d)$ in $\Lip(X,d)$. Recall that for a metric space $(X,d)$ and a number $\alpha \in (0,1]$, the snowflake $d^\alpha (x,y) \= (d(x,y))^\alpha$ of $d$ is also a metric with the same topology.

\begin{theorem}  \label{t_P_dense_in_lip}
Let $T\: X \rightarrow X$ be a continuous map on a compact metric space $(X,d)$. Assume that $\sP(X) \cap \Lip(X,d)$ is dense in $\Lip(X,d)$. Then $\sP(X) \cap \lip(X,d^\beta)$ is dense in $\lip(X,d^\beta)$ for each $\beta\in (0,1)$.
\end{theorem}

\begin{proof}
Fix an arbitrary $\beta \in (0,1)$. By Theorem~\ref{t_lock} it suffices to show that $\lock(X, d^\beta  )$ is dense in $\sP(X  ) \cap \lip(X, d^\beta  )$.

Note that $\Lip(X, d) \subseteq \Lip  (X, d^\beta  )$ and $d(x,y) \leq d^\beta (x,y) \diam_{d}(X)^{1-\beta}$ for all $x,\,y\in X$. Then it is straightforward to check that
\begin{equation} \label{e_Pf_t_Density_Lattes_lip_closures_sigma}
\cl_{d^\beta} ( \sF )  = \cl_{d^\beta} ( \cl_{d} ( \sF ) )
\end{equation}
for each subset $\sF$ of $\Lip(X, d)$. Thus, by (\ref{e_Pf_t_Density_Lattes_lip_closures_sigma}), the density of $\sP (X) \cap \Lip( X, d )$ in $\Lip( X, d)$, Proposition~\ref{p_lip_closed}, and Proposition~\ref{p_Lip_dense_lip_alpha},
\begin{equation*}
\cl_{d^\beta} ( \sP( X ) \cap \Lip( X, d ) ) 
= \cl_{d^\beta} ( \cl_{d} ( \sP( X )  \cap \Lip( X, d ) )  ) 
= \cl_{d^\beta} ( \Lip( X, d ) ) 
= \lip ( X, d^\beta ).
\end{equation*}
Since $\Lip(X, d) \subseteq \lip (X, d^\beta )$ and $\lip (X, d^\beta )$ is closed in $\Lip (X, d^\beta )$ (by Proposition~\ref{p_Lip_dense_lip_alpha} and Proposition~\ref{p_lip_closed}), we get that $\cl_{d^\beta} ( \sP( X )  \cap \lip ( X, d^\beta ) ) = \lip ( X, d^\beta )$.
\end{proof}

We are now ready to establish the remaining results from Section~\ref{sct_Introduction}.

Theorem~\ref{t_Density_Thurston}~(ii) and Theorem~\ref{t_Density_Lattes}~(ii) now follows from Theorem~\ref{t_P_dense_in_lip} together with Theorem~\ref{t_Density_Thurston}~(i) and Theorem~\ref{t_Density_Lattes}~(i), respectively.

\begin{proof}[Proof of Theorem~\ref{t_Density_Rational_little}]
Let $d$ be a visual metric on $\wh\C$ for $f$.

By \cite[Lemma~18.10]{BM17}, the identity map between $(\wh\C, \sigma)$ and $(\wh\C, d)$ is a quasisymmetry, and is thus bi-H\"older (see Remark~\ref{r_ChordalVisualQSEquiv}). Hence, there exist numbers $C_6> 0$, $C_7>0$, and $0 < \gamma < \eta < 1$ such that for all $x,\,y\in \wh\C$,
\begin{equation}     \label{e_Pf_t_density_Rational_lip_visual_vs_chordal}
\sigma (x, y) \leq C_{6} d(x,y)^\eta \leq C_7 \sigma(x,y)^\gamma
\end{equation}
and
\begin{equation}     \label{e_Pf_t_density_Rational_lip_Lip_inclusions}
\Lip( \wh\C, \sigma ) \subseteq \Lip ( \wh\C, d^\eta ) \subseteq \Lip ( \wh\C, \sigma^\gamma ).
\end{equation}

Fix an arbitrary $\beta \in (0,\gamma)$. 

By Theorem~\ref{t_lock} it suffices to show that $\lock(\wh\C, \sigma^\beta  )$ is dense in $\sP(\wh\C  ) \cap \lip(\wh\C, \sigma^\beta  )$.

By (\ref{e_Pf_t_density_Rational_lip_visual_vs_chordal}) and (\ref{e_Pf_t_density_Rational_lip_Lip_inclusions}) it is straightforward to check that
\begin{equation}  \label{e_Pf_t_density_Rational_lip_closures}
\cl_{d^\eta} ( \sF )  \subseteq \cl_{\sigma^\gamma} (  \sF ) 
\end{equation}
for each subset $\sF$ of $\Lip(\wh\C, d^\eta)$. Similarly, it follows from $\Lip(\wh\C, \sigma^\gamma) \subseteq \Lip (\wh\C, \sigma^\beta )$ and $\sigma^\gamma(x,y) \leq \sigma^\beta (x,y) \diam_{\sigma}(\wh\C)^{\gamma-\beta}$ for all $x,\,y\in \wh\C$ that 
\begin{equation} \label{e_Pf_t_density_Rational_lip_closures_sigma}
\cl_{\sigma^\beta} ( \sF )  = \cl_{\sigma^\beta} ( \cl_{\sigma^\gamma} ( \sF ) )
\end{equation}
for each subset $\sF$ of $\Lip(\wh\C, \sigma^\gamma)$.

It follows from Proposition~\ref{p_lip_closed}, Proposition~\ref{p_Lip_dense_lip_alpha}, and (\ref{e_Pf_t_density_Rational_lip_Lip_inclusions}) that
\begin{equation}  \label{e_Pf_t_density_Rational_lip_closures_d}
\cl_{\sigma^\beta} ( \Lip ( \wh\C, d^\eta ) )  = \lip ( \wh\C, \sigma^\beta ). 
\end{equation}

Then by (\ref{e_Pf_t_density_Rational_lip_closures_d}), the density of $\sP( \wh\C )  \cap \Lip( \wh\C, d^\eta )$ in $\Lip( \wh\C, d^\eta)$, (\ref{e_Pf_t_density_Rational_lip_closures}), (\ref{e_Pf_t_density_Rational_lip_closures_sigma}), (\ref{e_Pf_t_density_Rational_lip_Lip_inclusions}), Proposition~\ref{p_Lip_dense_lip_alpha}, and Proposition~\ref{p_lip_closed}, we have
\begin{align*}
\lip ( \wh\C, \sigma^\beta )
& = \cl_{\sigma^\beta} ( \Lip ( \wh\C, d^\eta ) ) \\
&= \cl_{\sigma^\beta} ( \cl_{d^\eta} ( \sP( \wh\C )  \cap \Lip ( \wh\C, d^\eta ) ) ) \\
&\subseteq  \cl_{\sigma^\beta} ( \cl_{\sigma^\gamma } ( \sP( \wh\C )  \cap \Lip ( \wh\C, d^\eta ) ) )  \\
& = \cl_{\sigma^\beta} (  \sP( \wh\C )  \cap \Lip ( \wh\C, d^\eta ) )  \\
&\subseteq \cl_{\sigma^\beta} (  \sP( \wh\C )  \cap \Lip ( \wh\C, \sigma^\gamma ) )   \\
&\subseteq \cl_{\sigma^\beta} (  \sP( \wh\C )  \cap \lip ( \wh\C, \sigma^\beta ) ) \\
&\subseteq \lip ( \wh\C, \sigma^\beta ).  
\end{align*}
Therefore, $\cl_{\sigma^\beta} (  \sP( \wh\C )  \cap \lip ( \wh\C, \sigma^\beta ) )  = \lip ( \wh\C, \sigma^\beta )$.
\end{proof}

\begin{proof}[Proof of Theorem~\ref{t_Distance_expanding}]
It follows from Theorems~\ref{t_lock} and~\ref{t_P_dense_in_lip} that the conclusion of this theorem follows from the property that $\sP(X) \cap \Lip(X,d)$ is dense in $\Lip(X,d)$ for each specific system. Indeed, the last property for distance expanding maps follows from \cite[Theorem~A]{Co16} and \cite{BZ15}. Moreover, for Axiom A attractors and Anosov diffeomorphisms, this property is established in \cite[Theorem~2.1]{HLMXZ19}.
\end{proof}

\begin{proof}[Proof of Theorem~\ref{t_gound_states}]
Let $f$ be a postcritically-finite rational map without periodic critical points or an expanding Thurston map on $X$, where $(X,\rho)$ is the Riemann sphere equipped with the chordal metric $(\wh\C, \sigma)$ in the former case or the topological $2$-sphere equipped with a visual metric $(S^2,d)$ in the latter case. Consider $\alpha \in (0,1]$ and $\phi\in \Lip(X,\rho^{\alpha})$. Since there exists a unique equilibrium state $\mu_{t\phi}$ for each $t\in\R$ (\cite[Theorem~1.1 and Corollary~1.2]{Li18}), we get from the definition of equilibrium states and (\ref{e_pressure}) that
\begin{equation} \label{e_equilibrium_state}
h_{\mu_{t\phi}}(f) + t \int\! \phi \,\mathrm{d}\mu_{t\phi} 
\geq h_{\nu}(f) + t \int\! \phi \,\mathrm{d}\nu
\end{equation}
for all $\nu \in \MMM(X,f)$. Note the topological entropy $h_{\operatorname{top}}(f) = \log(\deg f) < +\infty$, where $\deg f$ is the topological degree of $f$ (\cite{HP09,BM17}). If $\mu = \lim\limits_{i\to+\infty} \mu_{t_i\phi}$ in the weak$^*$ topology for some sequence $t_i$, $i\in\N$, of real numbers tending to $+\infty$, we divide both sides of (\ref{e_equilibrium_state}) by $t=t_i$ and get that
\begin{equation*}
\int\! \phi \,\mathrm{d}\mu
=\lim_{i\to+\infty}\int\! \phi \,\mathrm{d}\mu_{t_i\phi}
\geq \int\! \phi \,\mathrm{d}\nu.
\end{equation*}
Hence, $\mu\in\Mmax(f, \, \phi)$. The theorem now follows from the above together with Theorems~\ref{t_Density_Rational_little},~\ref{t_Density_Thurston}, and~\ref{t_Density_Lattes}.
\end{proof}

\appendix

\section{Illustrations of combinatorial structures of expanding Thurston maps}   \label{apx_illustrations}

The rest of the paper is completely independent of this Appendix, so readers fluent in expanding Thurston maps can safely skip it. The purpose of this appendix is to provide illustrations of the combinatorial objects we heavily rely on in this paper to readers who are less familiar with these maps so that they can gain some intuition. We give three examples of expanding Thurston maps $f$ and illustrate the associated combinatorial structures from the cell decompositions $\DD^n(f,\CC)$, $n\in\N_0$, of $S^2$ assuming the existence of $f$-invariant Jordan curves $\CC\subseteq S^2$ containing postcritical points $\post f$. We emphasize here that such a Jordan curve may not always exist for certain expanding Thurston maps (see \cite[Example~15.11]{BM17}), but these combinatorial constructions are always available for all Jordan curves containing postcritical points $\post f$. We impose the additional assumption to simplify the illustrations.

\begin{figure}[h]
 \begin{subfigure}{0.3\textwidth}
    \center
    \begin{overpic}
    [width=4.5cm, 
    tics=20]{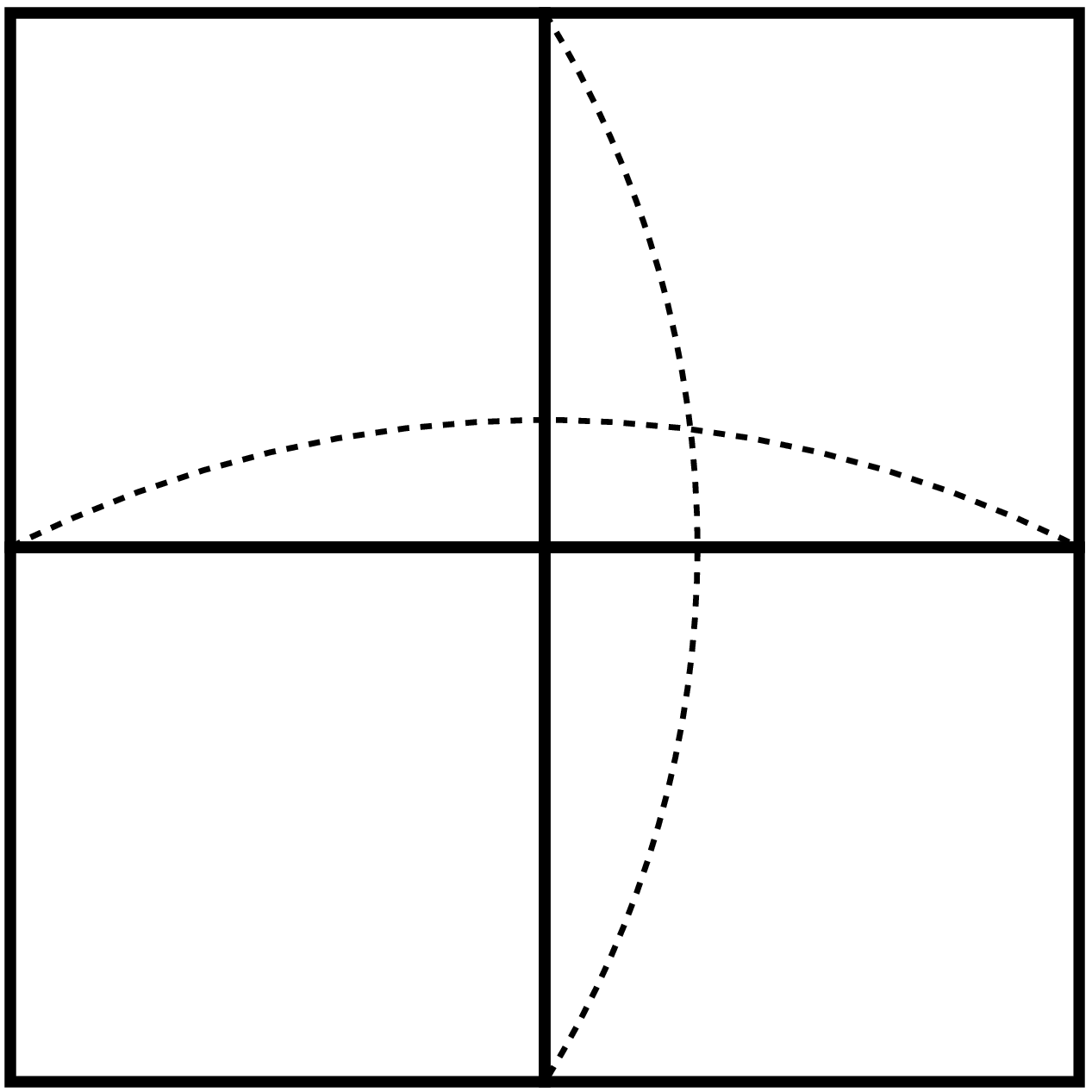}
    \put(4,4){$A$}    
    \put(114,4){$B$}
    \put(114,114){$C$}
    \put(4,114){$D$}  
    \put(52,4){$E$}  
    \put(52,52){$F$}  
    \put(4,52){$G$} 
    \put(114,52){$H$}  
    \put(54,114){$I$}      
    \put(82,80){$J$}  
    \end{overpic}
    \caption{}
    \label{f_a1}
 \end{subfigure}
 \begin{subfigure}{0.3\textwidth}   
    \center 
    \begin{overpic}
    [width=4.5cm, 
    tics=20]{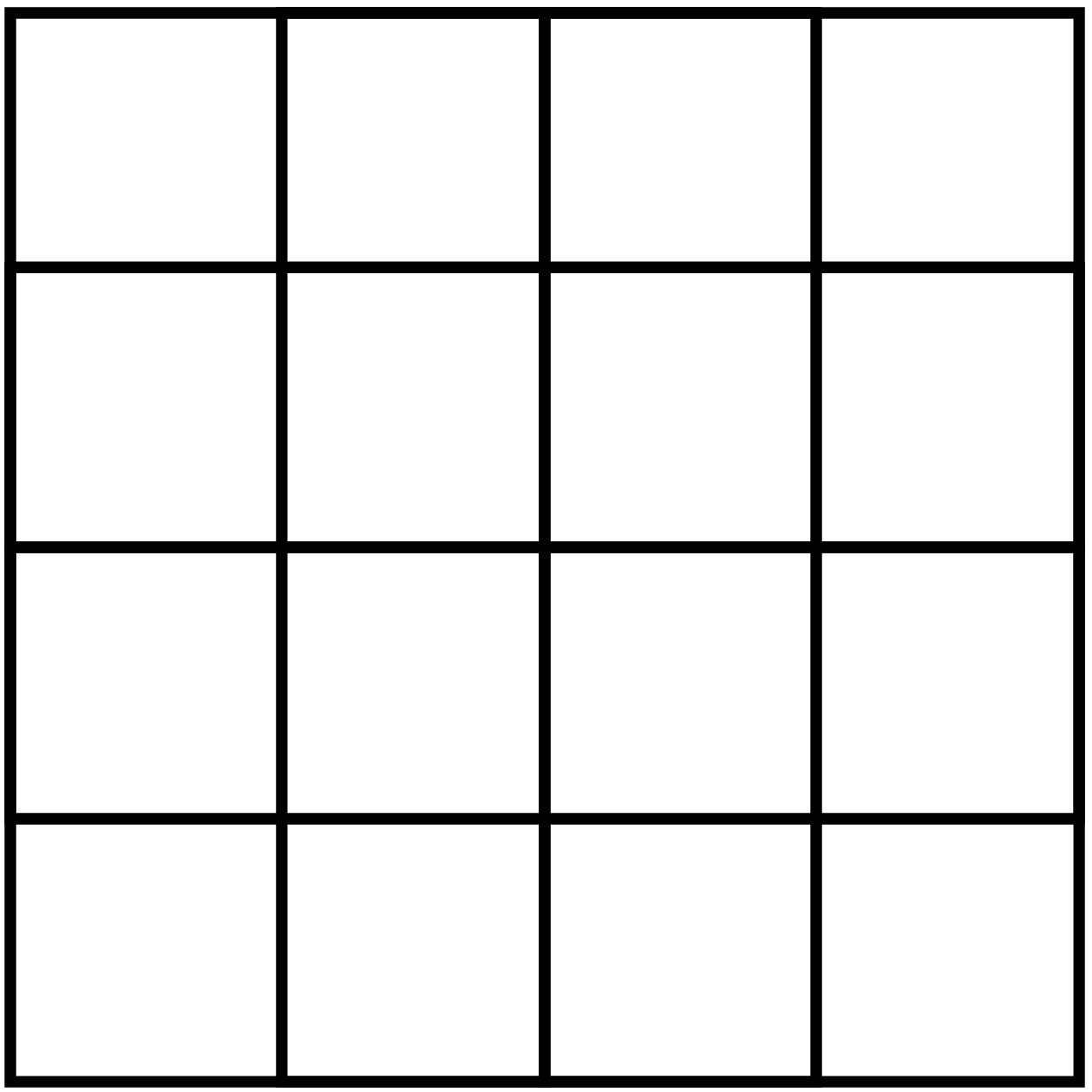}
    \end{overpic}
    \caption{}
    \label{f_a2}  
 \end{subfigure}  
 \begin{subfigure}{0.3\textwidth}   
    \center 
    \begin{overpic}
    [width=4.5cm, 
    tics=20]{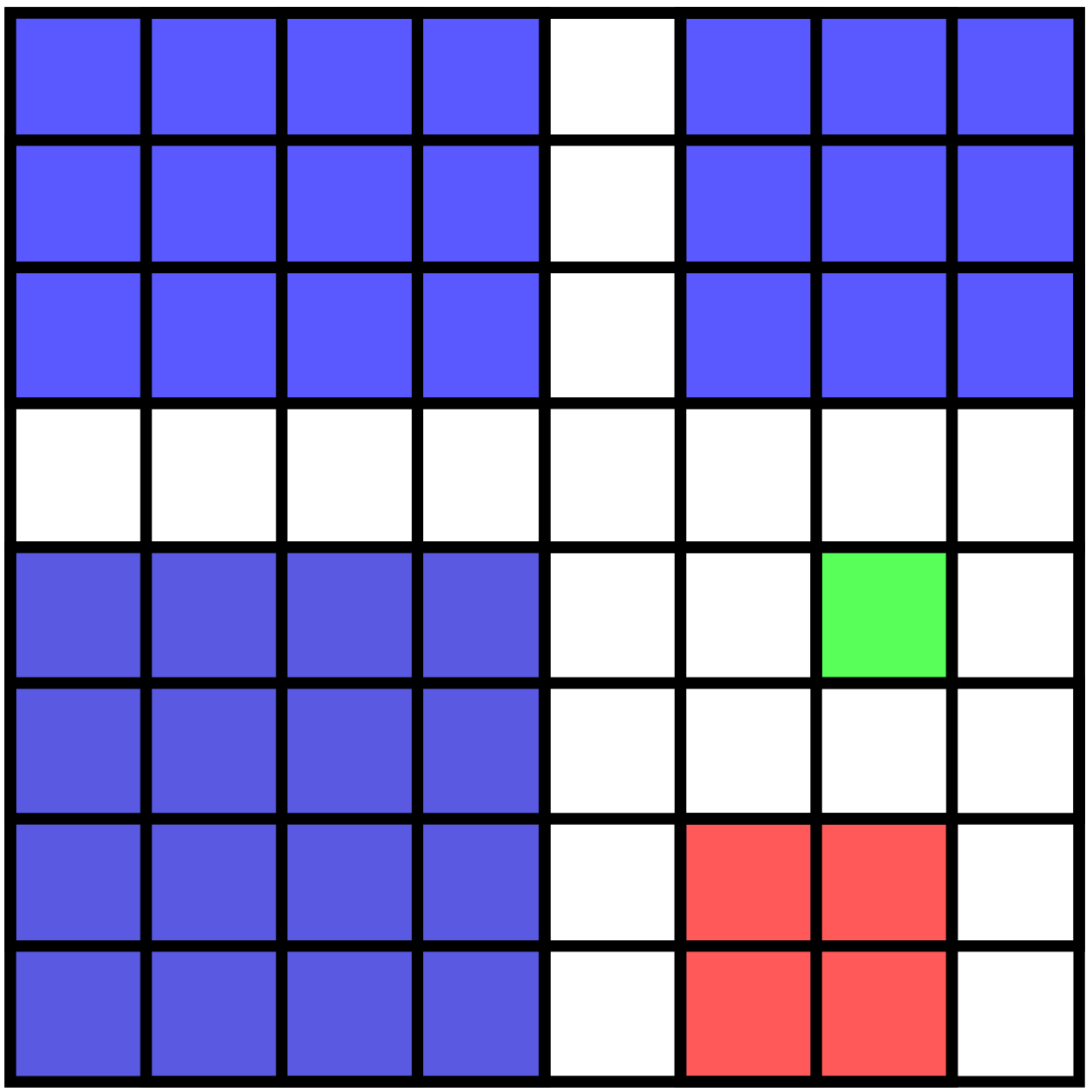}
    \put(34,36){$y$}    
    \put(97,99){$z$}
    \put(36,102){$x$}
    \put(98,20){$v$}
    \put(31,32){\color{yellow}${}_\bullet$}           
    \put(31,105){\color{yellow}${}_\bullet$} 
    \put(94,17){\color{yellow}${}_\bullet$} 
    \put(103,103){\color{yellow}${}_\bullet$}             
    \end{overpic}
    \caption{}
    \label{f_a3}  
 \end{subfigure}  
    \caption{Example 1: $\DD^n(f,\CC)$, $n = 1$ to $3$.}
    \label{f_Lattes}  
\end{figure}

The examples we choose are related to rational maps in different senses, in terms of topological conjugacy and a weaker notion called \emph{Thurston equivalence}. We refer the reader to \cite[Definition~2.4]{BM17} for the precise definition of the latter notion.

With the permission of the authors, Figures~\ref{f_Barycentric_Iterations} and~\ref{f_flap} are adapted from Figures~12.7 and~1.3 of their book \cite{BM17}, respectively.

\smallskip

\emph{Example 1.} The first example is an expanding Thurston map topologically conjugate to a Latt\`es map $g$ given by
\begin{equation*}
g(z) = 4 \frac{z(1-z^2)}{(1+z^2)^2} \qquad \text{for } z \in \widehat{\C}.
\end{equation*}
We consider the topological $2$-sphere $S^2$ as a homeomorphic copy of the \emph{pillow} $\Pillow$ obtained from two copies of the unit square $[0,1]^2 \subseteq \R^2 \cong \C$ glued together along their boundaries (Figure~\ref{f_Lattes}.(A)). Denote $A\= (0,0)$, $B\=(1,0)$, etc.\ as marked on Figure~\ref{f_Lattes}.(A). In particular, $J$ denotes the center of the square on the back side of the pillow. We divide each square into four small squares of equal size. We define a continuous and orientation-preserving map $f\: \Pillow \rightarrow \Pillow$ by requiring that $f$ maps each of the eight small squares by first enlarging it linearly by a factor of $2$ and then mapping the resulting square isometrically (with respect to the Euclidean metric) to the big square $ABCD$ either on the front or the back side of $\Pillow$ (depending on the small square). The map $f$ is uniquely determined if we specify that the small square $AEFG$ is mapped to the front side $ABCD$ of $\Pillow$ with $A\mapsto A$, $E\mapsto B$, $F\mapsto C$, and $G\mapsto D$. In this case, for example, the small square $EBHF$ is mapped to the big square $BADC$ on the back side of $\Pillow$ with $B\mapsto A$ and $H\mapsto D$. Then $f$ is a Thurston map with the set of critical points $\crit f = \{F,\,J,\,E,\,H,\,I,\,G\}$ and the set of postcritical points $\post f = \{A,\,B,\,C,\,D\}$. It is not difficult to check that $f$ is an expanding Thurston map. See \cite[Section~1.1]{BM17} for more details.

We denote the boundary of the big squares by $\CC$. The Jordan curve $\CC$ is $f$-invariant and contains $\post f$. All the line segments (including the dotted ones on the back of $\Pillow$) represent the set $f^{-1}(\CC)$. The map $f$ and Jordan curve $\CC$ induce the cellular Markov partitions $\DD^n(f,\CC)$, $n\in\N_0$, as recalled in Subsection~\ref{subsct_ThurstonMap}. The set $\X^0(f,\CC)$ of $0$-tiles (resp.\ $\X^1(f,\CC)$ of $1$-tiles) consists of the two big (resp.\ eight small) squares on the front side and back side of $\Pillow$. The set $\V^0(f,\CC)$ of $0$-vertices consists of $A$, $B$, $C$, and $D$. Points $A$ through $J$ form the set $\V^1(f,\CC)$ of $1$-vertices.

The line segments in Figure~\ref{f_Lattes}.(B) illustrate the set $f^{-2}(\CC)$ on either side of $\Pillow$. The $16$ small squares are $2$-tiles. There are a total of $32$ $2$-tiles in $\Pillow$.

\begin{figure}[h]
    \center
    \begin{overpic}
    [width=6cm, 
    tics=20]{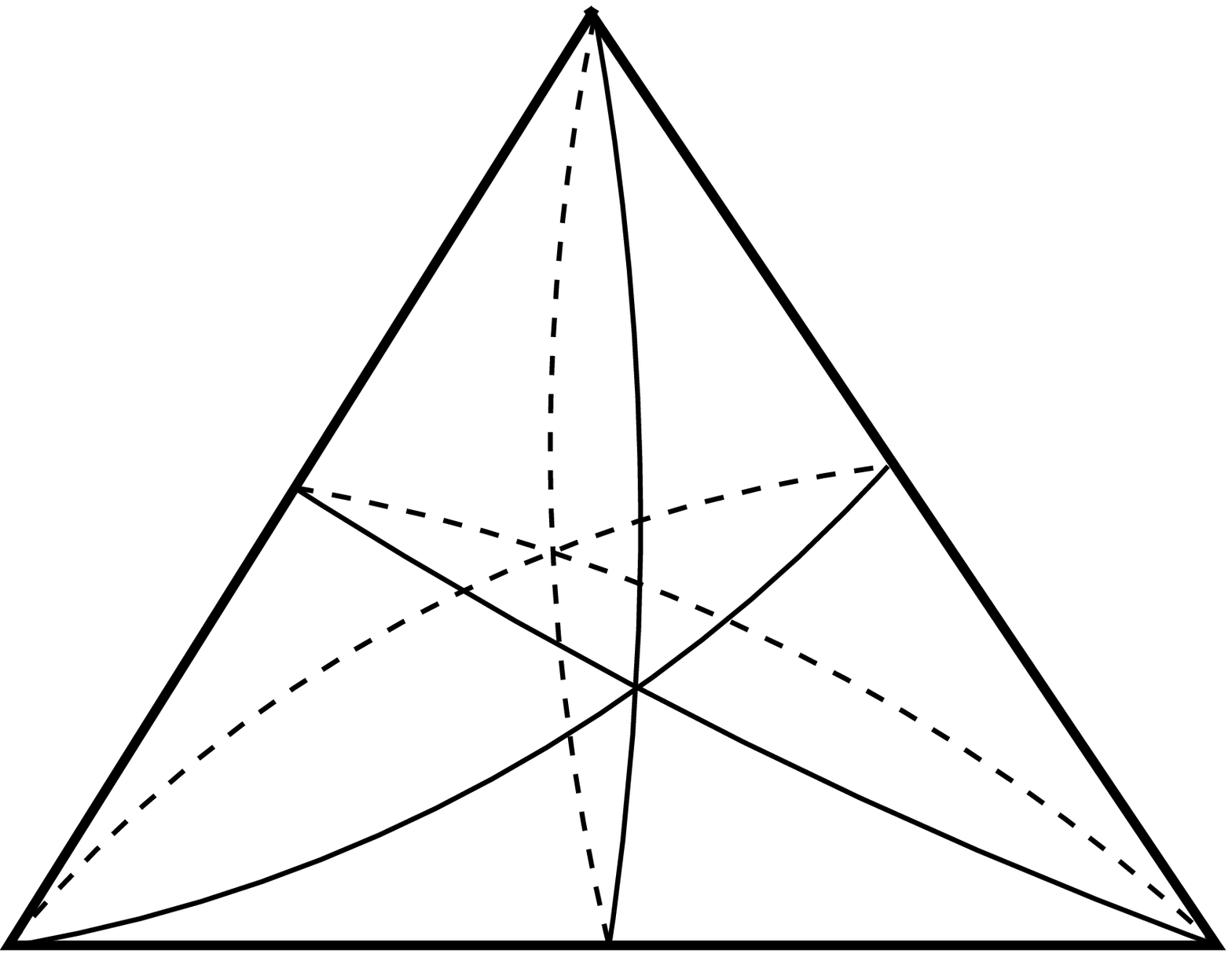}
    \end{overpic}
    \caption{Example 2.}
    \label{f_Barycentric}
\end{figure}

\begin{figure}
 \begin{subfigure}{0.45\textwidth}   
    \center 
    \begin{overpic}
    [width=7cm, 
    tics=20]{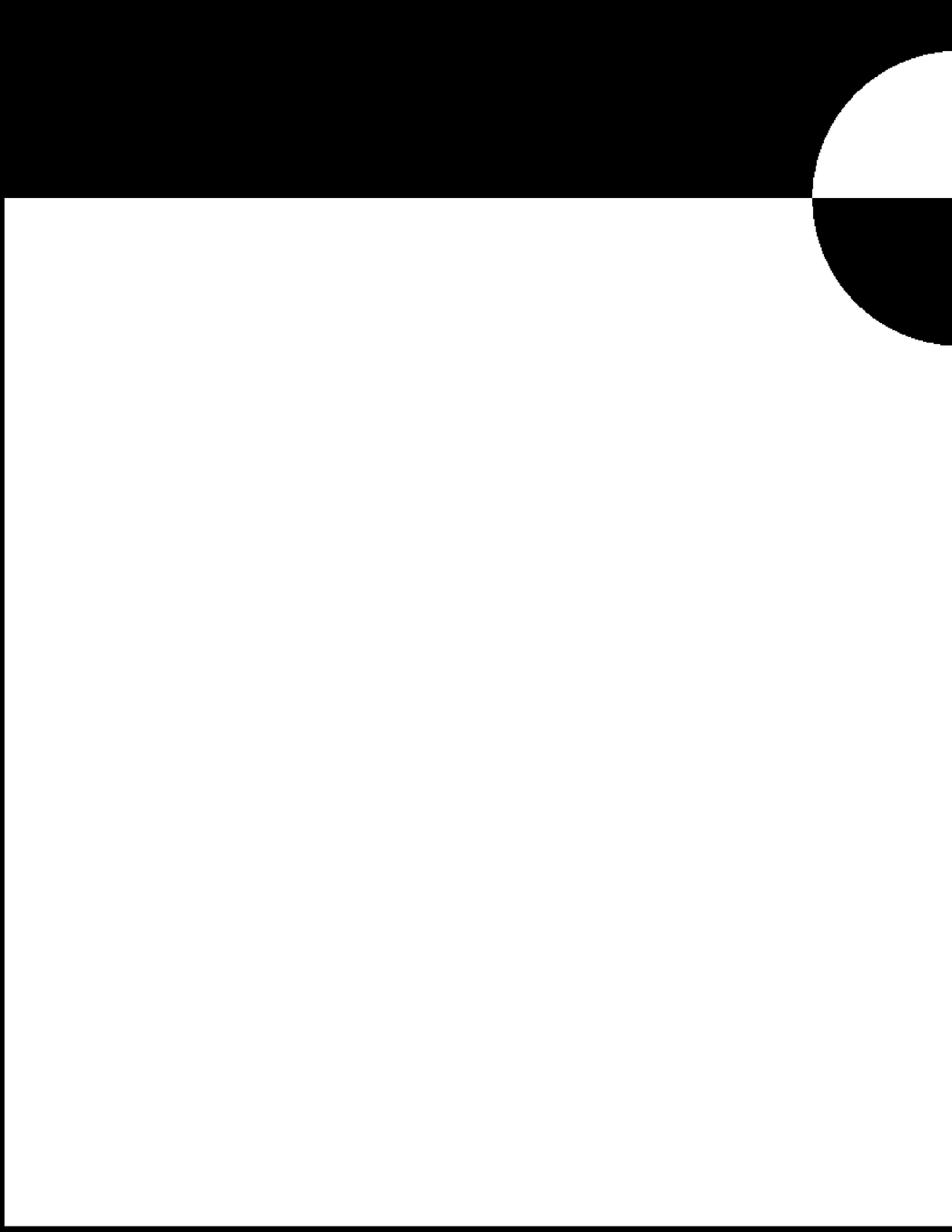}    
    \end{overpic}
 \end{subfigure}
 \begin{subfigure}{0.45\textwidth}   
    \center 
    \begin{overpic}
    [width=7cm, 
    tics=20]{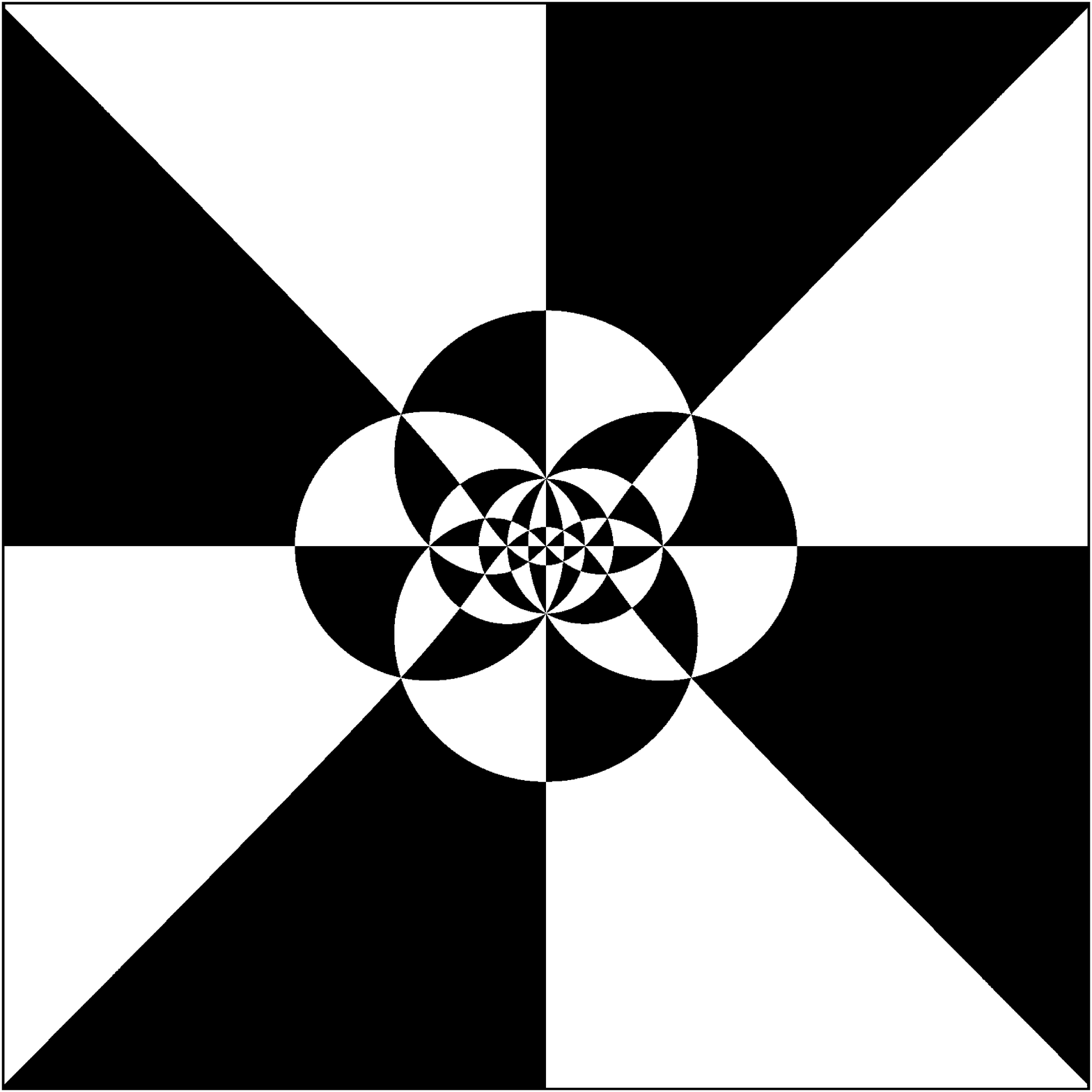}    
    \end{overpic}
 \end{subfigure} 
 \begin{subfigure}{0.45\textwidth}   
    \center 
    \begin{overpic}
    [width=7cm, 
    tics=20]{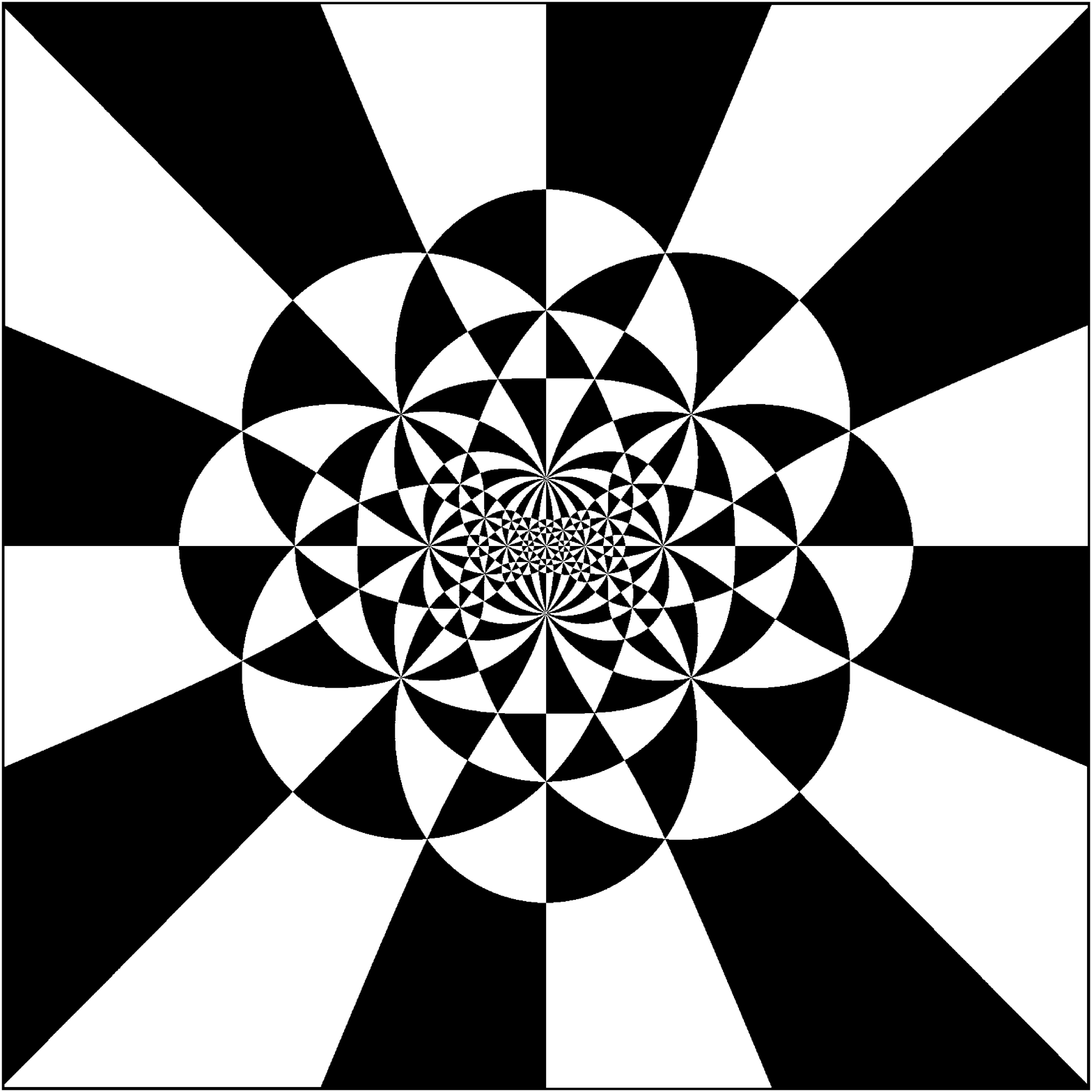}    
    \end{overpic}
 \end{subfigure}  
 \begin{subfigure}{0.45\textwidth}   
    \center 
    \begin{overpic}
    [width=7cm, 
    tics=20]{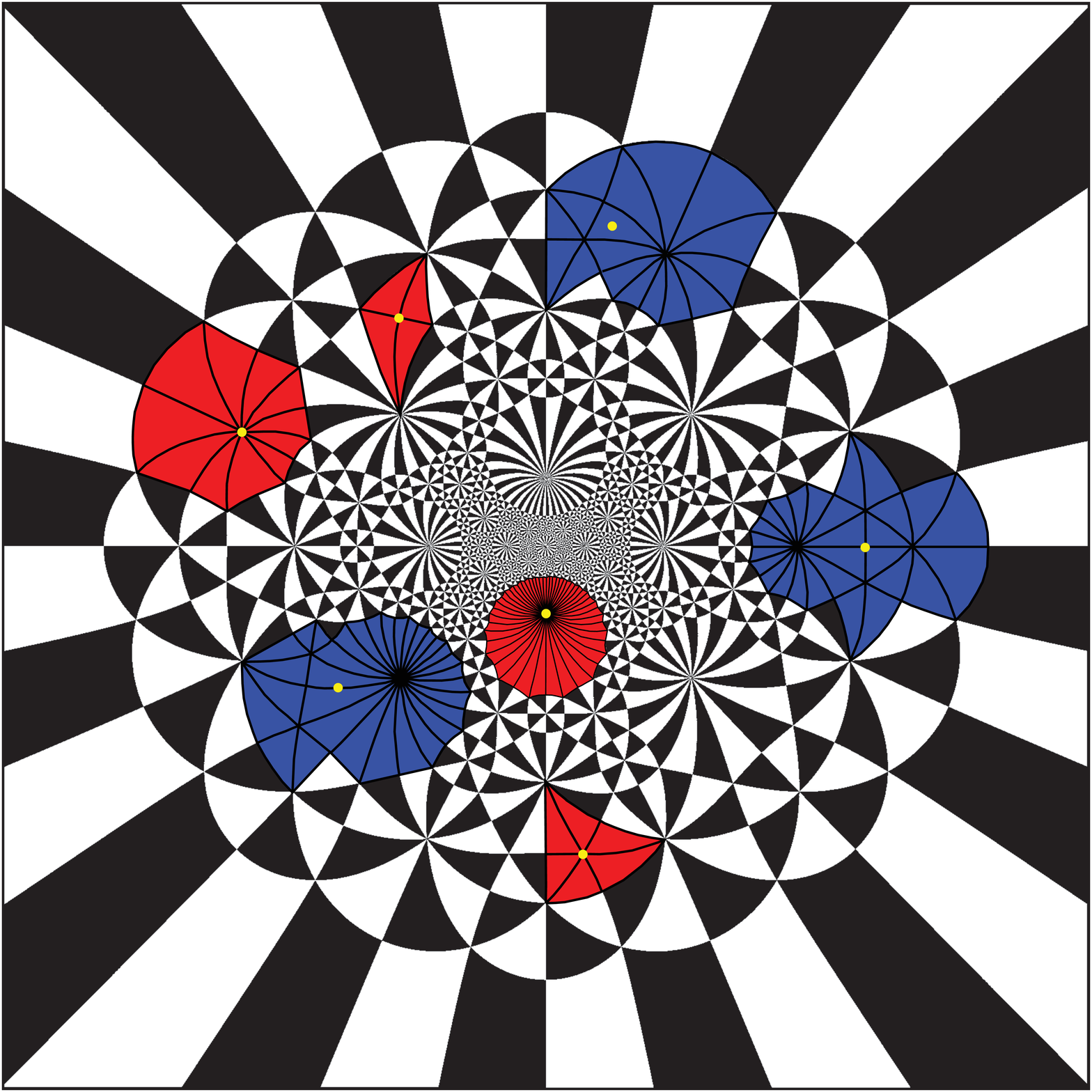}    
    \end{overpic}
 \end{subfigure}  
 \begin{subfigure}{0.45\textwidth}   
    \center 
    \begin{overpic}
    [width=7cm, 
    tics=20]{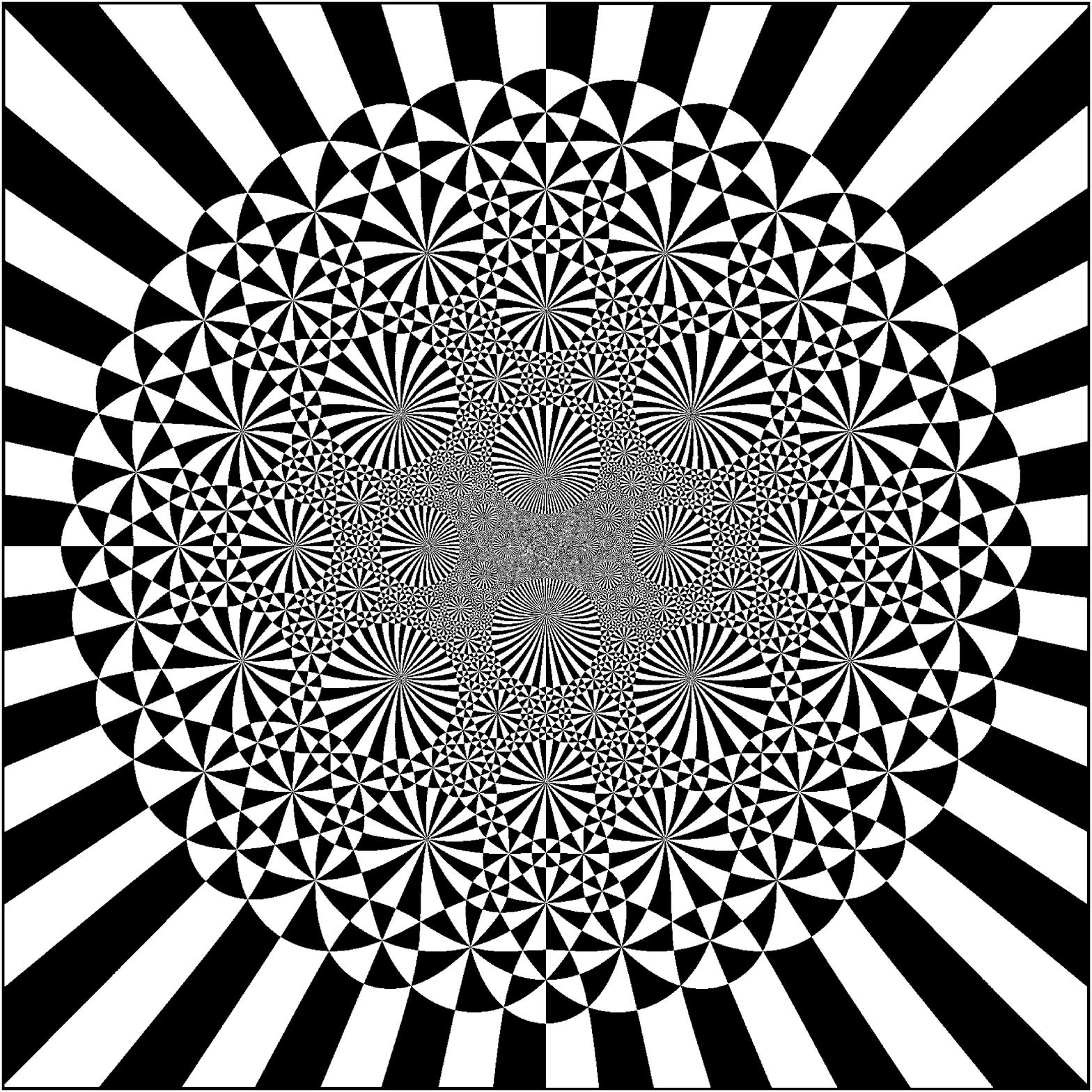}    
    \end{overpic}
 \end{subfigure}  
 \begin{subfigure}{0.45\textwidth}   
    \center 
    \begin{overpic}
    [width=7cm, 
    tics=20]{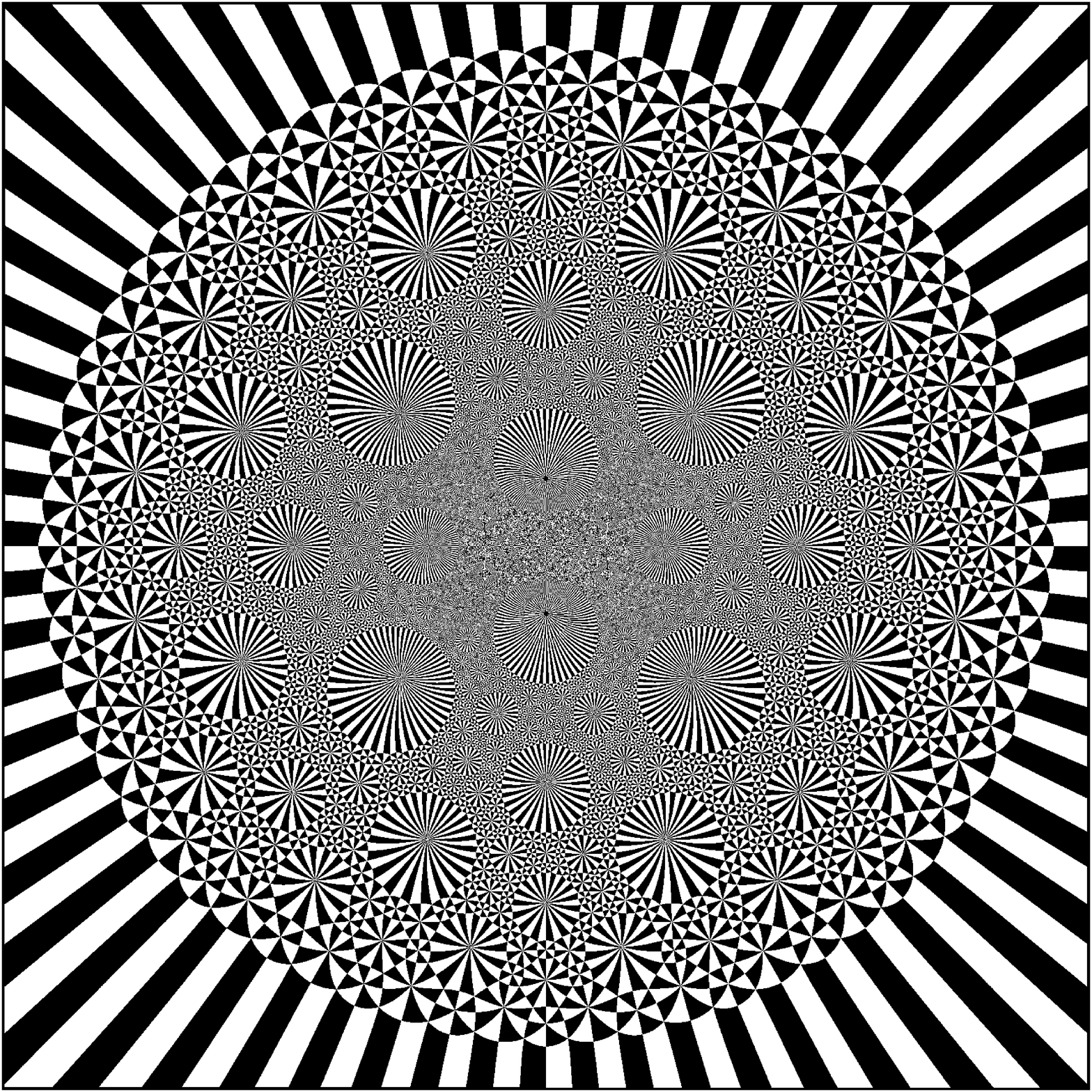}    
    \end{overpic}
 \end{subfigure}    
    \caption{Example 2: $\DD^n(f,\CC)$, $n = 1$ to $6$, adapted from \cite{BM17}.}
    \label{f_Barycentric_Iterations}  
\end{figure}

In Figure~\ref{f_Lattes}.(C), the line segments illustrate the set $f^{-3}(\CC)$ on either side of $\Pillow$. Each small square, such as the green one, is a $3$-tile. Recall the notions of flowers and bouquets from (\ref{e_Def_Flower}) and (\ref{e_Def_U^n}). The red $3$-tiles (minus the boundary of their union) form a $3$-flower $W^n(v)$ of a $3$-vertex $v \in \V^3(f,\CC)$. Three $3$-bouquets $U^3(x)$, $U^3(y)$, and $U^3(z)$ are illustrated in blue. Their shapes depend on the (combinatorial) locations of $x$, $y$, and $z$.

\smallskip

\begin{figure}[h]
    \center
    \begin{overpic}
    [width=14cm, 
    tics=20]{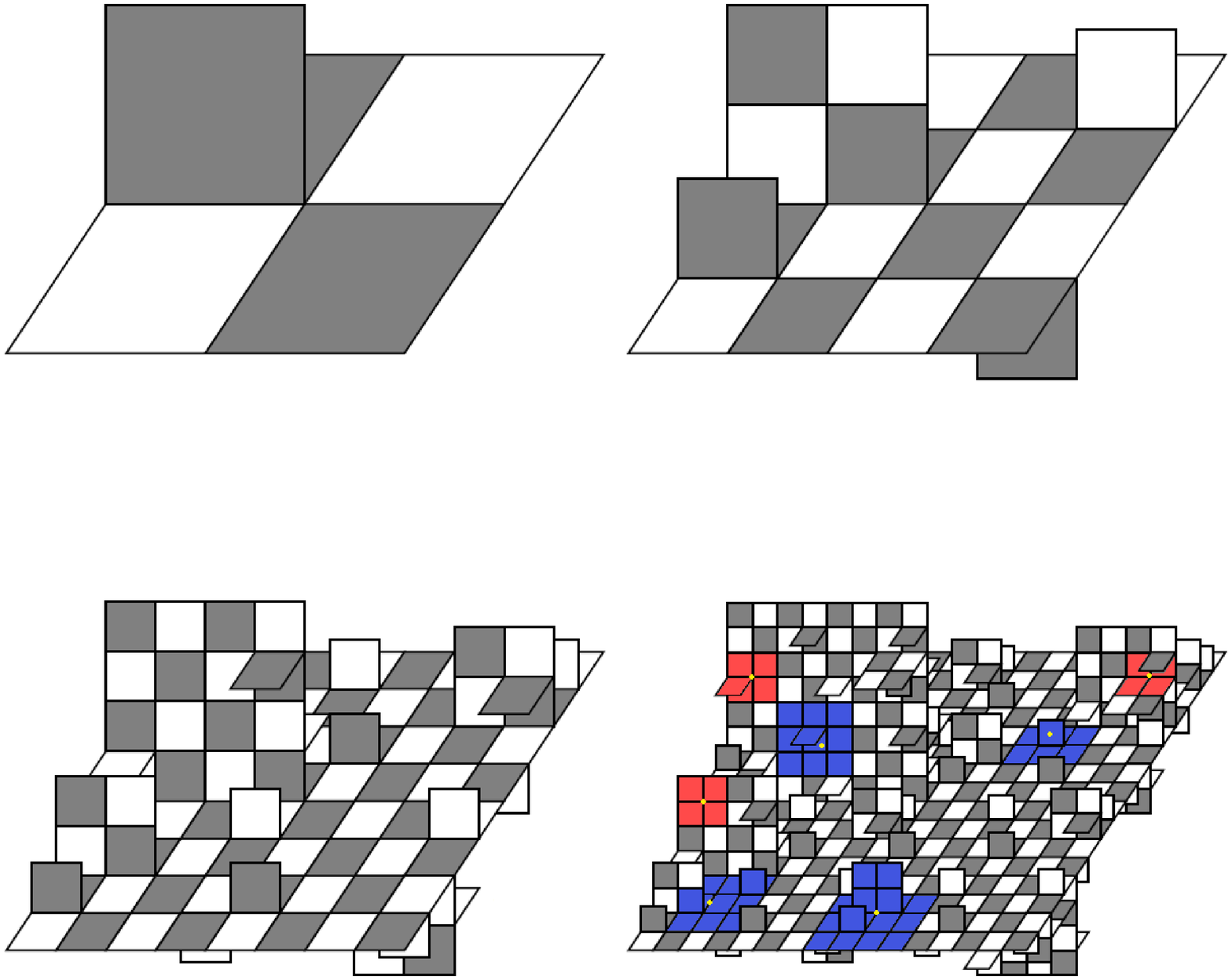}
    \end{overpic}
    \caption{Example 3: $\DD^n(f,\CC)$, $n = 1$ to $3$, adapted from \cite{BM17}.}
    \label{f_flap}
\end{figure}

\emph{Example 2.} The next example is obtained in a similar way. We take two copies of an equilateral triangle of equal size in $\R^2$ and glue them together along their boundaries to form a topological $2$-sphere $S^2$ (Figure~\ref{f_Barycentric}). The bisectors divide each equilateral triangle into six small triangles. We define a continuous and orientation-preserving map $f\: S^2 \rightarrow S^2$ by requiring that $f$ maps each of the twelve triangles by first enlarging it linearly to the shape of the original equilateral triangles and then mapping the resulting triangle isometrically (with respect to the Euclidean metric) to the equilateral triangle either on the front or the back side of $S^2$ (depending on the small triangle). The map $f$ is uniquely determined if we specify how one of the twelve small triangles is mapped. A map $f$ constructed in this way is an expanding Thurston map with three postcritical points. It is not \emph{obstructed}, and in fact one such $f$ is \emph{Thurston equivalent} to a rational map $g$ given by
\begin{equation*}
g(z) = 1 - \frac{ 54 (z^2 - 1)^2 }{ (z^2 + 3)^3} \qquad \text{for } z \in \widehat{\C}.
\end{equation*} 
However, $f$ cannot be conjugate to a rational map since it has a periodic critical point (compare \cite[Proposition~2.3]{BM17}). The boundary of the two original equilateral triangles is an $f$-forward-invariant Jordan curve $\CC$ containing $\post f$. See \cite[Examples~12.21 and~18.11]{BM17} for more details.

In Figure~\ref{f_Barycentric_Iterations}, the cell decompostions $\DD^n(f,\CC)$, $n= 1$ to $6$, are illustrated. The tiles are colored black or white, some $4$-flowers red, and some $4$-bouquets blue.

\smallskip

\emph{Example 3.} The last example is obtained from cutting, gluing, and pasting a big pillow $\Pillow$ and a small pillow $\Pillow_0$. Both pillows are marked the same way as the pillow in Figure~\ref{f_Lattes}.(A). We cut the edge $GF$ of $\Pillow$ open to create an opening of $\Pillow$. The boundary of the opening is homeomorphic to $S^1$ consisting of two arcs $e_1$ and $e_2$, both of which start at $G$ and end at $F$. The small pillow $\Pillow_0$ is $1/2$ the scale of $\Pillow$. We cut the edge $AB$ of $\Pillow_0$ open in a similar way to form two arcs $e'_1$ and $e'_2$, both of which start at $A$ and end at $B$. Note that $e_1$, $e_2$, $e'_1$, $e'_2$ are of the same length. We glue $e'_1$ to $e_1$ and $e'_2$ to $e_2$. The resulting surface is a topological $2$-sphere $S^2$. It contains $10$ small squares. We can define an expanding Thurston map $f\: S^2 \rightarrow S^2$ in a similar way as the two examples above, which maps each of the ten small squares to either the front or the back side of $S^2$. The Jordan curve $\CC$ on the original big pillow $\Pillow$ is the boundary of the front side and the back side of the resulting $S^2$. It remains invariant under $f$ and contains $\post f$. The cell decompostions $\DD^n(f,\CC)$, $n= 1,\,2,\,3,\,4$, of such a map $f$ are illustrated in Figure~\ref{f_flap}. The tiles are colored white or grey, some $4$-flowers red, and some $4$-bouquets blue. This map is obstructed, i.e., not Thurston equivalent to a rational map (\cite[Theorem~1.2]{BHI21}). In particular, it is not topologically conjugate to a rational map.

\subsection*{Acknowledgments} 
We want to extend our gratitude to Mario~Bonk and Daniel~Meyer for the permission to include in Appendix~\ref{apx_illustrations} two figures adapted from the original ones in their book \cite{BM17}. We also want to thank Yinying~Huang for pointing out several typos, and Tim Mesikepp as well as the anonymous referee(s) for valuable comments. Y.~Zhang is grateful to the Beijing International Center for Mathematical Research (BICMR), Peking University, for the hospitality during his visits when part of this work was done. Z.~Li was partially supported by NSFC Nos.~12101017, 12090010, 12090015, 12471083, and BJNSF No.~1214021. Y.~Zhang was partially supported by NSFC Nos.~12161141002, 11871262, 12271432, and Guangdong Basic and Applied Basic Research Foundation No.~2024A1515010974.

\subsection*{Data availability} 
Data sharing is not applicable to this article as no new data was created or analyzed in this study.

\section*{Declarations}
{\bf Competing Interests:} The authors have no conflicts of interest to disclose.

\end{document}